\documentclass[11pt, a4paper]{article}
\usepackage{graphicx}
\usepackage{color}
\usepackage{amsmath}
\usepackage{amssymb}
\usepackage{amscd}
\usepackage{bbm}
\usepackage{tikz}
\usepackage{tikz-cd}
\usepackage{framed}
\usepackage[normalem]{ulem}
\usetikzlibrary{positioning,arrows.meta,bending,calc}
\usepackage{hyperref}
\hypersetup{
    bookmarks=true,         
    colorlinks=true,       
    linkcolor=red,          
    citecolor=green,        
    filecolor=magenta,      
    urlcolor=cyan           
}
\usepackage[most]{tcolorbox}
\usepackage[utf8]{inputenc}
\usepackage{amsthm}
\usepackage{bbm} 
\usepackage{glossaries}
\usepackage[linesnumbered,lined,boxed,commentsnumbered]{algorithm2e}

\usepackage{marginnote}

\newtheorem{Theorem}{Theorem}
\newtheorem{Assumption}{Assumption}
\newtheorem{Definition}{Definition}
\newtheorem{Lemma}{Lemma}
\newtheorem{Proposition}{Proposition}
\newtheorem{Corollary}{Corollary}
\newtheorem{Remark}{Remark}

\newenvironment{Theorembis}[1]
     {\renewcommand{\theTheorem}{\ref{#1} (informal)}%
     \addtocounter{Theorem}{-1}%
     \begin{Theorem}}
     {\end{Theorem}} 
\newenvironment{Propositionbis}[1]
    {\renewcommand{\theProposition}{\ref{#1} (informal)}%
    \addtocounter{Proposition}{-1}%
    \begin{Proposition}}
    {\end{Proposition}}

\newenvironment{Corollaryrewrite}[1]
    {\renewcommand{\theCorollary}{\ref{#1} (rewrite)}%
    \addtocounter{Corollary}{-1}%
    \begin{Corollary}}
    {\end{Corollary}}

\newenvironment{Theoremrewrite}[1]
     {\renewcommand{\theTheorem}{\ref{#1} (rewrite)}%
     \addtocounter{Theorem}{-1}%
     \begin{Theorem}}
     {\end{Theorem}} 

\newcounter{BigConst}                     

\newcounter{SmallConst}                     

\newcounter{gamma}                     

\newcounter{kappa}                     

\newcounter{delta}                     

\newcounter{theta}                     

\newcounter{L}                     
\newcommand{\nL}{                   
    \refstepcounter{L}             
    \ensuremath{L_{\theL}}
    }
\newcommand{\oL}[1]{\ensuremath{L_{\ref*{#1}}}}  

\newcounter{eps}                     

\newcommand{\inr}[1]{\bigl< #1 \bigr>}

\newcommand{\norm}[1]{\left\|#1\right\|}%
\newcommand{\vertiii}[1]{{\left\vert\kern-0.25ex\left\vert\kern-0.25ex\left\vert #1 
    \right\vert\kern-0.25ex\right\vert\kern-0.25ex\right\vert}}

\newcommand{\vertiiii}[1]{{\left\vert\kern-0.25ex\left\vert\kern-0.25ex\left\vert\kern-0.25ex\left\vert #1 
    \right\vert\kern-0.25ex\right\vert\kern-0.25ex\right\vert\kern-0.25ex\right\vert}}

\newcommand{\Ent}{{\mathrm{Ent}}}

\newcommand{\beginproof}{{\bf Proof. {\hspace{0.2cm}}}}
\def \endproof
{{\mbox{}\nolinebreak\hfill\rule{2mm}{2mm}\par\medbreak}}

\DeclareMathOperator*{\argmin}{argmin}
\DeclareMathOperator*{\supp}{supp}

\DeclareMathOperator{\conv}{conv}

\DeclareMathOperator*{\KL}{KL}

\DeclareMathOperator*{\diam}{diam}

\DeclareMathOperator*{\Range}{Range}

\def\ds1{\textrm{1\kern-0.25emI}} 

\newcommand{\1}{\ensuremath{\mathbbm{1}}}

\newcommand \N{\mathbb{N}}

\newcommand \cD{{\cal D}}

\newcommand \cF{{\cal F}}

\newcommand \cH{{\cal H}}

\newcommand \cK{{\cal K}}
\newcommand \cL{{\cal L}}

\newcommand \cN{{\cal N}}
\newcommand \cO{{\cal O}}
\newcommand \cP{{\cal P}}
\newcommand \cQ{{\cal Q}}

\newcommand \cT{{\cal T}}
\newcommand \cU{{\cal U}}

\newcommand \cX{{\cal X}}

\newcommand \bE{{\mathbb E}}

\newcommand \bN{{\mathbb N}}

\newcommand \bP{{\mathbb P}}

\newcommand \bR{{\mathbb R}}

\newcommand \bT{{\mathbb T}}

\newcommand{\bbeta}{{\boldsymbol{\beta}}}

\newcommand{\vtheta}{{\boldsymbol{\theta}}}
\newcommand{\valpha}{{\boldsymbol{\alpha}}}
\newcommand{\vbeta}{{\boldsymbol{\beta}}}

\newcommand{\ve}{{\boldsymbol{e}}}

\newcommand{\vv}{{\boldsymbol{v}}}

\newcommand{\vzero}{{\boldsymbol{0}}}
\newcommand{\vx}{{\boldsymbol{x}}}

\newcommand{\vu}{{\boldsymbol{u}}}
\newcommand{\vw}{{\boldsymbol{w}}}

\newcommand{\vr}{\boldsymbol{r}}
\newcommand{\vn}{{\boldsymbol{n}}}

\newcommand{\vy}{{\boldsymbol{y}}}

\newcommand{\vW}{{\boldsymbol{W}}}

\DeclareMathOperator*{\Tr}{Tr}

\newcommand{\mysymbol}[3]{%
\newglossaryentry{#1}{%
      name={\ensuremath{#2}},%
      text={\ensuremath{#2}},%
      description={#3},%
      sort={#1}%
    }%
\expandafter\newcommand\expandafter{\csname smb#1\endcsname}{\gls{#1}}%
\expandafter\newcommand\expandafter{\csname #1\endcsname}{\ensuremath{#2}}%
}

\newcommand{\op}{\text{op}}

\usepackage{mathtools}
\mathtoolsset{showonlyrefs=true}

\usepackage{setspace} 
\usepackage[english]{babel}

\usepackage{enumitem}
\usepackage{xcolor}
\definecolor{darkblue}{RGB}{0, 0, 150}
\hypersetup{
    colorlinks=true,      
    allcolors=darkblue      
}

\usepackage{tikz}
\usepackage{amsfonts, amsmath}
\usetikzlibrary{cd, positioning, arrows.meta}
\usepackage{subcaption}
\usetikzlibrary{positioning, arrows.meta}

\newcommand{\fea}{\mathrm{feat}}
\newcommand{\neu}{\mathrm{neur}}

\begin{document}
\title{The Geometry of Statistical Feature Learning\\ in Mean-Field Langevin Dynamics}



\author{%
Zong Shang\thanks{ZS and TW contributed equally to this work.}\textsuperscript{\,\,\,}\thanks{CREST, ENSAE, Institut Polytechnique de Paris, 5, avenue Henry Le Chatelier 91120 Palaiseau, France. Email: \href{mailto:zong.shang@ensae.fr}{zong.shang@ensae.fr}}\and
Tomoya Wakayama\protect\footnotemark[1]\textsuperscript{\,\,\,}\thanks{RIKEN-AIP, Nihonbashi 1-4-1, Chuo-ku, Tokyo 103-0027, Japan. Email: \href{mailto:tomoya.wakayama@riken.jp}{tomoya.wakayama@riken.jp}}\and
Guillaume Lecu{\'e}\thanks{ESSEC Business School, 3 avenue Bernard Hirsch, 95021 Cergy-Pontoise, France. Email: \href{mailto:lecue@essec.edu}{lecue@essec.edu}}\and 
Taiji Suzuki\protect\footnotemark[3]\textsuperscript{\,\,\,}\thanks{Department of Mathematical Informatics, the University of Tokyo, Hongo 7-3-1, Bunkyo-ku, Tokyo 113-8656, Japan. Email: \href{mailto:taiji@mist.i.u-tokyo.ac.jp}{taiji@mist.i.u-tokyo.ac.jp}}
}

\date{\today}
\maketitle

\begin{abstract}
We introduce a geometric formulation of statistical feature learning for supervised regression. Feature learning is defined through a base--fiber decomposition: the base is the feature-side geometry produced by training, and the fiber is the learned feature space where estimation is performed. We prove this property for spherical mean-field Langevin dynamics, viewed as the Wasserstein gradient flow of a negative entropy-regularized empirical risk. In Gaussian multi-index models, the low-temperature stationary distribution concentrates near the hidden indices, forms a multi-spike structure, and yields parameter recovery with high probability, even though negative entropy regularization penalizes concentration. This concentration has a sharp transition at temperature $\lambda\asymp 1$. In Gaussian single-index models, the stationary measure satisfies a concentration property, with parity determining whether it lives on $S_2^{d-1}$ or $\mathbb{RP}^{d-1}$. The induced learned feature space aligns with the regression signal and yields rates $d/N$ and $Md/N$, up to logarithmic factors.
\end{abstract}

{
\hypersetup{linkcolor=black}
\tableofcontents
}

\section{Introduction}

This paper studies one of the central questions in the theory of neural networks: \emph{feature learning}. 
The remarkable performance of neural networks across a wide range of tasks is widely attributed to their ability to learn features adaptively during training~\cite{goodfellow_deep_2016}.
Motivated by this phenomenon, much of the existing literature focuses on whether the hidden-layer features of neural networks undergo substantial changes along the training dynamics, thereby distinguishing the feature-learning regime from the fixed-feature regime (also known as the lazy-training regime) \cite{chizat_lazy_2019,woodworth_kernel_2020,geigerDisentanglingFeatureLazy2020,yang_tensor_2021,atanasov_neural_2022,chen2022on,bachGradientDescentInfinitely2023,caronOverparameterisedShallowNeural2024}. However, a statistical understanding of feature learning, especially from the perspective of statistical prediction and estimation, requires answering two more direct questions: 
\begin{enumerate}
    \item \emph{what features are learned by training the neural network?} and
    \item \emph{why these learned features improve estimation?}
\end{enumerate}

We treat these questions through a geometric formulation: the first is formalized for general supervised-regression algorithms via a base--fiber organization of the learned feature space, while the second is analyzed for mean-field Langevin dynamics (MFLD) through its variational and PDE representations as a Wasserstein gradient flow and a nonlinear Fokker--Planck equation. We now fix the supervised regression framework in which the feature-learning property is stated. 

Given a probability space $(\cX,\bP_X)$, let $f^\star \in L^2(\bP_X)$ be an unknown function, called the regression function. Let $X$ be a random variable with probability distribution $\bP_X$, and let $\xi$ be another random variable in $\bR$, centered and independent of $X$. Define $Y = f^\star(X) + \xi$, called the label/output of $X$. Let $\bP$ be the joint distribution of $(X, Y)$. Let $\ell: (y_1, y_2) \in \bR^2 \mapsto (y_1 - y_2)^2$ be the squared loss function. A supervised regression problem is uniquely defined by the triplet $(\bP_X, f^\star, \xi)$. Let $\cF \subset L^2(\bP_X)$ be a collection of functions, called the statistical model. For convenience, we fix the sample size $N \in \mathbb{N}_+$ and let $\cD:=\{(X_i, Y_i)_{i=1}^N\}$ be a set of $N$ independent copies of $(X, Y)$, referred to as the training sample. We let $\{\hat f^{(t)}_N: (\boldsymbol{x}_i, y_i)_{i=1}^N \in (\cX\times\bR)^N \mapsto \hat f^{(t)}_N((\boldsymbol{x}_i,y_i)_{i=1}^N;\bullet) \in \mathcal{F}\}_{t \in \bT}$ be a learning algorithm, i.e., a collection of $\mathcal{F}$-valued measurable mappings, where $\hat f^{(t)}_N(\bullet) := \hat f^{(t)}_N((\boldsymbol{x}_i,y_i)_{i=1}^N;\bullet)$ is the output of the algorithm at time $t$. Here, $\bT$ equals either $\mathbb{N}$ or $\mathbb{R}_+\cup\{0\}$, corresponding to discrete-time and continuous-time algorithms, respectively. If the learning algorithm terminates at some finite time $T\in\bT$, we define $\hat f^{(t)}_N = \hat f^{(T)}_N$ for all $t > T$. We say the algorithm is convergent if for $\bP^{\otimes N}$-almost all  $\{(\vx_i,y_i)_{i=1}^N\}\subset (\cX\times\bR)^N$, $\lim_{t\to\infty}\hat f_N^{(t)}$ exists, where the limit is in the $L^2(\bP_X)$ sense. We denote this limit by $\hat f_N$, which is referred to as an estimator. In particular, at $t=0$, we use $f^{(0)}$ instead of $\hat f^{(0)}$. In the context of statistical prediction, we use the trained estimator $\hat f_N$ to predict the output $Y$ for a test input $X$. Its prediction risk (error) is defined as $P\ell_{\hat f_N} = \bE[(Y - \hat f_N(X))^2 \mid (X_i, Y_i)_{i=1}^N]$. Its $L^2(\bP_X)$-estimation error coincides with its population excess risk $P\ell_{\hat f_N} - P\ell_{f^\star}$, given by $\|\hat f_N - f^\star\|_{L^2(\bP_X)}^2$.

\subsection{Mean-field Shallow Neural Networks}

In this paper, the statistical model $\cF$ is the class of mean-field shallow neural networks.
Let $\sigma:\bR\to\bR$ be a continuous function, called an activation function. Let $A,W\in\bR_+\cup\{\infty\}$ be parameters. A \textbf{mean-field shallow neural network} (MFSNN) refers to functions of the form $f_\nu(\cdot)=\int_\Theta a\sigma(\langle\cdot,\boldsymbol{w}\rangle)\,\mathrm{d}\nu(a,\boldsymbol{w})$, where $\nu$ is a probability measure on $\Theta = [-A, A]\times WB_2^d$ where $WB_2^d = \{\vv\in\bR^d:\|\vv\|_2\leq W\}$ and $\|\cdot\|_2$ is the Euclidean norm on $\bR^d$. That is, $\cF = \{f_\nu(\cdot):\nu\in\cP(\Theta)\}$ is our statistical model, with $\cP(\Theta)$, the set of all probability measures on $\Theta$, as the parameter space. The MFSNN can be regarded as a neural network model parameterized by probability measures. In our definition, we allow $\nu$ to be a discrete probability measure, so an MFSNN does not necessarily correspond to an infinite-width neural network.

Note that since the hidden layer (the $\vw$ weights) and the output layer (the $a$ weights) are decoupled in the mapping $(a,\boldsymbol{w})\mapsto a\sigma(\langle\boldsymbol{w},\cdot\rangle)$, the MFSNN admits a natural fiber bundle-type representation. More precisely, the function class $\cF$ can be viewed through a base--fiber structure whose base is $\cP(WB_2^d)$.  For each $\varphi\in\cP(WB_2^d)$, the ambient fiber is the Hilbert space $L^2(\varphi)$.\footnote{In geometry, when one speaks of a ``base-fiber'' structure, it is usually assumed that the base forms a manifold, and so $\varphi$ is assumed to be in $\cP_{\mathrm{ac}}(WB_2^d)$ -- the set of absolutely continuous probability measures. In that case, $\cP_{\mathrm{ac}}(WB_2^d)$ admits a manifold structure in the sense of Otto calculus, \cite[Chapter 8]{villani_topics_2021}. However, for our purposes, there is no need to assume $\varphi$ to be absolutely continuous.} From this perspective, any $f_\nu\in\cF$ with $\nu\in\cP(\Theta)$ can be represented by a pair $(a,\varphi)$, where $\varphi$ is the $\boldsymbol{w}$-marginal of $\nu$, and $a(\boldsymbol{w})=\bE[\alpha\mid\boldsymbol{W}=\boldsymbol{w}]$ for $(\alpha,\boldsymbol{W})\sim\nu$: for all $\vx\in\bR^d$
\begin{equation*}
    f_\nu(\vx) = \bE\left[\alpha \sigma(\inr{\vW, \vx})\right] =  \bE\left[a(\boldsymbol{W})\sigma(\inr{\vW, \vx})\right] = \int_{W B_2^d} a(\vw)\sigma(\inr{\vw, \vx})d\varphi(\vw).
\end{equation*}
 When $A<\infty$, the neural-network model uses only the feasible part of the fiber $L^2(\varphi)$, determined by the output-weight constraint. This base--fiber viewpoint provides the global geometric organization used in the paper; the subsequent analysis refines it through multi-spike concentration on the base and low-dimensional alignment in the learned feature space.\footnote{The passage from $\nu$ to $(a,\varphi)$ is not injective. The network output depends on the conditional distribution of $\alpha$ given $\boldsymbol{W}$ only through its first moment $a(\boldsymbol{w})=\bE[\alpha\mid\boldsymbol{W}=\boldsymbol{w}]$. Thus different conditional laws of $\alpha$ with the same conditional mean give the same pair $(a,\varphi)$, and hence the same function $f_\nu$.}

\subsection{Mean-field Langevin Dynamics}

We introduce the training algorithm studied in this paper: MFLD, and then identify its long-time limit with the negative-entropy regularized empirical risk minimizer.

For a measurable function $h$ of $(X,Y)$, we write $Ph:=\mathbb E[h(X,Y)]$ and $P_Nh:=N^{-1}\sum_{i=1}^N h(X_i,Y_i)$. For $f\in\mathcal F$, define $\ell_f(x,y):=(y-f(x))^2$. When $f=f_\nu$, we write $\ell_\nu:=\ell_{f_\nu}$. Thus $P_N\ell_\nu=N^{-1}\sum_{i=1}^N(Y_i-f_\nu(X_i))^2$. For $\lambda\geq0$, define
\begin{align}\label{eq:def_empirical_free_energy}
P_N\ell_\nu^\lambda:=P_N\ell_\nu+\lambda\Ent^-(\nu).
\end{align}
Here, the {\bf{negative Shannon entropy}} is $\Ent^- : \nu \in \cP(\Theta) \mapsto \int_\Theta \frac{\mathrm{d}\nu}{\mathrm{d}\mathrm{Leb}}(\boldsymbol{\theta}) \log\left(\frac{\mathrm{d}\nu}{\mathrm{d}\mathrm{Leb}}(\boldsymbol{\theta})\right) \,\mathrm{d}\mathrm{Leb}(\boldsymbol{\theta})$, and $\frac{\mathrm{d}\nu}{\mathrm{d}\mathrm{Leb}}$ is a probability density function (pdf) of $\nu$ with respect to the Lebesgue measure $\mathrm{Leb}$. We set $\Ent^-(\nu)=+\infty$ if $\nu$ is not absolutely continuous with respect to the dominating Lebesgue measure. In the language of statistical physics, $\lambda$ is commonly referred to as the temperature of the system.

\paragraph{Mean-field Langevin dynamics.}
We now introduce the training algorithm considered in this paper. The MFLD is the Wasserstein gradient flow of $\nu\mapsto P_N\ell_\nu^\lambda$ on $\cP(\Theta)$ (more precisely, $\cP_{\mathrm{ac}}(\Theta)$ when $\lambda>0$). Let $\nu_0$ be the uniform distribution over $\Theta$. Let $\mathrm{int}(\Theta)$ be the interior set of $\Theta$. Writing $\mathrm{d}\nu_t=\rho_t \,\mathrm{d}\mathrm{Leb}$ and $V_t(\cdot):=(\delta P_N\ell_{\nu_t}/\delta\nu)(\cdot)$, where $(\delta P_N\ell_{\nu_t}/\delta\nu):(a,\vw)\mapsto -\frac{2a}{N}\sum_{i=1}^N(Y_i-f_{\nu_t}(X_i))\sigma(\langle\vw,X_i\rangle)$ is the first-order functional derivative of $P_N\ell_\bullet$ evaluated at $\nu_t$ \cite[Section 5.4.1]{carmona_probabilistic_2018}, this gradient flow formally satisfies the nonlinear Fokker--Planck equation:
\begin{align}\label{eq:nonlinear_Fokker_Planck}
\forall t\geq0,\quad 
\left\{\begin{array}{cc}
 &\partial_t\rho_t  = \nabla_{\boldsymbol{\theta}}\cdot\big(\rho_t\nabla_{\boldsymbol{\theta}} V_t+\lambda\nabla_{\boldsymbol{\theta}}\rho_t\big)\quad\mbox{ on } \mathrm{int}(\Theta), \mbox{ and }  \\
  &\big(\rho_t\nabla_{\boldsymbol{\theta}} V_t+\lambda\nabla_{\boldsymbol{\theta}}\rho_t\big)\cdot\boldsymbol{n}=0 \quad\text{on }\partial\Theta, 
\end{array}
\right.
\end{align}
for any $\vn\in N_\Theta(\vtheta)$, the normal cone of $\Theta$ evaluated at $\vtheta\in\partial\Theta$, \cite[Definition 5.2.3]{hiriart-urrutyConvexAnalysisMinimization1993}; for the compact product domain considered in this paper, the boundary condition is understood in the usual no-flux, or reflecting, sense. We say that $(\nu_t)_{t\geq 0}$ exhibits a mean-field effect, since in the corresponding nonlinear Fokker--Planck equation, the interaction potential function $V_t(\vtheta)=-\frac{2a}{N}\sum_{i=1}^N\big(Y_i-f_{\nu_t}(X_i)\big)\sigma(\langle \vw,X_i\rangle)$ implies that each particle $\vtheta=(a,\vw)\in\Theta$ interacts with the entire system, i.e., all other particles, through the integral term $f_{\nu_t}(X_i) = \int_\Theta a\sigma(\langle\vw,X_i\rangle)\mathrm{d}\nu_t(a,\vw)$. The above PDE corresponds to the reflected nonlinear McKean--Vlasov equation
\begin{align*}
\mathrm{d}\boldsymbol{\theta}_t = -\nabla_{\boldsymbol{\theta}} V_t(\boldsymbol{\theta}_t)\,\mathrm{d}t + \sqrt{2\lambda}\,\mathrm{d}\boldsymbol{B}_t + \,\mathrm{d}\boldsymbol{K}_t, \qquad \nu_t=\mathrm{Law}(\boldsymbol{\theta}_t),
\end{align*}
where $\boldsymbol{K}_t$ is the reflection process that keeps $\boldsymbol{\theta}_t$ inside $\Theta$, \cite{tanakaStochasticDifferentialEquations1979,lionsStochasticDifferentialEquations1984}. A projected Euler discretization gives the noisy gradient descent algorithm with step size $\eta$:
\begin{align*}
\boldsymbol{\theta}_{k+1} = \mathrm{Proj}_\Theta\Big( \boldsymbol{\theta}_k-\eta\nabla_{\boldsymbol{\theta}} V_k(\boldsymbol{\theta}_k)+\sqrt{2\lambda\eta}\,\boldsymbol{G}_k \Big), \qquad \boldsymbol{G}_k\sim\mathcal N(\boldsymbol{0},I_{d+1}),
\end{align*}
where $\mathrm{Proj}_\Theta$ denotes the Euclidean projection onto $\Theta$.

This is a full-batch mean-field training algorithm: the drift $\nabla_{\boldsymbol{\theta}} V_t$ is computed from the full empirical risk $P_N\ell_\nu$, and hence uses all $N$ samples at each time. On unbounded parameter spaces, quantitative convergence results for MFLD have already been established, for instance in \cite{hu_mean-field_2020,nitanda_convex_2022,chizat_mean-field_2022}. Since the parameter space is compact here, we prove, for completeness, the convergence of $(\nu_t)_{t\geq 0}$ to the variational limit defined below, together with its convergence rate (see Proposition~\ref{prop:convergence_speed} in the Appendix).

\paragraph{Long-time limit and negative-entropy regularized empirical risk minimization (RERM).}
The long-time limit of MFLD is characterized by the negative-entropy RERM problem. More precisely, the convergence result proved in Proposition~\ref{prop:convergence_speed} in the appendix shows that
\begin{align}\label{eq:def_hat_nu}
\nu_t\overset{t\to\infty}{\longrightarrow}\hat{\nu}_\lambda \in \arg\min\big(P_N\ell_\nu^\lambda : \nu \in \cP(\Theta)\big).
\end{align}
The estimator studied in this paper is $f_{\hat\nu_\lambda}$. When the minimizer is not unique \footnote{The Negative Shannon entropy is strictly convex. If $\lambda>0$ and \eqref{eq:def_hat_nu} achieves a minimum then it is unique. This may however not be the case when $\lambda=0$ that is for ``mean-field neural network''}, $\hat\nu_\lambda$ is understood as the limit selected by the above MFLD initialized from the prescribed initial distribution; equivalently, all random objects derived from $\hat\nu_\lambda$ (e.g., $\hat\varphi_\lambda$ and $\hat g_N$ defined later) are understood with respect to this selected limiting measure. In particular, its pdf admits the following self-consistent Gibbs measure representation, which makes the analysis of its properties challenging:
\begin{align*}
    \hat\nu_\lambda = \frac{\exp\left(-\frac{1}{\lambda}\frac{\delta P_N\ell_{\hat\nu_\lambda}}{\delta\nu}\right)\mathrm{d}\vtheta}{\int \exp\left(-\frac{1}{\lambda}\frac{\delta P_N\ell_{\hat\nu_\lambda}}{\delta\nu}\right)\mathrm{d}\vtheta},\mbox{ where }\frac{\delta P_N\ell_{\hat\nu_\lambda}}{\delta\nu}:(a,\vw) \mapsto -\frac{2a}{N}\sum_{i=1}^N(Y_i - f_{\hat\nu_\lambda}(X_i))\sigma(\langle\vw,X_i\rangle).
\end{align*}

\begin{Remark}
The negative Shannon entropy is not invariant with respect to the dominating measure. It is standard to choose the Lebesgue measure as a dominating measure when $\Theta$ can be endowed with such a measure. Our analysis also applies when $\Theta$ is finite. In that case, the dominating measure is usually the counting measure and for $\nu = \sum_{\boldsymbol{\theta}\in\Theta} a_{\boldsymbol{\theta}} \delta_{\boldsymbol{\theta}}$ (where $a_{\boldsymbol{\theta}}\geq0$ and $\sum_{\boldsymbol{\theta}} a_{\boldsymbol{\theta}}=1$), $\Ent^-(\nu) = \sum_{\boldsymbol{\theta} \in\Theta} a_{\boldsymbol{\theta}} \log(a_{\boldsymbol{\theta}})$, with the convention $0 \log(0) = 0$.
\end{Remark}

\begin{Remark}\label{remark:scaling}
Although the negative Shannon entropy is not invariant with respect to the dominating measure, under a scaling of the dominating measure the minimizer $\hat\nu_\lambda$ and the solution of the Wasserstein gradient flow remain invariant. In fact, let $\alpha>0$ be an arbitrary positive real number. If we choose the dominating measure to be $\alpha \,\mathrm{d}\mathrm{Leb}$, then the negative entropy satisfies $\Ent^-_\alpha(\nu)=\int_\Theta \frac{\mathrm{d}\nu}{\alpha \,\mathrm{d}\mathrm{Leb}}\log\Big(\frac{\mathrm{d}\nu}{\alpha \,\mathrm{d}\mathrm{Leb}}\Big)\alpha \,\mathrm{d}\mathrm{Leb} = \Ent^-(\nu)-\log(\alpha)$. Hence, $\argmin(P_N\ell_\nu+\lambda\Ent^-_\alpha(\nu):\nu\in\cP(\Theta)) = \argmin(P_N\ell_\nu+\lambda\Ent^-(\nu):\nu\in\cP(\Theta)).$ Such a scaling still does not alter the Wasserstein gradient flow, since the first variation is likewise invariant under the addition or subtraction of constants. In what follows, we may freely apply any scaling to the negative Shannon entropy without changing $\hat\nu_\lambda$.
\end{Remark}

\begin{Remark}\label{remark:aggregation}
The estimator $\hat\nu_\lambda$ can also be viewed as a continuous analogue of entropic convex aggregation, where for each $(a,\vw)\in\Theta$, $\vx\mapsto a\sigma(\langle\vw,\vx\rangle)$ plays the role of an expert; see, for instance \cite{koltchinskii_sparse_2009}.
\end{Remark}

\subsection{Our contributions}
This paper aims to identify the geometric--statistical mechanism underlying the feature learning phenomenon, in particular by relating it to mean-field Langevin dynamics and to the long-time behavior of the associated nonlinear Fokker--Planck equation. Accordingly, our main contributions are the following three points.
\begin{enumerate}
\item \textit{A geometric definition of the feature-learning property.} We introduce a mathematical definition of the feature-learning property for learning algorithms used in supervised regression in Definition~\ref{def:feature_learning_general}. This definition relies on a base--fiber structure: training selects a feature-side base point, while estimation is performed in the fiber over this base point, namely the learned feature space induced by it. It addresses two fundamental questions in the theory of neural networks: (i) \textit{what features are learned by the algorithm?} and (ii) \textit{how these learned features improve estimation?}.

\item \textit{Multi-spike concentration of the stationary distribution in the low-temperature regime.} Our analysis reveals that the stationary probability distribution of the nonlinear Fokker--Planck equation \eqref{eq:def_hat_varphi_sphere} associated with the Gaussian multi-index problem, as a \emph{random probability measure}, develops a multi-spike structure in the low-temperature regime: its local barycenters around each hidden index concentrates near the corresponding hidden index with high probability, thereby yielding parameter recovery. We further prove that this concentration phenomenon undergoes a sharp phase transition at the temperature scale $\lambda\asymp 1$. In the base--fiber viewpoint, this concentration describes the structure learned by the hidden-layer marginal on the base. The phenomenon is opposite to the direction of the regularization: negative entropy penalizes sparse or highly concentrated hidden-layer distributions, yet the low-temperature stationary distribution still develops multi-spike concentration. We present this result in Proposition~\ref{prop:stationary_distribution}. Our analysis further shows that, for the Gaussian single-index problem, this random probability measure satisfies a concentration property in the low-temperature regime with high probability. Moreover, depending on the parity of the information index of the link function, this concentration phenomenon takes place on either $S_2^{d-1}$ or $\bR\bP^{d-1}=S_2^{d-1}/\{\pm1\}$. We present this result in Proposition~\ref{prop:concentration_Levy_Milman}.

\item \textit{Rate improvement via alignment in the learned feature space.} For Gaussian single-index and multi-index problems with well-specified link functions, the multi-spike concentration on the base induces low-dimensional alignment in the learned feature space. Using this alignment, we establish that spherical MFLD achieves minimax optimal prediction rates, up to logarithmic factors (Theorems~\ref{thm:fixed-output-general-ie} and~\ref{thm:multi-index-fixed-output}). This is precisely the advantage of feature learning: practitioners do not need to know the latent directions in advance. Instead, MFLD automatically learns the hidden-layer marginal from the data, shaping a feature space that is intrinsically suited for estimation through its own training dynamics. In doing so, it simultaneously constructs a low-dimensional representation of the signal and an estimator capable of leveraging it. Unlike traditional high-dimensional statistical frameworks that assume an a priori structure (such as sparsity), here the low-dimensional structure emerges dynamically within the data-dependent feature space—yielding a genuinely a posteriori low-dimensional representation. Our results also establish the convergence of the moments tensor associated with this random probability measure respectively in Theorems~\ref{thm:fixed-output-general-ie} and~\ref{thm:multi-index-fixed-output}. In Section~\ref{sec:comparision_with_other_rates}, we compare this feature-learning approach to the estimation error with two other commonly used methods.
\end{enumerate}
Our analysis also yields the following two byproducts.
\begin{enumerate}
\setcounter{enumi}{3}
\item In Proposition~\ref{prop:hat_g_N}, we uncover a self-regularization (implicit bias) property of MFLD. Specifically, the latent estimator of MFLD can be cast as a regularized empirical risk minimizer (RERM) within the learned feature space, driven by a random regularization functional. This result provides a theoretical justification for two-stage or two-timescale training strategies commonly employed in neural network analysis.

\item Another consequence is that, under a support recovery condition, LASSO also exhibits the feature-learning property. In the appendix, we present this result and point out its connection with the support recovery property. This shows that feature learning, in our sense, is not restricted to neural networks, and thus introduces a new perspective for analyzing traditional or new statistical methods, namely whether they possess the feature-learning property.
\end{enumerate}

\subsection{Organization of the paper}

Section~\ref{sec:feature_learning_MFLD} presents the base--fiber formulation of the feature-learning property. Section~\ref{sec:feature_learning_Gaussian_single_multi_index} verifies this property for spherical MFLD in Gaussian index models and proves both the low-temperature concentration of its stationary distribution and the minimax optimal convergence rates. Finally, Section~\ref{sec:future} discusses our results and outlines directions for future research. The proofs of the results in this paper are provided in the Appendix.

\paragraph{Notation.}
Let $\|\cdot\|_2$ be the Euclidean norm. We write $B_2^d:=\{\boldsymbol w\in\mathbb R^d:\|\boldsymbol w\|_2\le 1\}$, $WB_2^d:=\{\boldsymbol w\in\mathbb R^d:\|\boldsymbol w\|_2\le W\}$, and $S_2^{d-1}:=\{\boldsymbol w\in\mathbb R^d:\|\boldsymbol w\|_2=1\}$. Let $\bR\bP^{d-1}=S_2^{d-1}/\{\pm 1\}$ be the real projection space. For a probability measure $\mu$, $\|\cdot\|_{L^2(\mu)}$ denotes the usual $L^2(\mu)$-norm. For a bounded function $g$, $\|g\|_\infty$ denotes its supremum norm. We write $\mathcal P_{\mathrm{ac}}(\Theta)$ for the set of probability measures on $\Theta$ that are absolutely continuous with respect to the chosen dominating measure. For elements $u,v$ of a Hilbert space $\cH$, $u\otimes v$ denotes the rank-one operator $h\mapsto \langle h,v\rangle_{\cH}u$. For $\boldsymbol u\in\mathbb R^d$, $\boldsymbol u^{\otimes m}$ denotes its $m$-fold tensor product. Denote by $\operatorname{Sym}^m(\mathbb R^d)$ the space of order-$m$ symmetric tensors, equipped with the Frobenius inner product characterized by $\langle \boldsymbol u^{\otimes m},\boldsymbol v^{\otimes m}\rangle_F=\langle \boldsymbol u,\boldsymbol v\rangle^m$. The corresponding Frobenius norm is denoted by $\|\cdot\|_F$. Let $\tau$ denote the uniform probability measure on $S_2^{d-1}$. For $\varphi\in\mathcal P(S_2^{d-1})$, define the negative Shannon entropy relative to $\tau$ by $\Ent_\tau^-(\varphi):=\int_{S_2^{d-1}}\log(d\varphi/d\tau)\,d\varphi$ if $\varphi\ll\tau$, and $\Ent_\tau^-(\varphi):=+\infty$ otherwise. For $k\in\mathbb N$, $C_b^k(\mathbb R)$ denotes the space of $k$-times continuously differentiable functions whose derivatives up to order $k$ are bounded. We use standard asymptotic notation as follows. Unless otherwise specified, all asymptotic notation used throughout this paper is understood in the regime where $d$ is fixed and $N\to\infty$. We use $O$ or $\cO$ to denote the big-O notation. For random variables $Z_N$ and positive deterministic sequences $a_N$, $Z_N=o_{\mathbb P}(a_N)$ means $Z_N/a_N\to0$ in probability, and $Z_N=O_{\mathbb P}(a_N)$ means that $Z_N/a_N$ is bounded in probability. For nonnegative $Z_N$, we write $Z_N=\Omega_{\mathbb P}(a_N)$ if there exists a constant $c>0$ such that $\mathbb P(Z_N\ge c a_N)\to1$. For deterministic nonnegative quantities, $A\lesssim B$ means $A\le CB$ for a numerical constant $C$ independent of the relevant problem parameters, and $A\asymp B$ means $A\lesssim B$ and $B\lesssim A$. We set $\psi:[0,\infty)\to\mathbb R$ by $\psi(0)=0$ and $\psi(t)=t(1+\log(e/t))$ for $t>0$.

\section{The Base--Fiber Geometry of Feature Learning}\label{sec:feature_learning_MFLD}

In this section, we formalize the feature-learning property from the base--fiber geometric viewpoint, which is the central conceptual contribution of this paper, and explain how, when applied to MFLD, this property is related to the evolution of the nonlinear Fokker--Planck equation in \eqref{eq:nonlinear_Fokker_Planck}.
Since feature learning is a key ability of neural networks, a substantial body of work has studied it, mostly focusing on training dynamics, examining whether the hidden-layer parameters, representations, or tangent kernels undergo nontrivial changes during training~\cite{chizat_lazy_2019,woodworth_kernel_2020,geigerDisentanglingFeatureLazy2020,yang_tensor_2021,atanasov_neural_2022,chen2022on,bachGradientDescentInfinitely2023,caronOverparameterisedShallowNeural2024}. However, a precise mathematical definition that identifies the geometry through which learned features improve estimation has not yet been systematically formulated. We take a geometric--statistical viewpoint, rather than only an optimization viewpoint, and propose a definition of the feature-learning property for learning algorithms in supervised regression. This definition should be viewed as a first attempt to formalize the phenomenon, with the aim of characterizing whether the features and the low-dimensional representation of the target produced by training can reduce the estimation error in a quantifiable way through the geometry of the learned feature space. MFLD serves as the main dynamical object for which this property is verified in the paper.

Before giving the definition, we provide an intuitive description. To this end, it is useful to distinguish \emph{feature engineering} from feature learning. 
In this paper, we use the term feature engineering broadly to refer to the problem-informed pre-specification of a representation class or estimation procedure, such as a family of features, a dictionary, a feature map, a kernel function, or a tailored estimator.
In statistical learning, this viewpoint often appears in the form of basis expansions, dictionaries or (static) feature maps: the original covariates are replaced, or augmented, by prescribed transformations, and estimation is then performed in the induced feature space; see, for instance \cite[Chapter~5]{hastie_elements_2009}. 
Typical examples include choosing Fourier or spline bases, specifying kernels and hyper-parameters for kernel regression, or selecting other function systems motivated by prior structural knowledge of the problem \cite{wahbaSplineModelsObservational1990,tsybakov_introduction_2009,steinwart_support_2008}. Another instructive example is the single-index model. If one knows in advance that the target function has the form $f^\star(\boldsymbol{x})=h(\langle \boldsymbol{w}^\star,\boldsymbol{x}\rangle)$ for some $\boldsymbol{w}^\star\in\bR^d$ and $h:\bR\to\bR$, then one may design an estimator or an algorithm around this low-dimensional structure: estimate, or otherwise exploit, the index direction $\boldsymbol{w}^\star$, and then perform nonparametric estimation on the one-dimensional function $h$. In this sense, single-index methods may be viewed as a form of feature engineering based on prior low-dimensional structure; their advantage is that they reduce a high-dimensional nonparametric estimation problem to a lower-dimensional one, thereby improving the convergence rate \cite{hardleOptimalSmoothingSingleIndex1993,gaiffasOptimalRatesAdaptation2007,brunaSurveyAlgorithmsMultiindex2025}. 

However, feature engineering has an intrinsic limitation. It may prescribe a feature space, or, in the base--fiber language, a fiber, but this does not ensure that the regression function is positioned in that fiber in a way that the estimator can effectively exploit. Thus, the effectiveness of pre-specified features depends not only on their own structure, but also on the position of the regression function relative to the exploitable directions. Even when the chosen features are natural and well structured, the resulting estimator can be effective only if the regression function is sufficiently aligned with those directions~\cite{shangFeatureSpaceDecomposition2026}. This motivates the following notion, which formalizes this in-fiber alignment.

\subsection{The Alignment Property in a Fixed Fiber}\label{sec:alignment_property}

In this subsection, the base point is fixed: we work in a prescribed Hilbert (feature) space, which should be viewed as a fixed fiber. This is the fiber-level part of the geometry developed in the paper. The relevant geometry is not the evolution of a measure on the base, but the relative position, inside this fiber, of the representative (oracle) of the regression function with respect to the principal directions of the covariance operator, alongside the head estimator's capacity to leverage this alignment for enhanced prediction. The notion of the alignment property was introduced in \cite{lecueSharpConvergenceRates2025} to describe how an estimator implicitly uses the principal directions of a reproducing kernel Hilbert space (RKHS) to make predictions. Before defining this notion, we recall standard  notation related to RKHS, referring the reader to \cite{scholkopf_learning_2001,shawe-taylor_kernel_2004,steinwart_support_2008,saitoh_theory_2016}. 

For an RKHS $(\cH,\langle\cdot,\cdot\rangle_\cH)$ on $\cX$, we view $\cH$ as the fixed fiber, denote by $\phi:\cX\to\cH$ its canonical feature map, and denote by $\Sigma=\bE[\phi(X)\otimes_\cH\phi(X)]$ the corresponding integral operator, that is, $\Sigma f=\bE[f(X)\phi(X)]$ for $f\in\cH$. We assume that $\Sigma$ admits a spectral decomposition, and denote its eigenvalue-eigenvector pairs by $(\sigma_j,\boldsymbol{e}_j)_{j=1}^\infty$, with $\sigma_1\geq\sigma_2\geq\cdots$. We identify $\cH'$, the dual space of $\cH$, with $\cH$ itself through the Riesz representation theorem, that is, we identify the bounded linear functional $g:h\in\cH\mapsto\langle g,h\rangle_\cH\in\bR$ with its uniquely corresponding element $g\in\cH$. Therefore, in the sequel, we shall not distinguish between $\langle g,\phi(\vx)\rangle_\cH$ and $g(\phi(\vx)) = (g\circ\phi)(\vx)$.


\begin{Definition}[Alignment property]\label{def:alignment_property}
Let $(\mathbb P_X,f^\star,\xi)$ be a supervised regression problem. Let $\mathcal H$ be a separable RKHS with canonical feature map $\phi:\mathcal X\to\mathcal H$, and suppose that the covariance operator $\Sigma=\mathbb E[\phi(X)\otimes_{\mathcal H}\phi(X)]$ admits an eigendecomposition $(\sigma_j,\ve_j)_{j\ge1}$. Let $\hat g_N,g_{\mathcal H}\in\mathcal H$ and $(k_N)_N$ be a non-decreasing sequence of integers. Given non-negative weights $\{\gamma_j\}_{j>k}$, $0<\delta<1$, a tolerance $\varepsilon_N(k,\delta)\ge 0$ such that $\varepsilon_N(k_N,\delta) = o_\bP(1)$ and $\varepsilon_N(k_N,\delta)$ is independent of $(g_\cH, \ve_j)_{j>k_N}$, and a sequence of non-decreasing deterministic functions $\omega_N:\mathbb R_+\to\mathbb R_+$ satisfying $\omega_N(0)=0$ and $\lim_{\eta\downarrow0}\limsup_{N\to\infty}\omega_N(\eta)=0,$ we say that $\hat g_N$ satisfies the $(g_{\mathcal H},k_N,\delta;\varepsilon_N,\omega_N)$-alignment property with respect to $\{\gamma_j\}_{j>k}$ if, with probability at least $1-\delta$, $\left\| \hat g_N\circ\phi-g_{\mathcal H}\circ\phi \right\|_{L^2(\mathbb P_X)}^2 \le \varepsilon_N(k_N,\delta) + \omega_N\big( \sum_{j>k_N}\gamma_j \langle g_{\mathcal H},\ve_j\rangle_{\mathcal H}^2 \big). $
\end{Definition}

An estimator $\hat g_N$ satisfying the alignment property has the following characteristic: when most of the energy of the target function $g_\cH$ is carried by the first $k_N$ eigenfunctions of $\cH$, or equivalently when the weighted tail energy $(\boldsymbol{e}_j)_{j>k}$ is small, the estimator can exploit this structure and achieves a faster estimation error $\left\|\hat g_N\circ\phi - g_\cH\circ\phi\right\|_{L^2(\mathbb P_X)}$; in other words, $\hat g_N$'s convergence rate gets faster as $g_\cH$ aligns more with the principal directions in $\cH$. Several estimators are known to satisfy this property through the Feature Space Decomposition method, including ridge regression, gradient flow, gradient descent, and principal components regression; see \cite[Chapter~1]{shangFeatureSpaceDecomposition2026}. A natural choice for $g_\cH$ is the oracle: i.e. the closest element in $\cH$ to $f^\star$ w.r.t. the $L^2(\bP_X)$-metric. It may however, not the only possible choice (see Corollary~\ref{coro:feature_learning_multi}).

However, an estimator $\hat g_N$ satisfying the alignment property can only leverage on a \textit{pre-existing} favorable alignment. In particular, when the alignment of the signal with the top $k_N$ eigenvectors in $\cH$ is poor, it is unable to modify this alignment structure and therefore may end up with bad statistical properties. This is precisely one of the limitations of feature engineering: manually selected features do not necessarily provide a favorable alignment; see \cite[Definition 6]{lecueSharpConvergenceRates2025}. The capacity to align the signal with a feature space tailored to the target prediction task is precisely a defining hallmark of feature learning, which we formally introduce below.

\subsection{A geometric and statistically oriented formulation of Feature Learning}

We now state the feature-learning property for a general learning algorithm in supervised regression. The definition is kept independent of the particular structure of MFLD; its MFLD realization will be identified in the next subsection.

\begin{Definition}[Feature-learning property]
\label{def:feature_learning_general}
Consider a supervised regression problem $(\bP_X,f^\star,\xi)$ and a convergent algorithm with final estimator $\hat f_N$. Let $K_0:\cX\times\cX\to\bR$ be the initial kernel and let $\{ (\vx_i,y_i)_{i=1}^N\in(\cX\times\bR)^N\mapsto K_\fea^{(N)}:\cX\times\cX\to\bR \}_{N}$ be an algorithm-induced kernel-generating rule. Let $\cH_\fea^{(N)}$ be the RKHS generated by $K_\fea^{(N)}$, and $\Sigma_\fea^{(N)} = \bE[K_\fea^{(N)}(X,\cdot)\otimes_{\cH_\fea^{(N)}} K_\fea^{(N)}(X,\cdot)|(X_i,Y_i)_{i=1}^N]$. We denote by $(\sigma_j^{(N)},\ve_j^{(N)})_{j\geq 1}$ the eigenvalue-eigenfunction pairs of $\Sigma_\fea^{(N)}$ with $\sigma_1^{(N)}\geq \sigma_2^{(N)}\geq\cdots$. We say that the algorithm or $\hat f_N$ has the \textbf{feature learning property} via the feature spaces $(\cH_\fea^{(N)})_{N}$ when solving the supervised learning problem $(\bP_X,f^\star,\xi)$, if there exist an algorithm-induced latent representative $\hat g_N\in\cH_\fea^{(N)}$ of $\hat f_N$, a target representative $g_\fea^{(N)}\in\cH_\fea^{(N)}$ of $f^\star$, integers $(k_N)_N$ with $k_N=o(N)$, prescribed non-negative real numbers $\{\gamma_j^{(N)}\}_{j>k_N}$ such that the following conditions hold when $N\to\infty$
\begin{enumerate}
    \item \label{item:feature_evolution} $\| K_\fea^{(N)} - K_0 \|_{L^2(\bP_X\otimes\bP_X)}=\Omega_\bP(1)$; 
    \item \label{item:latent} $\hat f_N(\vx) = \hat g_N(\phi_\fea^{(N)}(\vx))$ for any $\vx\in\cX$, where $\phi_\fea^{(N)}(\vx) = K_\fea^{(N)}(\vx,\cdot)$;
    \item \label{item:approximation_error} $\| f^\star - g_\fea^{(N)}(\phi_\fea^{(N)}(\cdot))\|_{L^2(\bP_X)}=o_\bP(1)$;
    \item \label{item:top_k} $\sum_{j>k_N}\gamma_j^{(N)}\langle g_\fea^{(N)},\ve_j^{(N)}\rangle_{\cH_\fea^{(N)}}^2 = o_\bP(1)$;
    \item \label{item:alignment} $\hat g_N$ satisfies the $(g_\fea^{(N)},k_N,\delta_N;\varepsilon_N,\omega_N)$-alignment property with respect to $\{\gamma_j^{(N)}\}_{j>k_N}$ with 
    $\delta_N\downarrow 0$, $\varepsilon_N(k_N,\delta_N)=o(1)$, and an admissible sequence $\{\omega_N\}_{N\ge1}$.
\end{enumerate}
Here, $\cH_\fea^{(N)}$ is called the \textbf{learned (input-space) feature space} and the top $k_N$ eigenvectors of $\Sigma_\fea^{(N)}$ are the \textbf{learned features}. $\{\gamma_j^{(N)}\}_{j>k_N}$ are fixed as part of the estimator-specific alignment structure and are not chosen post hoc from the tail coefficients of $g_\fea^{(N)}$.
\end{Definition}
We explain the meaning of Definition~\ref{def:feature_learning_general}. For ease of exposition, we drop the superscript $(N)$ in the following. Item~\emph{\ref{item:feature_evolution}} means that a nontrivial feature evolution has occurred during and after training: the constructed (RKHS) feature space  obtained after training is no longer the initial one (represented by their kernels). This is the dynamical condition commonly used in the existing literature, where feature learning is identified through whether the features undergo evolution, e.g., \cite{yang_tensor_2021}. From our perspective, the alignment between $f^\star$ and $\cH_\fea^{(N)}$ should improve. In particular, estimators such as spectral methods are excluded from this definition even in situation where their a priori given feature space has a good representation of $f^\star$ from the beginning because there is no dynamic/evolution improving the alignment: no features are learned, they are only used.  

Item~\emph{\ref{item:latent}} means that the estimator $\hat f_N$ can be factorized into a learned feature map $\phi_\fea$ and a latent estimator $\hat g_N$ (see Figure~\ref{fig:estimator_factorization}). This is the point where the base--fiber organization enters the definition: $\phi_\fea$ represents the feature-side object produced by training, while $\mathcal H_\fea$ is the learned feature space in which the latent estimator $\hat g_N$ acts. Thus this factorization separates the process of learning the features from the process of using the learned features for estimation. 

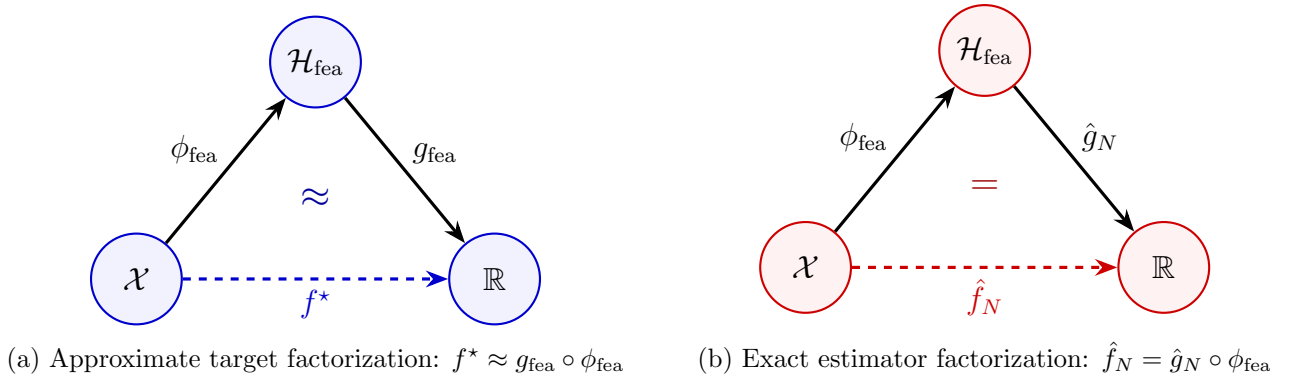
\begin{figure}[htbp]
\centering
\begin{subfigure}[b]{0.48\textwidth}
\centering
\begin{tikzpicture}[
    node distance=2.2cm and 1.8cm,
    auto,
    >=Stealth,
    every node/.style={font=\large},
    box/.style={circle, draw=blue!80!black, fill=blue!5, thick, minimum size=1.2cm}
]

    \node[box] (H) {$\mathcal{H}_{\text{fea}}$};
    \node[box, below left=2cm and 1.5cm of H] (X) {$\mathcal{X}$};
    \node[box, below right=2cm and 1.5cm of H] (R) {$\mathbb{R}$};

    \draw[->, very thick, line width=1.2pt] (X) -- node[above left=-2pt] {$\phi_{\text{fea}}$} (H);
    \draw[->, very thick, line width=1.2pt] (H) -- node[above right=-2pt] {$g_{\text{fea}}$} (R);
    
    \draw[->, very thick, dashed, line width=1.2pt, blue!80!black] (X) -- node[below] {$f^\star$} (R);

    \node[font=\Large, blue!60!black] at (0, -1.8) {$\approx$};

\end{tikzpicture}
\caption{Approximate target factorization: $f^\star \approx g_{\text{fea}} \circ \phi_{\text{fea}}$}
\label{fig:oracle_factorization}
\end{subfigure}
\hfill
\begin{subfigure}[b]{0.48\textwidth}
\centering
\begin{tikzpicture}[
    node distance=2.2cm and 1.8cm,
    auto,
    >=Stealth,
    every node/.style={font=\large},
    box/.style={circle, draw=red!80!black, fill=red!5, thick, minimum size=1.2cm}
]

    \node[box] (H) {$\mathcal{H}_{\text{fea}}$};
    \node[box, below left=2cm and 1.5cm of H] (X) {$\mathcal{X}$};
    \node[box, below right=2cm and 1.5cm of H] (R) {$\mathbb{R}$};

    \draw[->, very thick, line width=1.2pt] (X) -- node[above left=-2pt] {$\phi_{\text{fea}}$} (H);
    \draw[->, very thick, line width=1.2pt] (H) -- node[above right=-2pt] {$\hat{g}_N$} (R);
    
    \draw[->, very thick, dashed, line width=1.2pt, red!80!black] (X) -- node[below] {$\hat{f}_N$} (R);

    \node[font=\Large, red!60!black] at (0, -1.8) {$=$};

\end{tikzpicture}
\caption{Exact estimator factorization: $\hat{f}_N = \hat{g}_N \circ \phi_{\text{fea}}$}
\label{fig:estimator_factorization}
\end{subfigure}

\caption{$f^\star$ and $\hat f_N$ factorizations via the learned feature space $\mathcal{H}_{\text{fea}}$. The canonical feature map $\phi_{\text{fea}}$ produces $k$ top eigenvectors (the learned features) that are used to represent well $f^\star$ (i.e. $f^\star \approx g_{\text{fea}} \circ \phi_{\text{fea}}$ and \textit{item~4}): it is the low-dimensional structure emerging from feature learning. Then $\hat{g}_N$ (the head) leverages this alignment to enhance its predictive capability.}
\label{fig:feature_learning_pyramid}
\end{figure}

Item~\emph{\ref{item:approximation_error}} means that the learned feature space $\cH_\fea$ can approximate the regression function $f^\star$ well (see Figure~\ref{fig:estimator_factorization}). As discussed in Section~\ref{sec:alignment_property}, good approximation capability alone does not necessarily imply a small estimation error; one additionally requires the representative $g_\fea$ of $f^\star$ in $\cH_\fea$ to be favorably positioned relative to the directions that the head estimator $\hat g_N$ can exploit in order to reduce the estimation error. Items~\emph{\ref{item:top_k}}~and~\emph{\ref{item:alignment}} are designed to capture this requirement. Specifically, item~\emph{\ref{item:top_k}} requires the representative $g_{\fea}$ of $f^\star$ in the learned feature space $\cH_\fea$ to be essentially supported on the top $k$ (learned) directions of the covariance operator $\Sigma_{\fea}$. This is where feature learning differs from the classical high-dimensional statistics paradigm where low-dimensional structures are assumed a priori; here, with feature learning, this low-dimensional structure appears in the constructed/data-dependent feature space $\cH_\fea$. Item~\emph{\ref{item:alignment}} then requires the latent estimator $\hat g_N:\cH_\fea\to \bR$ to exploit this positional relationship through the alignment property. In other words, if the representative $g_{\fea}$ has small weighted tail energy (in the eigenbasis of $K_\fea$), then the prediction error rate  $\|\hat g_N\circ \phi_\fea-g_\fea\circ\phi_\fea\|^2_{L^2(\mathbb P_X)}$, will be faster since the latter is faster as the tail energy of $g_\fea$ gets smaller (that is the meaning of the alignment property of $\hat g_N$). Since $\hat f_N(\cdot)=\hat g_N(\phi_\fea(\cdot))$, this converts the learned signal-feature alignment into an improved convergence rate of the prediction error of the final estimator. The condition $k=o(N)$ ensures that the effective number of relevant learned directions remains negligible compared with the sample size in order to enhance consistency of $\hat f_N$, since we expect a $\cO(k/N)$ convergence rate.

\emph{Definition~\ref{def:feature_learning_general} characterizes feature learning as the simultaneous process of (i) constructing a data-dependent, task-relevant representation of \(f^\star\), whose effective component is predominantly \(k\)-dimensional, and (ii) effectively utilizing this representation via a simple estimator \(\hat{g}_N\).} Structurally, this definition encapsulates feature evolution (Item~1), feature learning proper (Items~3 and 4), and the exploitation of these features by the head (Items~2 and 5). Intuitively, the feature map $\phi_{\text{fea}}$ constructs a $k$-dimensional subspace within $\mathcal{H}_{\text{fea}}$ that captures most of the spectral energy of $f^\star$. By establishing this alignment between the target signal and the learned feature space, $\phi_{\text{fea}}$ allows the $\hat{g}_N$ to efficiently estimate the corresponding $k$ coefficients via a ridge-type regression. We naturally view this factorization $\hat{f}_N(\vx) = \hat{g}_N(\phi_{\text{fea}}^{(N)}(\vx))$ through the lens of a \textbf{polar decomposition} of $\hat f_N$: $\phi_{\text{fea}}^{(N)}$ aligns the feature geometry, while \textbf{the head} $\hat{g}_N$ acts as a linear estimator that aggregates these aligned components for prediction. Most importantly, feature learning constructs this low-dimensional structure within the signal while providing an estimator capable of directly exploiting it. Ultimately, by leveraging this created low-dimensional structure, the estimator $\hat f_N$ achieves an oracle-like convergence rate of $\mathcal{O}(k/N)$, mirroring the regime where $f^\star$ is a $k$-sparse signal in $\mathcal{H}_{\text{fea}}$.

Finally, let us comment on the relationship between feature learning and adaptivity in Statistics. In \cite[Section 2]{lepskiTheoryAdaptiveEstimation2023}, an estimator is said adaptive if it can simultaneously achieve minimax optimality over a set of classes of problems; like the set of all Sobolev function classes. In Theorems~\ref{thm:fixed-output-general-ie}-\ref{thm:rip-multi-index-fixed-output} below, we indeed prove that, when the link function is well-specified, MFLD without bias terms\footnote{Here, we refer to the bias parameter $b$ in the neuron activation $\sigma(\langle w, x \rangle + b)$, which is omitted in this work. This structural bias should not be confused with the concepts of explicit or implicit bias, which describe the propensity of an estimator to favor specific structures via explicit or implicit regularization techniques.} 
is adaptive to Gaussian single-/multi-index problems, i.e. adaptive to the direction(s) of the low dimensional structure. However, in Section~\ref{sec:counter_example}, we provide counterexamples showing that \emph{such} MFLD is not adaptive to single-index problems with misspecified link functions. This is the first way in which our work differs from the existing literature on adaptivity. Rather than designing a new estimator that is adaptive over a prescribed class of single- or multi-index problems, as in \cite{bietti_learning_2022,bietti_learning_2025,ba_high-dimensional_2022, ba_learning_2023,dandi_how_2024,damian_neural_2022,damian_smoothing_2023, leeNeuralNetworkLearns2024}, our goal is to study the adaptivity of a natural training algorithm, namely MFLD. More precisely, we aim to identify the class of problems for which MFLD is adaptive, and to characterize the boundary beyond which this adaptivity breaks down.

In addition, the feature-learning property defined in this paper is different from minimax optimality in the sense of \cite{lepskiTheoryAdaptiveEstimation2023, tsybakov_introduction_2009}: what we aim for is problem-specific adaptivity, rather than adaptivity in the minimax sense over a class of problems. These are two different types of questions. As shown in Definition~\ref{def:feature_learning_general}, we focus on the property of the estimator $\hat f_N$ for a specific supervised regression problem $(\bP_X,f^*,\xi)$, rather than a minimax property that holds uniformly over a class of problems. Thus, this is a more fine-grained property as opposed to the worst case approach in minimax theory.

It is worth noting that our problem-specific adaptivity does not contradict the no-free-lunch theorem. The no-free-lunch theorem states that no learning rule can perform well uniformly over all possible statistical problems~\cite[Section~7.1]{devroye_probabilistic_1996}. Our point is different: we do not claim that MFLD, or neural networks more generally, can solve every problem without structure. We show that for certain structured problems, the relevant structure need not be given to the statistician in advance; it can be learned from the training data and then used for estimation. In this sense, the problems solved by feature learning should not be viewed as ``free lunches''. They are better viewed as ``inexpensive lunches'': the useful structure is already present in the data-generating mechanism and the algorithm needs to discover and exploit it at some statistical cost. The type of structure that can be learned depends on both the architecture of the neural network and the algorithm used to train it (in particular on its explicit or implicit biases). For instance, in this work, we show that shallow neural networks trained by MFLD can learn the single-/multi-index structure between $Y$ and $X$.  The no-free-lunch theorem rules out success that is freely available for all problems, not the possibility that many important problems carry learnable structure. From this viewpoint, the empirical success of neural networks aligns with a fundamental selection effect inherent to real-world tasks in science and engineering: these problems are rarely arbitrary high-dimensional tasks, but rather possess intrinsic low-dimensional or geometric structures that training dynamics can automatically uncover. Consequently, a central objective is to characterize the spectrum of structures that a specific pairing of architecture and optimization algorithm can discover autonomously, utilizing unifying concepts such as feature learning and implicit bias.

In what follows, we use MFLD as the main example to instantiate the above definition. We also give the LASSO as another example in the appendix.

\subsection{The Base--Fiber structure of MFLD}

We now specialize the preceding definition through the long-time limit $\hat\nu_\lambda$ of MFLD. The construction below identifies the base--fiber structure generated by $\hat\nu_\lambda$ in this dynamical setting.

Recall $\hat\nu_\lambda$ in \eqref{eq:def_hat_nu}. Let $(\hat A,\hat W)\sim\hat\nu_\lambda$, where $\hat A\in[-A,A]$ and $\hat W\in WB_2^d$ a.s., and denote by $\hat\varphi_\lambda$ the marginal distribution of $\hat W$.
Define the neuron map $\varphi_{\neu}:\cX\to L^2(\hat\varphi_\lambda)$ by $\varphi_{\neu}(\boldsymbol{x})(\boldsymbol{w})=\sigma(\langle \boldsymbol{w},\boldsymbol{x}\rangle)$. We assume that $\mathbb E_X\big[\|\varphi_{\neu}(X)\|_{L^2(\hat\varphi_\lambda)}^2|\cD\big]<\infty$, $\bP^{\otimes N}$-almost surely.
Let
\begin{align}\label{eq:def_K_fea}
    K_{\fea}:(\boldsymbol{x}_1,\boldsymbol{x}_2)\in\cX\times\cX
    \mapsto
    \mathbb E_{\hat W\sim\hat\varphi_\lambda}
    \big[
        \sigma(\langle\hat W,\boldsymbol{x}_1\rangle)
        \sigma(\langle\hat W,\boldsymbol{x}_2\rangle)
    \big]
    =
    \langle
        \varphi_{\neu}(\boldsymbol{x}_1),
        \varphi_{\neu}(\boldsymbol{x}_2)
    \rangle_{L^2(\hat\varphi_\lambda)} .
\end{align}
Let $\cH_\fea$ be the RKHS on the input space $\cX$ generated by $K_{\fea}$, and denote its canonical feature map by $\phi_{\fea}(\boldsymbol{x}):= K_{\fea}(\boldsymbol{x},\cdot) \in \cH_\fea$. We call $L^2(\hat\varphi_\lambda)$ the \textbf{learned fiber}, or learned coefficient space; and $\cH_\fea$ the \textbf{learned feature space}, or learned input-space RKHS, induced by $\hat\varphi_\lambda$. The connection between these two spaces is defined through the following equation (see, for instance, \cite[Theorem 4.21]{steinwart_support_2008}):
\begin{align}\label{eq:quotient_H_fea}
    \forall g\in\cH_\fea,\quad \norm{g}_{\cH_\fea} = \inf\left(\norm{a}_{L^2(\hat\varphi_\lambda)}:\, g(\cdot) = \langle a,\varphi_{\neu}(\cdot)\rangle_{L^2(\hat\varphi_\lambda)}\right).
\end{align}By the Riesz representation theorem together with \eqref{eq:quotient_H_fea}, for any $g\in\cH_{\fea}'$, the dual space of $\cH_\fea$, identified with its Riesz representative, there exists a unique $a\in L^2(\hat\varphi_\lambda)$ that achieves the infimum of \eqref{eq:quotient_H_fea}, such that for all $\vx\in\cX$
\begin{align}\label{eq:correspondence_g_a}
    g(\phi_\fea(\vx)) = \langle\phi_\fea(\vx),g\rangle_{\cH_\fea} = \langle a,\varphi_{\neu}(\vx)\rangle_{L^2(\hat\varphi_\lambda)} = \bE_{\hat W\sim\hat\varphi_\lambda}[a(\hat W)\sigma(\langle\vx,\hat W\rangle)].
\end{align}

We now explain the rationale behind the generating rule $(\vx_i,y_i)_{i=1}^N \in(\cX\times\bR)^N \mapsto K_\fea$ for the reproducing kernel. 

Consider the nonlinear Fokker--Planck equation initialized from the uniform distribution $\nu_0$ on $\Theta$, namely the family $(\nu_t)_{t\geq0}$ described by \eqref{eq:nonlinear_Fokker_Planck}. Let $(A^{(t)},W^{(t)})\sim\nu_t$, and denote by $\varphi_t$ the distribution of $W^{(t)}$. The flow $(\nu_t)_{t\geq0}$ evolves in the full parameter (probability measures) space $\cP(\Theta)$, while $(\varphi_t)_{t\geq0}$ is its projection onto the hidden-layer base parameter space $\cP(W B_2^d)$. For each $t\geq0$, let $K^{(t)}$ and $\cH^{(t)}$ be the kernel and input-space RKHS generated by $\varphi_t$ in the same way as above, namely $K^{(t)}(\vx_1,\vx_2)=\langle\varphi_{\neu}(\vx_1),\varphi_{\neu}(\vx_2)\rangle_{L^2(\varphi_t)}$. At $t=0$, the hidden-layer marginal is the initialization, and $\cH^{(0)}$ corresponds to the random features kernel~\cite{ghorbani_linearized_2021,bartlett_deep_2021}. As $t$ increases, the nonlinear Fokker--Planck flow moves in the parameter-measure space and, through the marginal map, moves on the hidden-layer base; correspondingly, $\cH^{(t)}$ is the feature space formed at time $t$ from the initial random-features space $\cH^{(0)}$. This evolution from $\cH^{(0)}$ to $\cH^{(t)}$ is the feature-learning process of MFLD. Since the empirical free energy $P_N\ell_\nu^\lambda$ is in general not displacement convex on $\cP(\Theta)$, the finite-time dynamics may spend time near spurious stationary points, in the Wasserstein geometry. Nevertheless, the thermal fluctuations induced by the Langevin diffusion select the long-time limit $\hat\nu_\lambda$. Hence, as $t\to\infty$, the hidden-layer marginal $\varphi_t$ converges to $\hat\varphi_\lambda$, and $\cH^{(t)}$ converges to the final learned feature space $\cH_{\fea}$. In this sense, the feature learning process of MFLD is geometrically described by the path $\cH^{(0)}\to\cH^{(t)}\to\cH_{\fea}$ induced by the nonlinear Fokker--Planck flow through the map $\nu_t\mapsto\varphi_t\mapsto\cH^{(t)}$. 

This perspective separates the base feature learning dynamics from the fiber-wise estimation in MFLD. The nonlinear Fokker--Planck flow evolves in $\cP_{\mathrm{ac}}(\Theta)$, but the feature-learning component is described by the motion of its hidden-layer marginal path $(\varphi_t)_{t\geq0}$ on the base $\cP_{\mathrm{ac}}(WB_2^d)$. Once the limiting base point $\hat\varphi_\lambda$ is reached, the use of these features for estimation takes place in the fiber $L^2(\hat\varphi_\lambda)$ over this base point. Thus $(\varphi_t)_{t\geq0}$ is responsible for evolving the initial features into the learned features along the base, while the final output layer, equivalently the latent/head estimator $\hat g_N$, is responsible for using these features inside the learned fiber for estimation. To illustrate this point, we let
\[
    \hat a_N(\boldsymbol w)
    :=
    \mathbb E[\hat A\mid \hat W=\boldsymbol w],
    \qquad
    \hat\varphi_\lambda\text{-a.e.},
\]
since $|\hat A|\le A$, we have $\hat a_N\in L^2(\hat\varphi_\lambda)$.
By \eqref{eq:correspondence_g_a}, there exists a unique $\hat g_N\in\cH_\fea$ such that the following polar decomposition of the limit estimator of MFLD holds $\bP_X$ almost surely for any $\vx$
\[
    \hat f_N(\boldsymbol x)
    =
    \mathbb E_{(\hat A,\hat W)\sim\hat\nu_\lambda}
    \left[
        \hat A\sigma(\langle \hat W,\boldsymbol x\rangle)
    \right] 
    =
    \mathbb E_{\hat W\sim\hat\varphi_\lambda}
    \left[
        \hat a_N(\hat W)
        \sigma(\langle \hat W,\boldsymbol x\rangle)
    \right] 
    = \hat g_N(\phi_\fea(\vx)).
\]
Therefore, $\hat g_N$ is a naturally arising latent estimator, whose role is to make predictions using the features learned by $\phi_\fea$. In particular, it does not directly observe the input $\vx$, but only its feature map representation $\phi_\fea(\vx)$.

The important point is that this RKHS $\cH_\fea$ is not fixed before training: it is learned during training. In feature engineering, a statistician chooses a feature map a priori, and estimation is then carried out in this fixed structure. In MFLD, by contrast, the hidden-layer marginal $\hat\varphi_\lambda$ is produced by the algorithm itself. Therefore, the kernel $K_\fea$ (and, consequently, its canonical feature map $\phi_\fea$, RKHS and top eigenvectors) adapts to this specific supervised regression problem through the information carried by the training data. This is a key difference between feature learning and feature engineering: MFLD learns not only the coefficients, but also the feature representation itself in which those coefficients are estimated. The alignment property can therefore be used along directions discovered by training, rather than along directions fixed a priori.

\begin{Remark}[Connection with Barron spaces]
The introduction of $\cH_\fea$ is natural from the viewpoint of Barron spaces. Functions in Barron spaces are well approximated by shallow neural networks \cite{barron_universal_1993} and so Barron spaces serve as statistical models. In the case $W=1$ and for the ReLU activation function,  the Barron space is equivalent to $\bigcup_{\varphi\in\mathcal P(S_2^{d-1})}\cH_\varphi$, where $\cH_\varphi$ is the RKHS generated by $K_\varphi(\boldsymbol{x}_1,\boldsymbol{x}_2)=\mathbb E_{\boldsymbol{W}\sim\varphi}[\sigma(\langle \boldsymbol{W},\boldsymbol{x}_1\rangle)\sigma(\langle \boldsymbol{W},\boldsymbol{x}_2\rangle)]$, see \cite{eMathematicalPerspectiveMachine2023}. From this perspective, feature learning in MFLD can be understood as selecting, from this family of RKHSs, the specific data-dependent input-space RKHS $\cH_\fea$ determined by $\hat\varphi_\lambda$. The associated coefficient space is $L^2(\hat\varphi_\lambda)$, and a coefficient function $a\in L^2(\hat\varphi_\lambda)$ induces the predictor $x\mapsto\langle a,\varphi_{\neu}(x)\rangle_{L^2(\hat\varphi_\lambda)}=\int_{WB_2^d}a(\boldsymbol w)\sigma(\langle \boldsymbol w,x\rangle)\,\mathrm d\hat\varphi_\lambda(\boldsymbol w)$. The polar decomposition $\hat f_N = \hat g_N(\phi_\fea(\vx))$ is particularly relevant in Barron space since the probability measure $\hat\varphi_\lambda$ on $S_2^{d-1}$ plays the role of a '\textit{multi-dimensional direction/angle}' and $\hat a_N$ in the definition of the latent estimator $\hat g_N(\phi_\fea(\vx))  = \langle \hat a_N,\varphi_{\neu}(\vx)\rangle_{L^2(\hat\varphi_\lambda)}$ plays the role of a '\textit{multi-dimensional radius}' in the multiple directions -- in particular, the top eigenvectors -- in $\cH_\fea \cong \cH_{\hat \varphi_\lambda}$.
\end{Remark}

\subsection{Self-Regularization and Alignment in the Learned Feature Space}

We explore Definition~\ref{def:feature_learning_general}'s implications for both interpretability and generalization of neural networks
 
\paragraph{Regarding interpretability.} Definition~\ref{def:feature_learning_general} answers the question raised at the beginning of this paper: \emph{what features are learned by neural networks?} The learned features are the top $k$ eigenvectors in $\cH_\fea$ (i.e. the eigenfunctions of its operator $\Sigma_{\fea}$). Admittedly, to theoretically dissect the significance of these features—for instance, to explore which physically meaningful properties of $f^\star$ they reflect—it remains necessary to conduct case-specific analyses within concrete model frameworks.

\paragraph{In terms of generalization.} Factorizations depicted in Figure~\ref{fig:estimator_factorization} naturally yields to the estimation error decomposition: 
    \[
    \|\hat f_N-f^\star\|_{L^2(\mathbb P_X)}
    \le
    \| \hat g_N\circ\phi_\fea - g_\fea\circ\phi_\fea\|_{L^2(\mathbb P_X)}
    +
    \| g_{\fea}\circ\phi_\fea-f^\star\|_{L^2(\mathbb P_X)},
    \]
    which provides a new approach for analyzing the prediction performances of estimators possessing the feature-learning property. The estimation error of $f^\star$ can be decomposed into the estimation error of $g_\fea$ and the approximation error of $f^\star$ to the learned feature space $\cH_\fea$. The hope is that the created $k$-dimensional representation of $f^\star$ in $\cH_\fea$ can be exploited to achieve a final $\cO(\sqrt{k/N})$ convergence rate.

The feature-learning property of $\hat f_N$ depends on both the (created) alignment of $f^\star$ with to top eigenvectors of $\cH_\fea$ as well as on the ability  of the latent estimator $\hat g_N$ to exploit this alignment, i.e. if it satisfies the alignment property as introduced in Definition~\ref{def:alignment_property}. The following proposition shows that for MFLD, such a $\hat g_N$ always exists (as constructed via \eqref{eq:correspondence_g_a}) and that it has a self-regularization / implicit bias property. The proof of the following proposition is deferred to Appendix~\ref{app:mfld-foundations}.

\begin{Propositionbis}{prop:hat_g_N}
Assume $0<A<\infty$ and $\lambda>0$. There exists an extended-valued convex functional $\Psi:\cH_\fea\to\mathbb R\cup\{+\infty\}$ such that, $\mathbb P^{\otimes N}$-almost surely, the latent estimator $\hat g_N(\cdot)$ satisfies 
\[
    \hat g_N\in
    \argmin_{g\in\cH_\fea}
    \left\{ \frac1N\sum_{i=1}^N
        \Bigl(Y_i-g\circ\phi_\fea(X_i)\Bigr)^2
        +
        \lambda\Psi(g)
    \right\}.
\]
Moreover, $\Psi$ is $1/A^2$-strongly convex:
for any $g_1,g_2\in\operatorname{dom}(\Psi)$ and any
$\zeta_2\in\partial\Psi(g_2)$,
$$ \Psi(g_1)-\Psi(g_2)-\langle\zeta_2,g_1-g_2\rangle_{\cH_\fea}
\ge (2A^2)^{-1}\|g_1-g_2\|_{\cH_\fea}^2.$$
\end{Propositionbis}

Proposition~\ref{prop:hat_g_N} yields three key insights:
\begin{enumerate}
    \item Proposition~\ref{prop:hat_g_N} characterizes the self-regularization effect in the limit of MFLD. Specifically, its latent estimator $\hat{g}_N$ is a RERM on $\cH_\fea$ with a data-dependent regularization function that behaves as a ridge-like regularization. This observation shows that MFLD used to train a SNN has an implicit bias toward some structure: signals that are mostly supported on the top eigenvectors of $\cH_\fea$. It is interesting to note that $\hat f_N = f_{\hat\nu_\lambda}$ has an explicit\footnote{The negative Shannon entropy is sometimes looked by the MFLD community as an implicit bias appearing only in the long-time limit \eqref{eq:def_hat_nu} but not explicitly in noisy gradient descent/flow (even though, it appears explicitly in \eqref{eq:nonlinear_Fokker_Planck} as the Wasserstein's gradient flow of the entropy regularized empirical risk). However, starting our analysis from the definition of $\hat\nu_\lambda$, the negative Shannon entropy appears explicitly.} bias toward probability measures with a small negative Shannon entropy, see \eqref{eq:def_empirical_free_energy}. The connection between the two explicit and implicit biases is not straightforward and needs to go through the factorization via the learned feature space  $\cH_\fea$ (see Figure~\ref{fig:estimator_factorization}). 

    \item Furthermore, this regularization term possesses a nontrivial Bregman divergence almost surely, which implies that $\Psi$ is strongly convex with respect to $\|\cdot\|_{\cH_\fea}$ norm. Hence, the head estimator $\hat g_N$ is a  'ridge-type' regression estimation. Furthermore, we know that ridge (and many other spectral methods)  satisfy the alignment property from Definition~\ref{def:alignment_property} \cite{gavrilopoulosGeometricalAnalysisKernel2025,lecueSharpConvergenceRates2025}, and so, it is legitimate to anticipate that it is also the case for $\hat g_N$, even though we were not able to formally prove it in this work since it may require to further develop a specific feature space decomposition analysis \cite{lecueSharpConvergenceRates2025} for the regularization function $\Psi$ appearing in Proposition~\ref{prop:hat_g_N}.

    \item In existing studies on the feature learning theory of shallow neural networks, many works model the training process by manually modifying the training algorithm such that the output layer performs a ridge regression (e.g., the two-stage training and two-timescale training; see \cite{bietti_learning_2022,berthier_learning_2024,marion_leveraging_2023,takakura_mean-field_2024}). Proposition~\ref{prop:hat_g_N} theoretically justifies and confirms the validity of such modeling. It also shows that SNN trained by MFLD does it by itself: there is no need to implicitly force this ridge regularization; it does it implicitly.
\end{enumerate}

\section{Low-Temperature Geometry and Estimation for Spherical MFLD in Gaussian Index Models}\label{sec:feature_learning_Gaussian_single_multi_index}

In this section, we verify the feature-learning property of spherical MFLD for Gaussian index models. The central object is the long-time limit $\hat\varphi_\lambda$ of the spherical nonlinear Fokker--Planck equation. We show that, in the low-temperature regime $\lambda = o(1)$, this stationary hidden-layer marginal develops a multi-spike structure on $S_2^{d-1}$, with its local barycenters concentrating near the hidden indices with high probability in Gaussian multi-index problems. Moreover, in the Gaussian single-index model, it exhibits a concentration phenomenon that depends on the parity of the information index in low temperature regime. This low-dimensional structure formed on the base appearing in the low-temperature regime is then used to obtain sharper estimation error rates for the regression function via the feature learning property.

Throughout this section, we consider Gaussian single-index and multi-index supervised regression models with well-specified link functions. The estimator is spherical MFLD on $S_2^{d-1}$, defined later in \eqref{eq:def_hat_varphi_sphere}.

\subsection{Gaussian index models as a lens on Statistical Feature Learning}

Gaussian single- and multi-index models provide a useful lens through which to study the statistical notion of feature learning developed in this paper. In these models, the regression function depends on the covariate only through one or several unknown one-dimensional projections. If the statistician knows in advance that the problem possesses a low-dimensional structure, then one could have designed a dedicated estimator (via an explicit regularization term, for instance) specifically for this problem. This is the viewpoint of feature engineering and structure inducing regularization. Our goal here is different. We ask whether a generic mean-field neural-network model, trained using MFLD without being given the fact that a low-dimensional structure exists, can learn a feature space in which the target function is well approximated by a low-dimensional object that can be efficiently estimated.

This distinction is the reason why index models are useful in this section. Classical methods for single- and multi-index models often use the prior knowledge that the target has an index structure, see \cite{brunaSurveyAlgorithmsMultiindex2025,gaiffasOptimalRatesAdaptation2007}. In contrast, in neural-networks theory, one would like to understand when an association 'neural network architecture / algorithm' can discover this structure from the data by itself with no a priori given hint on that structure. If this happens, then the improvement in the estimation error rate is not due to a hand-crafted feature map, but is due to the feature learning capability of the training dynamics. This viewpoint has motivated a growing line of work on neural networks for single- and multi-index models \cite{ben_arous_online_2021,ben_arous_high-dimensional_2022,damian_neural_2022,ba_high-dimensional_2022,abbe_sgd_2023,dandi_how_2024,cui_asymptotics_2024,mousavi-hosseini_neural_2023,damian_smoothing_2023,mousavi-hosseini_towards_2023,leeNeuralNetworkLearns2024,nichani_provable_2023,bietti_learning_2022,bietti_learning_2025,ba_learning_2023,han_precise_2025,montanari_dynamical_2025,montanariPhaseTransitionsFeature2026}.

We then introduce the well-specified Gaussian index models considered in this section. Let $(\mathrm{He}_j)_{j\ge 0}$ be the probabilist's one-dimensional Hermite polynomials \cite[pp. 16]{pisier_volume_1989}, normalized by
$
    \mathbb E[\mathrm{He}_j(G)\mathrm{He}_k(G)]
    =
    k!\mathbf 1_{j=k},
    \mbox{ where }
    G\sim \mathcal N(0,1).
$
For the chosen activation function $\sigma$, define its $k$-th Hermite coefficient by
$
    b_k=\mathbb E[\sigma(G)\mathrm{He}_k(G)].
$
The information exponent of $\sigma$ is
\[
    \mathrm{IE}(\sigma)
    =
    \min\{k\in\mathbb N_+:\ b_k\neq 0\},
\]
with the convention that $\mathrm{IE}(\sigma)=+\infty$ if no such $k$ exists. This index measures the first Hermite level at which the activation function carries information about a hidden direction $\vw^\star_j$ (see the definition of the multi-index model below). In many analyses of single-index learning, this quantity appears as a measure of difficulty. These difficulties arise from two main sources: statistical challenges and challenges related to the training dynamics.
\begin{enumerate}
    \item From a statistical perspective, the information index determines the phase transition point of the alignment efficiency between the regression function and the feature functions of a given RKHS; see \cite[Definition 6]{lecueSharpConvergenceRates2025}. For example, for analytic spectral methods on random features kernels with well-specified activation functions (that is, on $\cH^{(0)}$), such as kernel ridge regression, gradient descent, or gradient flow, one can prove that their estimation error for $f^\star$ is at least $\bE[\xi^2]\frac{d^{\mathrm{IE}(\sigma)}}{N}$~ \cite{lecueSharpConvergenceRates2025}. Indeed,  when  $\sigma$ is orthogonal to the first $k$ one-dimensional Hermite polynomials in $L^2(G)$, so is $\sigma(\inr{\vw, X})$ to all $d^{\mathrm{IE}(\sigma)}$, $d$-dimensional Hermite polynomials in $\bR^d$.

    \item From the perspective of training dynamics, the information index characterizes the first feature direction captured after the algorithm starts to run, \cite{ben_arous_online_2021,ben_arous_high-dimensional_2022}. When $d$ is large, the initialization of SGD makes the landscape near the equator $\{\vw:\langle\vw,\vw_\star\rangle=0\}$ relatively flat, and hence the algorithm remains near saddle points for a long time, until it escapes at around time $t\sim d^{\mathrm{IE}(\sigma)/2-1}$.
\end{enumerate}


\begin{Assumption}[Single-index model]\label{ass:fixed-output-general-ie}
    For fixed $d\geq 2$, let $X\sim \mathcal N(\boldsymbol 0,I_d)$. The response variable is generated as $Y=f^\star(X)+\xi$, where $\xi$ is independent of $X$, satisfies $\mathbb E[\xi]=0$, and there exists $B_\xi>0$ such that $\|\xi\|_{L^\infty}\leq B_\xi$.
    There exists $\boldsymbol w^\star\in S_2^{d-1}$ such that $f^\star(x)=\sigma(\langle \boldsymbol w^\star,x\rangle)$. The activation function satisfies $\sigma\in C_b^3(\mathbb R)$, and there exist constants $B_\sigma,L_\sigma,M_\sigma,T_\sigma>0$ such that $\|\sigma\|_{L^\infty}\leq B_\sigma$, $\|\sigma'\|_{L^\infty}\leq L_\sigma$, $\|\sigma''\|_{L^\infty}\leq M_\sigma$, and $\|\sigma^{(3)}\|_{L^\infty}\leq T_\sigma$.
Moreover, the information exponent of $\sigma$ is finite: $\mathrm{IE}(\sigma)<\infty$.
\end{Assumption}

\begin{Assumption}[Multi-index model]\label{ass:multi_index}
    For fixed $d\geq 2$, let $X \sim \mathcal{N}(\boldsymbol{0}, I_d)$, and let $\xi$ be a zero-mean random variable independent of $X$, with $\|\xi\|_{L^\infty}\le B_\xi$. There exist $\boldsymbol{w}_1^\star, \dots, \boldsymbol{w}_M^\star \in S_2^{d-1}$ and $a_1^\star,\cdots,a_M^\star>0$ such that $\sum_{j=1}^Ma_j^\star = 1$ and $f^\star(\cdot) =  \sum_{j=1}^M a_j^\star\sigma(\langle \boldsymbol{w}_j^\star, \cdot \rangle)$. For fixed $M\geq 2$,
    \begin{equation}\label{eq:target_separation}
    \Delta_\star:=\min_{i\neq j}\bigl(1-\langle \boldsymbol{w}_i^\star,\boldsymbol{w}_j^\star\rangle\bigr)>0.
    \end{equation} The activation function is $\sigma \in C_b^3(\mathbb{R})$, and there exist $B_\sigma, L_\sigma, M_\sigma$, and $T_\sigma$, satisfying $\|\sigma\|_{L^\infty} \leq B_\sigma, \|\sigma'\|_{L^\infty} \leq L_\sigma, \|\sigma''\|_{L^\infty} \leq M_\sigma, \|\sigma^{(3)}\|_{L^\infty} \leq T_\sigma$. $b_1 \neq 0, \dots, b_M \neq 0$, where $b_j$ is the $j$-th Hermite coefficient of $\sigma$.
\end{Assumption}

The assumptions above are well-specified in the following sense. The same activation function $\sigma$ is used both in the target function and in the learning algorithm. Thus the main question is not whether the neural-network class can represent the target; it can. The main question is to show whether that the training procedure can learn that such a structure exists and then find and exploit the hidden directions (without being told that they exist). Comments on the activation functions are given in Section~\ref{sec:comment_activation_functions} in appendix.


In this section, we introduce a special form of MFLD in order to disentangle the feature learning phase from the 'feature utilization' phase, and thereby study only the features learned by $\phi_\fea$. We study these models using MFLD on the Euclidean sphere, namely, we take $\mathcal{F} = \{f_\varphi(\cdot) = \int_{S_2^{d-1}}\sigma(\langle \cdot,\boldsymbol{w}\rangle)\,\mathrm{d}\varphi(\boldsymbol{w}): \varphi\in\mathcal{P}(S_2^{d-1})\}$. Let $\varphi_0$ and $\tau$ be the uniform distributions over $S_2^{d-1}$. Let $\nabla_S$ be the Riemannian gradient on $S_2^{d-1}$, and $\nabla_S\cdot$ be the Riemannian divergence (see \cite[Chapter 3]{boumal_introduction_2023}). Let $P_N\ell_\varphi = \frac{1}{N}\sum_{i=1}^N\ell(f_\varphi(X_i),Y_i)$ be the empirical risk. For any $t\geq 0$ and $\lambda\geq 0$, let $\varphi_t(\mathrm{d}\vw) = \rho_t(\vw)\tau(\mathrm{d}\vw)$, where
\begin{align}\label{eq:def_hat_varphi_sphere}
    \partial_t\rho_t = \nabla_S\cdot(\rho_t\nabla_S\frac{\delta P_N\ell_{\varphi_t}}{\delta\varphi}+\lambda\nabla_S\rho_t).
\end{align}We let $\hat\varphi_\lambda = \lim_{t\to\infty}\varphi_t$ in KL divergence (hence $f_{\varphi_t}\to \hat f_N=f_{\hat\varphi_\lambda}$ in $L^2(\bP_X)$), whose existence and convergence speed are shown in Proposition~\ref{prop:convergence_spherical_MFLD} in the Appendix.
The spherical MFLD $(\varphi_t)_{t\geq 0}$ differs slightly from the previously studied $(\nu_t)_{t\geq 0}$ in that here we only train the hidden layer. Although this simplification is made to reduce the difficulty of the analysis, this form of MFLD still exhibits the feature-learning property. We emphasize that $\hat\varphi_\lambda$ is a \emph{random} probability measure. Since the nonlinear Fokker-Planck equation involves a random interaction potential function, all statements made below that hold with high probability are understood with respect to $\bP^{\otimes N}$, the joint distribution of the data.

\subsection{Low-Temperature Multi-Spike Structure on the Base}\label{sec:stationary_distribution_concentration}

We next record the low-temperature structure of the stationary hidden-layer marginal $\hat\varphi_\lambda$ - the main probabilistic contribution of the paper. This result describes the geometric structure formed on the base by the long-time limit of the spherical nonlinear Fokker--Planck equation. Quantitatively, it is the localization estimate underlying the feature-learning mechanism and the estimation error convergence rates proved later.

\paragraph{Gaussian multi-index problem.} The proof of Proposition~\ref{prop:stationary_distribution} is given in Section~\ref{sec:proof_statioanry}.
Recall that $\psi$ is defined by $\psi(0)=0$ and $\psi(t)=t(1+\log(e/t))$ for $t>0$.

\begin{Proposition}\label{prop:stationary_distribution}
Grant Assumption~\ref{ass:multi_index}. Let $0<c_0<1$, $C\geq 1$ and $C_0\geq e$ be absolute constants depending only on $M,\Delta_\star^{-1},B_\sigma,L_\sigma,M_\sigma,T_\sigma,B_\xi$ and $\max_{1\leq k\leq M}|b_k|^{-1}$. For any $x\geq 1$, $N\geq 2$ and $\lambda d\leq c_0$, let
\begin{align*}
        r_*^2 = C\left[
        \frac{Md+d\log(C_0dN)+x}{N}
        +
        \psi(\lambda d)
        \right].
\end{align*}
Let $V_1,\ldots,V_M$ be the Voronoi partition generated by $\vw_1^\star,\ldots,\vw_M^\star$, that is, $V_j = \{\vw\in S_2^{d-1}: j=\min(\argmin_{1\leq\ell\leq M}\|\vw-\vw_\ell^\star\|_2)\}$. There is a constant $C_{\rm vor}$, depending only on $M$, $\Delta_\star^{-1}$, and $\max_{1\le k\le M}|b_k|^{-1}$, such that with probability at least $1-4\exp(-x)$,
\begin{align}\label{eq:uniform_parameter_recovery}
&\max_{1\le j\le M}\bigg|\int_{V_j}\hat\varphi_\lambda(\mathrm{d}\vw)-a_j^\star\bigg|+\max_{1\le j\le M}\int_{V_j}\big\|\vw-\vw_j^\star\big\|_2^2\hat\varphi_\lambda(\mathrm{d}\vw)\le C_{\rm vor}r_*.
\end{align}In particular, for any $\rho>0$, let $S_\rho = \cup_{j=1}^M\{\vv\in S_2^{d-1}:\|\vv-\vw_j^\star\|_2\leq \rho\}$. Then with the same probability, $\hat\varphi_\lambda(S_2^{d-1}\backslash S_\rho)\leq\frac{r_*}{\rho^2}$.
\end{Proposition}

When $\lambda=o(1)$, the rate $r_*$ tends to zero. Hence, throughout the low-temperature regime, the stationary measure of the spherical nonlinear Fokker--Planck equation develops a multi-spike structure around the true parameters. We have also obtained (see Proposition~\ref{prop:multi-index-voronoi-concentration}) an  upper bound for $\|\int_{V_j}(\vw-\vw_j^\star)\,\mathrm{d}\hat\varphi_\lambda(\vw)\|_2$.  That is, as $N\to\infty$, the centroid of the stationary distribution within each Voronoi cell $V_j$ associated with each hidden index $\vw_j^\star$, that is, $\int_{V_j}\vw\mathrm{d}\hat\varphi_\lambda$, converges in probability to $\vw_j^\star$ in the $\|\cdot\|_2$ norm, while the corresponding weight of the stationary distribution within that Voronoi cell, that is, $\int_{V_j}\mathrm{d}\hat\varphi$, converges in probability to $a_j^\star$. In this sense, after evolving from the uniform initialization, the nonlinear Fokker--Planck equation recovers the hidden parameters of the regression function. This is the structure formed on the base in the feature-learning mechanism.

We emphasize that this structure is not imposed by an explicit sparsity-inducing regularization. Neither the spherical nonlinear Fokker--Planck equation nor the variational problem defining $\hat\varphi_\lambda$ contains such a term. The negative Shannon entropy acts in the opposite direction: it penalizes sparse or singular hidden-layer distributions, since these measures have infinite negative Shannon entropy relative to the uniform measure. Nevertheless, the dynamics produces an approximately sparse stationary measure and achieves parameter recovery. In the neural network literature, this phenomenon is commonly referred to as implicit regularization; see, for instance, \cite{bartlett_deep_2021}. Proposition~\ref{prop:stationary_distribution} shows that, in the present setting, this implicit regularization is realized through the feature-learning mechanism.

By contrast, when $\lambda=\omega(1)$, Lemma~\ref{lemma:high_temperature_regime} in the appendix shows that, for every $\rho\to0$ and Borel sets sequence $(A_\rho)_{\rho>0}$ such that $\tau(A_\rho)\to 0$, one has $\hat\varphi_\lambda(S_2^{d-1}\backslash A_\rho)\geq1-o(1)$ almost surely. Thus the multi-spike concentration of the stationary hidden-layer marginal undergoes a sharp phase transition at the temperature scale $\lambda\asymp1$.

Although this involves a slight abuse of terminology, we shall still refer to the phenomenon described in Proposition~\ref{prop:stationary_distribution} as a concentration property. We prefer not to use the term localization, since it can easily be confused with localization techniques in mathematical statistics, \cite{koltchinskii_oracle_2011}. From the perspective of statistical physics, Proposition~\ref{prop:stationary_distribution} may be viewed as an analogue of the Bovier--Gayrard localization theorem for the Hopfield model in mean-field spin glasses; see \cite[Theorem 4.3.2]{talagrandMeanFieldModels2011}.

\paragraph{Gaussian single-index problem.}
In Proposition~\ref{prop:stationary_distribution}, the concentration phenomenon we obtain is not measure concentration in the sense of L\'evy--Milman \cite{milmanNewProofTheorem1971,ledouxConcentrationMeasurePhenomenon2005}. In this paragraph, we show that, in the Gaussian single-index problem, one can observe such a measure concentration phenomenon, and that this phenomenon depends strongly on the parity of the information index.

\begin{Proposition}\label{prop:concentration_Levy_Milman}
Grant Assumption~\ref{ass:fixed-output-general-ie}. Let $\lambda>0$ and $\varepsilon = \frac{2\sqrt{\mathrm{IE}(\sigma)!}}{\kappa_{\mathrm{IE}(\sigma)}|b_{\mathrm{IE}(\sigma)}|} r_*$, where $r_*^2 \sim \frac{d\log(dN)}{N}+\psi(\lambda d)$. For all $x>0$, with probability at least $1-4\exp(-x)$, the following hold.
\begin{enumerate}
    \item If $\mathrm{IE}(\sigma)$ is odd: Let $Z\sim \hat\varphi_\lambda$. For any $1$-Lipschitz function $F:S_2^{d-1}\to\bR$ (with respect to $\|\cdot\|_2$),
    \begin{align*}
        \forall t > 2\sqrt{2\varepsilon},\quad  \bP\left[ \left| F(Z) - \bE [F(Z)] \right| \geq t \right]\leq \frac{8\varepsilon}{t^2}.
    \end{align*}
    Moreover, for any $\rho>0$, $\hat\varphi_\lambda(S_2^{d-1}\backslash B_2(\vw_\star;\rho))\leq \frac{\varepsilon}{\rho^2}$, where $B_2(\vw_\star;\rho)=\{\vw:\|\vw-\vw_\star\|_2\leq \rho\}$.

    \item If $\mathrm{IE}(\sigma)$ is even: For all $\vw\in S_2^{d-1}$, let $[\vw] := \{\vw,-\vw\}$ be the antipodal equivalence class of $\vw$. Let $\bR\bP^{d-1}=S_2^{d-1}/\{\pm 1\}$ be the equivalence class modulo sign, with quotient map $\pi:\vw\in S_2^{d-1}\to [\vw] \in  \bR\bP^{d-1}$. Define $d_{\bR\bP}([\vu],[\vv])=\min\{\|\vu-\vv\|_2,\|\vu+\vv\|_2\}$ as the projective metric. Let $\bar\varphi_\lambda = \pi_\sharp\hat\varphi_\lambda$ and $Z\sim\bar\varphi_\lambda$. Then, for any $1$-Lipschitz function $G:\bR\bP^{d-1}\to\bR$ with respect to metric $d_{\bR\bP}$,
    \begin{align*}
        \forall t>2\sqrt{2\varepsilon},\quad \bP\left( \left| G(Z) - \bE[G(Z)] \right| \geq t \right)\leq \frac{8\varepsilon}{t^2}.
    \end{align*}Moreover, for any $\rho>0$, $\bar\varphi_\lambda(\bR\bP^{d-1}\backslash B_{\bR\bP}([\vw_\star];\rho))\leq \frac{\varepsilon}{\rho^2}$ and $\hat\varphi_\lambda(S_2^{d-1}\backslash (B_2(\vw_\star;\rho)\cup B_2(-\vw_\star;\rho)))\leq \frac{\varepsilon}{\rho^2}$, where $B_{\bR\bP}([\vw_\star];\rho) = \{[\vw]\in\bR\bP^{d-1}:d_{\bR\bP}([\vw];[\vw_\star])\leq \rho\}$.
\end{enumerate}   
\end{Proposition}

We now provide some comments on this proposition.
\begin{enumerate}

    \item The concentration behavior of the long-time limit of the spherical nonlinear Fokker-Planck equation depends on the parity of the information index. This dependence on the information index is different from the relation between the training dynamics of SGD and the information index discovered in \cite{ben_arous_high-dimensional_2022} since the latter mainly concerns the landscape of the training dynamics near its initial position, namely around $t\approx 0$. In contrast, Proposition~\ref{prop:concentration_Levy_Milman} concerns the concentration behavior of the long-time limit, namely as $t\to\infty$. Its relation with $\mathrm{IE}(\sigma)$ is as follows. By the Hermite decomposition of the estimation error (proved in \eqref{eq:Hermite_expansion_estimation_error}) $\| f_\varphi - f^\star \|_{L^2(\bP_X)}^2 = \sum_{k\geq 1}\frac{b_k^2}{k!} \| \int \vw^{\otimes k}\varphi(\mathrm d\vw) - \vw_\star^{\otimes k}\|_F^2$, a small estimation error forces $\| \int \vw^{\otimes \mathrm{IE}(\sigma)}\varphi(\mathrm d\vw) - \vw_\star^{\otimes \mathrm{IE}(\sigma)}\|_F\geq \langle\vw_\star^{\otimes \mathrm{IE}(\sigma)}, \vw_\star^{\otimes\mathrm{IE}(\sigma)} - \int\vw^{\otimes\mathrm{IE}(\sigma)}\varphi(\mathrm{d}\vw)\rangle_F =  1-\int\langle\vw,\vw_\star\rangle^{\mathrm{IE}(\sigma)}\mathrm{d}\varphi(\vw)$ to converge to $0$ (this inequality is proved in Proposition~\ref{prop:m-chaos-localization} in appendix). When $\mathrm{IE}(\sigma)$ is even, one has $\langle \vw^{\otimes \mathrm{IE}(\sigma)},\vw_\star^{\otimes \mathrm{IE}(\sigma)}\rangle_F = |\langle \vw,\vw_\star\rangle|^{\mathrm{IE}(\sigma)}$, and hence $d_{\bR\bP}^2([\vw],[\vw_\star]) \leq 2(1-|\langle\vw,\vw_\star\rangle|)\leq 2(1-|\langle\vw,\vw_\star\rangle|^{\otimes\mathrm{IE}(\sigma)})$ is forced to decrease to $0$ when the estimation error decreases to $0$. Therefore, when $\mathrm{IE}(\sigma)$ is even, the stationary distribution of the nonlinear Fokker-Planck equation develops a two-spike structure on the sphere: its probability mass is primarily concentrated near
    $\{\vw:\|\vw-\vw_\star\|_2\leq \rho\}\cup\{\vw:\|\vw+\vw_\star\|_2\leq \rho\}$,
    where $\rho>0$ is a small constant.
    In contrast, when $\mathrm{IE}(\sigma)$ is odd, the inequality
    $\int \big(1-\langle \vw,\vw_\star\rangle^{\mathrm{IE}(\sigma)}\big)\mathrm d\varphi(\vw)\leq \big\|\int \vw^{\otimes\mathrm{IE}(\sigma)}\varphi(\mathrm d\vw)-\vw_\star^{\otimes\mathrm{IE}(\sigma)}\big\|_F$
    implies a single-spike structure concentrated near
    $\{\vw:\|\vw-\vw_\star\|_2\leq \rho\}$.

    \item From the statistical physics viewpoint, Proposition~\ref{prop:stationary_distribution} should not be interpreted as a statement about an exact $\mathbb Z_2$-symmetry of the self-consistent Hamiltonian. It is instead a leading-order selection statement for a low-temperature nonlinear Gibbs fixed point. In the Curie--Weiss and Hopfield models, the antipodal structure of the low-temperature states is tied to the even dependence of the Hamiltonian on the relevant order parameter \cite[pp. 239]{talagrandMeanFieldModels2011}. In the present problem, the full empirical Hamiltonian need not be antipodally symmetric; the dichotomy is determined by the first non-vanishing Hermite chaos. If $\mathrm{IE}(\sigma)$ is odd, the leading interaction distinguishes $\boldsymbol w_\star$ from $-\boldsymbol w_\star$, and the low-temperature geometry selects an oriented direction on $S_2^{d-1}$. If $\mathrm{IE}(\sigma)$ is even, the leading interaction factors through the antipodal quotient, and the natural low-temperature state space is $\mathbb{RP}^{d-1}$.
    
    \item 

    As a Riemannian manifold, or more generally as a metric measure space, the appearance of measure concentration on $\bR\bP^{d-1}$ is to be expected; note that $d_{\bR\bP}$ is equivalent to the Riemannian geodesic distance on $\bR\bP^{d-1}$; see \cite[Section 5.7]{bakryAnalysisGeometryMarkov2014}. However, we emphasize that the novelty of Proposition~\ref{prop:concentration_Levy_Milman} lies in the fact that $\bar\varphi_\lambda$ is not the Riemannian volume measure, but rather the long-time limit of a random spherical nonlinear Fokker-Planck equation arising from the Gaussian single-index problem, that concentrates around the hidden index. Nevertheless, it still satisfies a measure concentration property. To the best of our knowledge, this phenomenon has not been observed before.
\end{enumerate}

\paragraph{Moments tensors govern the low-temperature geometry.} The low-temperature geometry described in Proposition~\ref{prop:stationary_distribution} and Proposition~\ref{prop:concentration_Levy_Milman} is governed by the convergence of moments tensors. For any $m\in\bN_+$, denote by $\cT_m(\varphi) = \int \vw^{\otimes m}\varphi(\mathrm d\vw)$ the $m$-th moments tensor, whenever the integral exists. The convergence of $\|\cT_m(\hat\varphi_\lambda)-\cT_m(\varphi^\star)\|_F$ is a widely studied object in nonconvex optimization, for instance in the likelihood landscape of mixture models, \cite{fanLikelihoodLandscapeMaximum,katsevichLikelihoodMaximizationMoment2023} and in the saddle-to-saddle dynamics of gradient flows in learning Gaussian multi-index models, \cite{bietti_learning_2025}. In the Gaussian multi-index problem, we proved in Proposition~\ref{prop:multi-index-localization} that for any degree $M$ polynomial $Q:\vw\in S_2^{d-1}\mapsto c_0 + \sum_{m=1}^M\langle A_m,\vw^{\otimes m}\rangle_F \in \bR$ with some symmetric tensors $A_m\in\mathrm{Sym}_m(\bR^d)$ (see, for instance, \cite[Equation 9.7]{michalekInvitationNonlinearAlgebra2021} for the correspondence between homogeneous polynomials and symmetric tensors), the following holds almost surely:
\begin{align*}
    &\left| \int Q\mathrm{d}\hat\varphi_\lambda - \int Q \mathrm{d}\varphi^\star \right| \leq \sum_{m=1}^M\|A_m\|_F\| \cT_m(\hat\varphi_\lambda) - \cT_m(\varphi^\star)\|_F.
\end{align*}Therefore, the proofs of Proposition~\ref{prop:stationary_distribution} and Proposition~\ref{prop:concentration_Levy_Milman} proceed by constructing suitable polynomials, sometimes referred to as certificates. Therefore, the emergence of the low-temperature geometry of $\hat\varphi_\lambda$ requires the first $M$ moments tensors of $\hat\varphi_\lambda$ to converge, in Frobenius norm, to the corresponding moments tensors of $\varphi^\star$. An analogous phenomenon also holds in the single-index model: its concentration property requires the $\mathrm{IE}(\sigma)$-th moments tensor to converge, in Frobenius norm, to the corresponding moments tensor of $\delta_{\vw_\star}$.

We emphasize that the present paper studies the long-time limit, which is fundamentally different from the finite-time training dynamics considered in the existing literature. In particular, $\frac{b_m}{\sqrt{m!}}\cT_m(\varphi^\star)$ corresponds to the Hermite tensor introduced in \cite{bietti_learning_2025}. In \cite{bietti_learning_2025}, the authors proved that Hermite polynomials govern the saddle-to-saddle dynamics of gradient flow in Gaussian multi-index problems; see also \cite{abbe_sgd_2023}. This is a finite-time phenomenon. In contrast, such a phenomenon is not visible in the long-time regime that interests us here, more precisely when $t=\Omega(\exp(N))$. Understanding feature learning for MFLD at finite times, in particular the estimation error of moments tensors at finite times and its influence on the resulting low-temperature geometry, is an interesting direction for future research. The theoretical framework developed in the next subsection links moments tensors, low-temperature geometry, and their impact on statistical properties. It thereby provides a geometric--statistical perspective for studying the feature learning property of finite-time training dynamics.

\subsection{Dimension Reduction, Top-\texorpdfstring{$k$}{k} Signal Approximation - Theoretical Framework Based on the Feature-Learning Property}\label{sec:proof_strategy}

We now explain how the low-temperature geometry yields a dimension reduction of the statistical model $\mathcal F$. The control of \eqref{eq:uniform_parameter_recovery} shows that the stationary hidden-layer marginal concentrates near the hidden indices. Hence, although $\mathcal F=\{f_\varphi:\varphi\in\mathcal P(S_2^{d-1})\}$ is an infinite-dimensional nonlinear mean-field model, the part selected by $\hat\varphi_\lambda$ is governed by the first-order behavior of the ridge functions near $\boldsymbol w_1^\star,\ldots,\boldsymbol w_M^\star$. More precisely, Proposition~\ref{prop:multi-index-rademacher} in the appendix shows that $f_{\hat\varphi_\lambda}$ lies within a negligible $L^2(\mathbb P_X)$-distance of the first-order space
\begin{align}\label{eq:def_finite_dimensional_space}
&\operatorname{Span}\big\{\sigma(\langle\boldsymbol w_j^\star,\cdot\rangle),\ \boldsymbol x\mapsto\sigma'(\langle\boldsymbol w_j^\star,\boldsymbol x\rangle)\langle\boldsymbol u,\boldsymbol x\rangle:1\leq j\leq M,\ \boldsymbol u\in\bR^d\big\}.
\end{align}
This space has dimension at most $M(d+1)$. Thus the low-temperature multi-spike structure compresses the infinite-dimensional nonlinear model $\mathcal F$ to an $O(Md)$-dimensional first-order model, up to a negligible remainder.

The same finite-dimensional space also captures the main component of the regression function. Since $f^\star(\cdot)=\sum_{j=1}^M a_j^\star\sigma(\langle\boldsymbol w_j^\star,\cdot\rangle)$, the target belongs to the zeroth-order part of \eqref{eq:def_finite_dimensional_space}. Therefore, the low-temperature geometry selects a learned feature subspace in which the main component of $f^\star$ is represented by only $k=M(d+1)$ learned directions. Equivalently, with the choice $\gamma_j=\sigma_j$, the target representative $\mathbf 1_{\fea}\in L^2(\hat\varphi_\lambda)$ has negligible weighted tail energy outside the top $k$ learned directions, namely $\sum_{j>k}\sigma_j\langle g_{\fea},e_j\rangle_{L^2(\hat\varphi_\lambda)}^2=o_{\mathbb P}(1)$ - which we prove later in Corollary~\ref{coro:feature_learning_multi}. This is the top-$k$ signal approximation component of the feature-learning property, obtained from the dimension reduction induced by the low-temperature geometry.

We now summarize the estimation-error framework provided by the feature-learning property. This framework does not replace the uniform convergence argument; since $\hat f_N$ is a data-dependent random function, some uniform or localized stochastic control remains unavoidable. Its role is to identify where this stochastic control should be carried out: not directly over the original infinite-dimensional nonlinear model $\mathcal F$, but around the low-dimensional structure revealed by the learned representation near $f^\star$. The central step is a dimension reduction of $\cF$ around $f^\star$: the learned representation should not merely make $f^\star$ representable, but should make the main component of $f^\star$ effectively approximable by a small number of leading learned directions in the learned feature space. The localized uniform convergence argument is then carried out around these learned directions, yielding an estimation error bound for the final estimator. In the proof, we use localization techniques; see \cite{koltchinskii_oracle_2011}, to establish the connection between the above dimensionality reduction phenomenon and the estimation error. This eventually yields, simultaneously, Proposition~\ref{prop:stationary_distribution}, Proposition~\ref{prop:concentration_Levy_Milman}, and the estimation error convergence rate in Theorem~\ref{thm:fixed-output-general-ie} and Theorem~\ref{thm:multi-index-fixed-output} below.

\subsection{Statistical consequences of the emergence of a low-dimensional structure at low-temperature}

We now record the statistical consequences of the low-temperature geometry described above. The following results show that the dimension reduction induced by the stationary hidden-layer marginal leads to the $d/N$ and $Md/N$ estimation rates where $d$ and $dM$ are the dimensions of the low-dimensional representations of $f^\star$ emerged in the feature space under the single index and multiple index model assumptions respectively.

\subsubsection{The $d/N$ convergence rate in the Gaussian single-index model}

The following theorem is the main result of this subsection. The proof can be found in Section~\ref{sec:proof_thm_fixed-output-general-ie}.

\begin{Theorem}
\label{thm:fixed-output-general-ie}
Grant Assumption~\ref{ass:fixed-output-general-ie}. There exist constants $c_0\in(0,1)$, $C\ge1$, and $C_0\ge e$, depending only on $\mathrm{IE}(\sigma),B_\sigma,L_\sigma,M_\sigma,T_\sigma,B_\xi,|b_{\mathrm{IE}(\sigma)}|^{-1}$, such that the following holds. For any $x\ge1$, $N\ge2$, and any $\lambda\geq0$ such that $\lambda d\le c_0$, let
\begin{align}\label{eq:def_r_star_single}
    r_*^2 = C\left[ \frac{d\log(C_0dN)+x}{N} + \psi(\lambda d) \right].
\end{align}
Then, with probability at least $1-4e^{-x}$,
\begin{align}
\label{eq:result_single_index}
\|f_{\hat\varphi_\lambda}-f^\star\|_{L^2(\bP_X)}^2 + \lambda\Ent_\tau^-(\hat\varphi_\lambda) \le r_*^2,\mbox{ and }\norm{\bE[\hat W^{\otimes \mathrm{IE}(\sigma)}|(X_i,Y_i)_{i=1}^N]-(\boldsymbol{w}^\star)^{\otimes \mathrm{IE}(\sigma)}}_F^2\leq \frac{\mathrm{IE}(\sigma)!}{b_{\mathrm{IE}(\sigma)}^2}r_*^2.
\end{align}
\end{Theorem}
We make several remarks on Theorem~\ref{thm:fixed-output-general-ie}.
\begin{enumerate}
    \item To the best of our knowledge, Theorem~\ref{thm:fixed-output-general-ie} is the first result, under such general assumptions, in which MFLD attains the minimax optimal rate (up to logarithmic factors) for both $f^\star$ and $\boldsymbol{w}^\star$, with the information exponent entering only through constants. This implies that the practitioner need not know in advance that the problem has a well-specified single-index structure: a ``blind'' application of MFLD discovers and adapts to it, exhibiting the feature-learning phenomenon. Moreover, $\lambda\asymp \frac{1}{N}$ is the choice that yields the fastest computational convergence rate of the algorithm among all temperature parameters (see Proposition~\ref{prop:convergence_spherical_MFLD} in the Appendix), while preserving the optimal statistical properties in Theorem~\ref{thm:fixed-output-general-ie}.

    \item We explain the difference between this paper and prior work such as \cite{ben_arous_online_2021}. This paper studies statistical learning problems and thus focuses on the estimation error of the single- or multi-index regression function, that is, $f^\star$, while work such as \cite{ben_arous_online_2021} mainly focuses on the estimation error of the parameter, namely the single-index $\boldsymbol{w}^\star$. These are two different types of problems. In fact, when there is misspecification between the link function and the activation function of the neural network, good parameter recovery may result in poor recovery of the regression function \cite{ben_arous_online_2021}.

    \item Importantly, although the case $\lambda=0$ is allowed, Theorem~\ref{thm:fixed-output-general-ie} does not yield benign overfitting. This is because for a bounded model such as $\mathcal{P}(S_2^{d-1})$, the observation vector $(Y_1, \dots, Y_N)$ does not necessarily lie in $P_{\sigma}\mathcal{F} = \{(f(X_i))_{i=1}^N : f \in \mathcal{F}\}$, and thus an overfitting estimator may not exist (see Proposition~\ref{prop:no_interpolation_spherical_mean_field} for a formal statement). Thus, the case $\lambda=0$ corresponds to empirical risk minimization, or to a mean-field shallow neural network without Langevin diffusion; see, for instance \cite{mei_mean_2018,rotskoff_trainability_2022,chizat_global_2018,sirignano_mean_2020}.

\end{enumerate}

\subsubsection{The $dM/N$ convergence rate in the Gaussian Multi-Index Model}\label{sec:application_Gaussian_Multi_Index}

The main result of this section is the following theorem. The proof of Theorem~\ref{thm:multi-index-fixed-output} can be found in Section~\ref{sec:proof_thm_multi-index-fixed-output}.
\begin{Theorem}\label{thm:multi-index-fixed-output}
    Grant Assumption~\ref{ass:multi_index}.
    For any $x\geq 1$, $N\geq 2$ and $\lambda d\leq c_0$ where $c_0$ is the constant in Proposition~\ref{prop:stationary_distribution}, for $r_*$ in Proposition~\ref{prop:stationary_distribution}, with probability at least $1-4e^{-x}$, for any $1\leq m\leq M$,
\begin{align*}
    &\|f_{\hat\varphi_\lambda}-f^\star\|_{L^2(\bP_X)}^2+\lambda\Ent_\tau^-(\hat\varphi_\lambda)
    \le r_*^2, \mbox{ and } 
    \norm{\bE\left[(\hat W)^{\otimes m}\big|(X_i,Y_i)_{i=1}^N\right]-\sum_{j=1}^Ma_j^\star(\boldsymbol{w}_j^\star)^{\otimes m}}_F^2
    \le \frac{m!}{b_m^2}r_*^2.
\end{align*}
\end{Theorem}
We discuss several implications of Theorem~\ref{thm:multi-index-fixed-output}.
\begin{enumerate}

    \item From a Bayesian viewpoint, $\hat\varphi_\lambda$ may be interpreted as the posterior distribution on the model's parameters $\{\vw_1^\star,\cdots,\vw_M^\star\}$, obtained by evolving the initial uniform prior $\varphi_0$ through MFLD; the information contained in the training samples is incorporated into the hidden-layer distribution through this evolution.

    \item In Theorem~\ref{thm:multi-index-fixed-output}, we require the multi-indices to have a large angular separation, i.e., \eqref{eq:target_separation}. Note that the absolute constant $C$ depends on $\Delta_\star^{-M}$. This implies that when $M$ is large (e.g., $M \gg \exp(d)$), or there exists a very small angular separation between two indices, the rate deteriorates rapidly. We conjecture that, in order to achieve parameter recovery, this separation condition cannot be removed, just as there is a gap between parameter recovery and density estimation in the Gaussian mixture problem; see \cite{dossOptimalEstimationHighdimensional2023}. However, if the goal is only prediction, we conjecture that this condition should be removable.

    \item We emphasize that the minimax optimal rate in Theorem~\ref{thm:multi-index-fixed-output} (as well as Theorem~\ref{thm:fixed-output-general-ie}) is achieved when $\lambda \lesssim 1/N$, which falls into the low-temperature regime of the system; and in particular, when $\lambda=0$, which corresponds to the mean-field neural network without Langevin diffusion. We emphasize that, when $\lambda=0$, we only establish the statistical properties, whereas the convergence of the corresponding algorithm remains an open problem (see \cite{chizat_global_2018}).
    
    \item Although the spherical MFLD can achieve the minimax optimal convergence rate for the estimation error, Proposition~\ref{prop:convergence_spherical_MFLD} in the Appendix shows an $\Omega(\exp(N))$ time guarantee to converge. This is because, in the Wasserstein space $(\cP_{\mathrm{ac}}(S_2^{d-1}),W_2)$ where $W_2$ is the Wasserstein-$2$ distance, the (regularized) empirical loss landscape $\{P_N\ell_\varphi^\lambda:\varphi\in\cP_2(S_2^{d-1})\}$ is no longer displacement convex and thus admits many spurious stationary points \cite[Lemma A.2]{chizat_mean-field_2022}. Consequently, the evolution of the Wasserstein gradient flow can be trapped in metastable states; however, due to the thermal fluctuations induced by the diffusion term in Langevin dynamics, it can escape local minima on a time scale larger than $\exp(N)$ and converge to $\hat\varphi_\lambda$.
\end{enumerate}

The above multi-index result applies under a general separation condition, but its constants may deteriorate rapidly with the number of indices. We now present a complementary refinement for a better-conditioned multi-index structure. The key assumption is that the Gram matrix of the true directions is close to the identity.
This result should be viewed as an $M$-uniform prediction theorem under an additional geometric assumption, rather than as a replacement for the general separated multi-index theorem.

\begin{Assumption}[Multi-index model with frame structure]\label{ass:rip_multi_index}
Let $d\geq 2$, $X\sim\mathcal N(0,I_d)$, and let $\xi$ be a zero-mean random variable independent of $X$, with $\|\xi\|_{L^\infty}\le B_\xi$. There exist $M\ge2$, $\vw_1^\star,\ldots,\vw_M^\star\in S_2^{d-1}$, $a_1^\star,\ldots,a_M^\star>0$, $\sum_{j=1}^M a_j^\star=1$, such that $f^\star(\cdot)=\sum_{j=1}^M a_j^\star\sigma(\langle \vw_j^\star,\cdot\rangle)$. Let $W_\star=[\vw_1^\star|\cdots|\vw_M^\star]\in\mathbb R^{d\times M}$ and $G_\star=W_\star^\top W_\star$. There exists $\eta\in(0,1/2)$, independent of $M,d,N$, such that  $\|G_\star-I_M\|_{\rm op}\le \eta .$
The activation function $\sigma\in C_b^3(\mathbb R)$, and there exist $B_\sigma,L_\sigma,M_\sigma,T_\sigma<\infty$, satisfying $\|\sigma\|_{L^\infty}\le B_\sigma$, $\|\sigma'\|_{L^\infty}\le L_\sigma$, $\|\sigma''\|_{L^\infty}\le M_\sigma$, and $\|\sigma^{(3)}\|_{L^\infty}\le T_\sigma$. Moreover, $b_2b_3b_4\ne0$, where $b_j$ is the $j$-th Hermite coefficient of $\sigma$.
\end{Assumption}

The frame condition says that the true index directions are not only separated, but almost orthonormal as a whole system.
Thus no linear combination of the true features is strongly amplified or strongly cancelled. This is the same geometric idea as the frame assumptions used in compressed sensing, high-dimensional statistics~\cite{buhlmann2011statistics,foucartMathematicalIntroductionCompressive2013} and dictionnary learning~\cite{gribonval2010dictionary}.
In the present proof, this condition replaces separation-based conditioning by uniform frame rates. It is a widely used assumption, as in \cite{dandi_how_2024,bietti_learning_2025}.


\begin{Theorembis}{thm:rip-multi-index-fixed-output}
Grant Assumption~\ref{ass:rip_multi_index}. There exist constants $c_0\in(0,1)$, $C\ge1$, and $C_0\ge e$ that are independent of $M,d,N$, such that the following holds. Let $x\ge1$, $N\ge2$, and $\lambda\ge0$ satisfy $\lambda d\le c_0$. For $\lambda\ge 0$, define
\[
r_{\rm frame}^2:=C\left[\frac{M d\log(C_0dN)+x}{N}+\frac{M (d\log(C_0dN))^2}{N^2}+\psi(\lambda d)\right].
\]
Then, with probability at least $1-4e^{-x}$, for $m=2,3,4$, 
\[
\|f_{\hat\varphi_\lambda}-f^\star\|_{L^2(\mathbb P_X)}^2+ \lambda\Ent_\tau^-(\hat\varphi_\lambda)\le r_{\rm frame}^2,\mbox{ and }\Big\|\int_{S_2^{d-1}} \vw^{\otimes m}\hat\varphi_\lambda(\mathrm{d}\vw)-\sum_{j=1}^M a_j^\star (\vw_j^\star)^{\otimes m}\Big\|_F^2\le \frac{m!}{b_m^2}r_{\rm frame}^2,
\]
If $N\ge d\log(C_0dN)$, then for $\lambda\ge 0$, $\|f_{\hat\varphi_\lambda}-f^\star\|_{L^2(\mathbb P_X)}^2+\lambda\Ent_\tau^-(\hat\varphi_\lambda)\le C\left[\frac{M d\log(C_0dN)+x}{N}+\psi(\lambda d)\right]$.
\end{Theorembis}

The main improvement is in the prediction rate. Up to logarithmic factors, the leading statistical term is $Md/N$, which is the natural parametric scale for $M$ unknown directions in $d$ dimensions, together with their mixture weights. In particular, the constants in the rate do not hide an exponential dependence on $M$. This is the advantage of the frame condition: under a well-conditioned multi-index structure, prediction behaves as if one were estimating $O(Md)$ effective parameters, rather than paying the much larger conditioning cost that appears in the general separated case.

After proving the feature-learning property of MFLD for Gaussian single-/multi-index problems with well-specified link functions, a natural question is whether spherical MFLD can also solve Gaussian single-/multi-index problems with misspecified link functions. In Section~\ref{sec:counter_example}, we provide a counterexample showing that there exists a Gaussian single-index problem defined by a Sobolev link function for which spherical MFLD is not consistent. This naturally leads to the following question: which problems can spherical MFLD effectively solve? We provide such a class of functions in Section~\ref{sec:problems_solved_byMFLD}. We already know how to prove that, when simultaneously training a two-layer mean-field neural network with bias, MFLD achieves minimax optimal convergence rates for the Gaussian multi-index problem (as well as some other high-dimensional regression models); however, this is not the main focus of the present paper, and we will present this part in a future sequel.



\subsection{Feature-Learning Property and Non-Laziness of Spherical MFLD}\label{subsec:FL-to-nonlazy}

We now assemble the preceding ingredients. The low-temperature dimension reduction gives the top-$k$ signal approximation, while the alignment property controls how the latent estimator uses the learned feature space. It remains to record the nontrivial feature evolution: the learned feature kernel cannot remain close to the random-features kernel at initialization. This also show the non-laziness of spherical MFLD.

Recall that for $\varphi\in\mathcal P(S_2^{d-1})$, we define $K_\varphi(\boldsymbol{x},\boldsymbol{x}')=\int_{S_2^{d-1}}\sigma(\langle \boldsymbol{w},\boldsymbol{x}\rangle)\sigma(\langle \boldsymbol{w},\boldsymbol{x}'\rangle)\,\mathrm{d}\varphi(\boldsymbol{w})$. Let $\tau$ be the uniform distribution on $S_2^{d-1}$. The proof of Corollary~\ref{cor:feature-kernel-movement} may be found in Section~\ref{sec:proof_coro_feature_kernel_movement}.

\begin{Corollary}[Feature-kernel evolution]\label{cor:feature-kernel-movement}Let $x\ge1$, $N\ge2$, and $\lambda d\le c_0$, and let $r_*$ be defined as in Theorem~\ref{thm:fixed-output-general-ie} and Theorem~\ref{thm:multi-index-fixed-output} respectively. There exists $C_{\rm ker}\ge1$, depending only on $M,\Delta_\star^{-1},B_\sigma,L_\sigma$, and $\max_{1\le m\le M}|b_m|^{-1}$ in the multi-index case, and depending only on $\mathrm{IE}(\sigma),B_\sigma,L_\sigma,M_\sigma,T_\sigma,B_\xi$, and $|b_{\mathrm{IE}}|^{-1}$ in the single-index case, such that the following holds.
\begin{enumerate}
    \item Grant Assumption~\ref{ass:fixed-output-general-ie}. If $r_*\leq 1$ and $C_{\mathrm{ker}}\sqrt{r_*}\leq\frac{|b_{\mathrm{IE}(\sigma)}|^2}{2\mathrm{IE}(\sigma)!}(1-\frac{(2(\mathrm{IE}(\sigma))-1)!!}{d(d+2)\cdots(d+2(\mathrm{IE}(\sigma))-2)})$, then with probability at least $1-4e^{-x}$, $\|K_{\hat\varphi_\lambda}-K_\tau\|_{L^2(\bP_X\otimes \bP_X)}\ge \frac{|b_{\mathrm{IE}(\sigma)}|^2}{2\mathrm{IE}(\sigma)!}(1-\frac{(2(\mathrm{IE}(\sigma))-1)!!}{d(d+2)\cdots(d+2(\mathrm{IE}(\sigma))-2)}).$

    \item Grant Assumption~\ref{ass:multi_index}.  If $r_*\le1$ and $C_{\rm ker}\sqrt{r_*}\le \frac{1}{2}\max_{1\le m\le M}\frac{|b_m|^2}{m!}\big\|\sum_{j=1}^M a_j^\star(\vw_j^\star)^{\otimes m}\otimes(\vw_j^\star)^{\otimes m}-\int_{S_2^{d-1}}\vw^{\otimes m}\otimes\vw^{\otimes m}\,\mathrm{d}\tau(\vw)\big\|_F$, then with the same probability, $$\|K_{\hat\varphi_\lambda}-K_\tau\|_{L^2(\bP_X\otimes \bP_X)}\ge \frac{1}{2}\max_{1\le m\le M}\frac{|b_m|^2}{m!}\Big\|\sum_{j=1}^M a_j^\star(\vw_j^\star)^{\otimes m}\otimes(\vw_j^\star)^{\otimes m}-\int_{S_2^{d-1}}\vw^{\otimes m}\otimes\vw^{\otimes m}\,\mathrm{d}\tau(\vw)\Big\|_F.$$
\end{enumerate}

\end{Corollary}

Corollary~\ref{cor:feature-kernel-movement} shows that the learned feature kernel is separated from the random-features kernel whenever the learned hidden-layer distribution has recovered the hidden multi-index structure. Consequently, any MFLD trajectory $(\varphi_t)_{t\ge0}$ initialized at $\varphi_0=\tau$ and satisfying $K_{\varphi_t}\to K_{\hat\varphi_\lambda}$ in $L^2(\bP_X\otimes\bP_X)$ is non-lazy in the sense of feature-kernel evolution. 
This conclusion complements the lazy-training viewpoint of~\cite{chizat_lazy_2019}. It is not obtained from the usual scaling argument in the training-dynamics literature, such as~\cite{yang_tensor_2021,chizat_lazy_2019}, where one studies an infinite-width limit of neural networks. Rather, it follows from the recovery of the hidden multi-index structure by the terminal hidden-layer distribution, while the initial kernel $K_\tau$ is rotationally invariant.

The next corollary combines the feature-kernel evolution with the top-$k$ signal approximation obtained from the low-temperature dimension reduction and the alignment property of the latent estimator. Its proof may be found in Section~\ref{sec:prop_feature_learning_multi}.

\begin{Corollary}[Feature-learning property]\label{coro:feature_learning_multi}
Under the assumptions of Corollary~\ref{cor:feature-kernel-movement}, let $M=1$ in the single-index case and let $M$ be the number of indices in the multi-index case. Let $\mathbf 1_{\fea}\in L^2(\hat\varphi_\lambda)$ be the constant-one function, and take $a_{\fea}=\hat a_N=\mathbf 1_{\fea}$. Let $g_{\fea},\hat g_N\in\cH_\fea$ be the corresponding elements in the learned feature space. Then $g_\fea\circ\phi_\fea=\hat g_N\circ\phi_\fea=f_{\hat\varphi_\lambda}$. Moreover, with $k=M(d+1)$, let $(\sigma_j,\ve_j)_{j\ge1}$ be the eigenvalue--eigenvector pairs of $\Sigma=\mathbb E[\phi_{\fea}(X)\otimes \phi_{\fea}(X)\mid (X_i,Y_i)_{i=1}^N]$, and take the weights $\gamma_j=\sigma_j$. Then $\hat g_N$ satisfies the $(g_{\fea},k,\delta;\varepsilon_N,\omega_N)$-alignment property with respect to these weights, for instance with $\omega\equiv0$, and $\sum_{j>k}\gamma_j\langle g_{\fea},\ve_j\rangle_{\cH_\fea}^2=o_{\mathbb P}(1)$. In the multi-index case, together with Corollary~\ref{cor:feature-kernel-movement}, for any $\lambda=o(1)$ satisfying the assumptions above, the spherical MFLD $\hat\varphi_\lambda$ satisfies the feature-learning property from Definition~\ref{def:feature_learning_general}.
\end{Corollary}

The main motivation to study spherical MFLD is to focus the analysis on the feature learning part of the algorithm, i.e. how the weights $\vw_j^\star$ are learned by MFLD and not on the estimation property of the head estimator $\hat g_N$. As a result, our choice of $g_\fea$ and $\hat g_N$ is trivially given by $a_{\fea}=\hat a_N=\mathbf 1_{\fea}$:  for $\hat \vW\sim \hat\varphi_\lambda$, we have for all $x\in\bR^d$, $g_\fea(x)= \hat g_N(x) = \inr{\hat a_N, \varphi_{\rm neur}(x)}_{L^2(\hat\varphi_\lambda)} = \bE[\sigma(\inr{\hat \vW, x})].$

\section{Summary and Outlook}\label{sec:future}

We have introduced a geometric formulation of the feature-learning property and developed a framework based on it. The organizing principle is a base--fiber mechanism: training produces a feature-side base point, the base point determines the learned fiber and the induced learned feature space, and estimation is performed through a latent estimator using this learned structure. In this formulation, feature learning is not identified merely with dynamics of parameters or kernels. It is identified with the construction of a RKHS in which most of the energy of an approximation of the regression function is supported on  a small number of top directions (the learned features) that can then be exploited by the final head estimator.

We have applied this framework to mean-field Langevin dynamics. In this case, MFLD is viewed as the Wasserstein gradient flow of a negative entropy-regularized empirical risk , or equivalently, at the pdf level, as a nonlinear Fokker--Planck equation. Its long-time limit selects the base point of the learned geometry. For spherical MFLD, this base point is the stationary hidden-layer marginal $\hat\varphi_\lambda$, which induces both the learned fiber $L^2(\hat\varphi_\lambda)$ and the learned feature space $\mathcal H_{\fea}$. This gives a geometric separation between learning features on the base and using the induced feature structure for estimation. Our main results establish two related consequences of this viewpoint.
\begin{enumerate}
\item In the low-temperature regime, the stationary measure of the random spherical nonlinear Fokker--Planck equation develops a multi-spike structure on the base in Gaussian multi-index problem. More precisely, its local barycenters around the hidden indices concentrate near the corresponding hidden index, and hence the long-time limit achieves parameters recovery. This phenomenon is not imposed by an explicit sparsity-inducing regularization; the negative entropy acts in the opposite direction. Thus the multi-spike structure is a low-temperature geometric feature of the stationary measure. In addition, we prove that, in the Gaussian single-index problem, the concentration property of this random probability measure depends on the parity of the information index of the link function.

\item This low-temperature geometry induces a dimension reduction of the mean-field model. Around the hidden indices, the effective part of the nonlinear model is governed by a first-order $O(Md)$-dimensional structure, and the main component of $f^\star$ is captured by a small number of leading directions in the learned feature space. It shows a good alignment of the signal with the feature space onto which a ridge-like head estimator can leverage to get improved convergence rates. As a statistical consequence, spherical MFLD attains the minimax optimal rates $d/N$ in the single-index case and $Md/N$ in the $M$-index case, up to logarithmic factors, for well-specified link functions. In this sense, the low-dimensional structure need not be supplied to the statistician in advance: it is formed by the Wasserstein dynamics and then used for estimation.
\end{enumerate}

One natural direction is to move from the spherical MFLD studied in this paper to the full MFLD. In the present work, only the hidden layer is trained, and we analyze the low-temperature long-time limit of the corresponding nonlinear Fokker--Planck equation on $S_2^{d-1}$. A more general setting would allow the output layer, the bias terms and the hidden layer to evolve simultaneously, and would require studying the Wasserstein gradient flow of MFLD, or equivalently its nonlinear Fokker--Planck equation, on the full parameter space $\Theta$. A basic question is then whether the low-temperature stationary measure of the full MFLD still forms a geometric structure favorable for estimation, and whether this structure still induces a dimension reduction around $f^\star$. This becomes especially important for misspecified single-index and multi-index models, where recovery of the hidden parameters alone need not imply good estimation of the regression function. In such cases, the evolution of the output layer may be essential for forming the correct learned feature space. Thus a central goal is to understand how the low-temperature geometry of the full MFLD generates learned feature spaces and leads to an effective estimation mechanism in more general low-dimensional nonparametric regression problems.


\section*{Acknowledgments}


This work was carried out during ZS's long-term visit to RIKEN–AIP, Japan. ZS gratefully acknowledges financial support from the RIKEN–AIP Overseas Student Collaboration Program. 
TW was partially supported by JSPS KAKENHI (26K21188) and RIKEN Incentive Research Project. Part of this work was completed while ZS was visiting Cornell University. ZS thanks Florentina Bunea for her warm hospitality, and acknowledges the financial support provided by the École Universitaire de Recherche `DATA EFM'.
TS was partially supported by JSPS KAKENHI (24K02905) and JST CREST (PMJCR2015). This research is supported by the National Research Foundation, Singapore and the Ministry of Digital Development and Information under the AI Visiting Professorship Programme (award number AIVP-2024-004). Any opinions, findings and conclusions or recommendations expressed in this material are those of the author(s) and do not reflect the views of National Research Foundation, Singapore and the Ministry of Digital Development and Information.

The authors also thank Francis Bach, Etienne Boursier, Florentina Bunea, Victor-Emmanuel Brunel, Ahmed El Alaoui, Anna Korba, Lucas Resende, Pierfrancesco Urbani, Martin Wainwright, and Marten Wegkamp for valuable discussions.

\appendix


\section{Mathematical and Variational Preliminaries}
This section collects the basic notation and mathematical tools used in the appendix. We also derive Hermite calculus and variational identities that are used repeatedly in the proofs.

\subsection{Notation.}
For $r,\rho>0$, define $B_{L^2(\bP_X)}(h_0;r):=\{h:\|h-h_0\|_{L^2(\bP_X)}\le r\}$ and $S_{L^2(\bP_X)}(h_0;r):=\{h:\|h-h_0\|_{L^2(\bP_X)}=r\}$. Let $\bR\bP^{d-1}=S_2^{d-1}/\{\pm 1\}$ be the real projection space. When $h_0=f_{\nu^\star}$, we also write $B_{L^2(\bP_X)}(r;\nu^\star):=\{\nu\in\cP(\Theta):\|f_\nu-f_{\nu^\star}\|_{L^2(\bP_X)}\le r\}$, and define $S_{L^2(\bP_X)}(r;\nu^\star)$ analogously. For probability measures $\mu,\nu$, let $\KL(\mu\|\nu):=\int\log(d\mu/d\nu)\,d\mu$ if $\mu\ll\nu$, and $\KL(\mu\|\nu):=+\infty$ otherwise. We write $B_2^d:=\{\boldsymbol w\in\bR^d:\|\boldsymbol w\|_2\le1\}$, $S_2^{d-1}:=\{\boldsymbol w\in\bR^d:\|\boldsymbol w\|_2=1\}$, and $WB_2^d:=\{\boldsymbol w\in\bR^d:\|\boldsymbol w\|_2\le W\}$. Let $\tau$ be the uniform probability measure on $S_2^{d-1}$, and let $\Ent_\tau^-(\varphi):=\int\log(d\varphi/d\tau)\,d\varphi$ if $\varphi\ll\tau$, and $+\infty$ otherwise. For $\boldsymbol u\in\bR^d$, $\boldsymbol u^{\otimes m}$ denotes the $m$-fold tensor product; $\operatorname{Sym}^m(\bR^d)$ is equipped with the Frobenius inner product characterized by $\langle\boldsymbol u^{\otimes m},\boldsymbol v^{\otimes m}\rangle_F=\langle\boldsymbol u,\boldsymbol v\rangle^m$, and the corresponding norm is $\|\cdot\|_F$. For a finite signed measure $\nu$ on $S_2^{d-1}$, define $\mathcal T_m(\nu):=\int_{S_2^{d-1}}\boldsymbol w^{\otimes m}\nu(d\boldsymbol w)$. For a real random variable $Z$, define its $\psi_1$-Orlicz norm by $\|Z\|_{\psi_1}:=\inf\{c>0:\mathbb E\exp(|Z|/c)\le 2\}$. We say that $Z$ is sub-exponential if $\|Z\|_{\psi_1}<\infty$. For a linear operator or matrix $A$, $\|A\|_{\rm op}$ denotes its operator norm induced by the Euclidean norm; for matrices, $\|A\|_{\rm op}:=\sup_{\|u\|_2=1}\|Au\|_2$. Here $\Delta_M:=\{\boldsymbol\lambda\in\mathbb R_+^M:\sum_{j=1}^M\lambda_j=1\}$ denotes the probability simplex.


\begin{Definition}[Covering number]
Let $(T,d)$ be a semi-metric space and let $\eta>0$. An $\eta$-net of $T$ is a subset $\mathcal N_\eta\subset T$ such that for every $t\in T$, there exists $s\in\mathcal N_\eta$ with $d(t,s)\le\eta$. The covering number $N(T,d,\eta)$ is the minimal cardinality of an $\eta$-net of $T$, and $\log N(T,d,\eta)$ is called the metric entropy.
\end{Definition}

\subsection{Hermite functions}\label{sec:Hermite_functions_Background}
Let $\bP_X$ be the standard Gaussian measure $\cN(0, I_d)$ on $\mathbb{R}^d$. We first introduce the probabilist' normalized Hermite polynomials $\{\mathrm{He}_k\}_{k \ge 0}$ corresponding to a one-dimensional standard Gaussian random variable $G \sim N(0,1)$. This sequence can be defined by the power series expansion $\exp(tz - t^2/2) = \sum_{k=0}^\infty \frac{\mathrm{He}_k(z)}{k!} t^k$. This sequence forms an orthogonal basis of $L^2(\cN(0,1))$, satisfying the orthogonality relation $\mathbb{E}[\mathrm{He}_j(G)\mathrm{He}_k(G)] = k!\mathbf{1}_{\{j=k\}}$, \cite[pp. 16]{pisier_volume_1989}.

In the multivariate case, for a multi-index $\boldsymbol{\alpha} = (\alpha_1, \dots, \alpha_d) \in \mathbb{N}^d$, define the multivariate Hermite polynomial as $\mathrm{He}_{\boldsymbol{\alpha}}(\boldsymbol{x}) = \prod_{j=1}^d \mathrm{He}_{\alpha_j}(x_j)$. Since $X\sim \cN(0,I_d)$ has independent standard Gaussian coordinates, the one-dimensional orthogonality directly yields $\mathbb{E}[\mathrm{He}_{\boldsymbol{\alpha}}(X)\mathrm{He}_{\boldsymbol{\beta}}(X)] = \boldsymbol{\alpha}!\mathbf{1}_{\{\boldsymbol{\alpha}=\boldsymbol{\beta}\}}$, where the multi-index factorial is $\boldsymbol{\alpha}! = \prod_{j=1}^d \alpha_j!$.

Define the $m$-th homogeneous subspace $C_m$ of $L_2(\bP_X)$ as the linear space spanned by all multivariate Hermite polynomials of total degree $|\boldsymbol{\alpha}| = \sum_{j=1}^d \alpha_j = m$, namely $C_m := \operatorname{span}\{\mathrm{He}_{\boldsymbol{\alpha}}(\cdot) : |\boldsymbol{\alpha}| = m\}$, whose dimension is $|\{\valpha\in\bN^d:\, |\valpha|=m\}|$. Since polynomials are dense in $L^2(\bP_X)$, and $C_m$ of different degrees are mutually orthogonal, we naturally obtain the orthogonal decomposition $L^2(\bP_X) = \bigoplus_{m=0}^\infty C_m$.

For any given unit direction vector $\boldsymbol{v} = (v_1, \dots, v_d)^\top \in S_2^{d-1}$ and $\boldsymbol{x} \in \mathbb{R}^d$, by the independence among the polynomial bases, we can derive the multivariate expansion of $\mathrm{He}_m(\langle \boldsymbol{v}, \boldsymbol{x} \rangle)$. Considering $\|\boldsymbol{v}\|_2 = 1$, we separate its generating function:
\begin{align*}
\sum_{m=0}^\infty \frac{\mathrm{He}_m(\langle \boldsymbol{v}, \boldsymbol{x} \rangle)}{m!} t^m &= \exp\left(t \langle \boldsymbol{v}, \boldsymbol{x} \rangle - \frac{t^2}{2}\right) = \prod_{j=1}^d \exp\left(t v_j x_j - \frac{(tv_j)^2}{2}\right) \\
&= \prod_{j=1}^d \left( \sum_{\alpha_j=0}^\infty \frac{\mathrm{He}_{\alpha_j}(x_j)}{\alpha_j!} (tv_j)^{\alpha_j} \right) = \sum_{m=0}^\infty \left( \sum_{|\boldsymbol{\alpha}|=m} \frac{\mathrm{He}_{\boldsymbol{\alpha}}(\boldsymbol{x})}{\boldsymbol{\alpha}!} \boldsymbol{v}^{\boldsymbol{\alpha}} \right) t^m.
\end{align*}
Comparing the coefficients of $t^m$ on both sides yields
\begin{align}\label{eq:Hermite_on_inner_product}
    \mathrm{He}_m(\langle \boldsymbol{v}, \boldsymbol{x} \rangle) = \sum_{|\boldsymbol{\alpha}|=m} \frac{m!}{\boldsymbol{\alpha}!} \boldsymbol{v}^{\boldsymbol{\alpha}} \mathrm{He}_{\boldsymbol{\alpha}}(\boldsymbol{x})
\end{align}where $\boldsymbol{v}^{\boldsymbol{\alpha}}=\prod_j v_j^{\alpha_j}$. This shows that $\mathrm{He}_m(\langle \boldsymbol{v}, \cdot \rangle) \in C_m$. In fact, these directional projections completely span the entire space $C_m$, which we prove in the following lemma~\ref{lemma:C_m_is_Span_Hermite}.

\begin{Lemma}\label{lemma:C_m_is_Span_Hermite}
Let $d \ge 2$ and $m \ge 1$. We have $C_m = \operatorname{span}\{\boldsymbol{x} \mapsto \mathrm{He}_m(\langle \boldsymbol{v}, \boldsymbol{x} \rangle) : \boldsymbol{v} \in S_2^{d-1}\}$.
\end{Lemma}
\beginproof We already know that $\operatorname{span}\{\boldsymbol{x} \mapsto \mathrm{He}_m(\langle \boldsymbol{v}, \boldsymbol{x} \rangle) : \boldsymbol{v} \in S_2^{d-1}\}\subset C_m$. Now we prove the reverse inclusion. Take any $\phi\in C_m$. Since $\operatorname{span}\{\boldsymbol{x} \mapsto \mathrm{He}_m(\langle \boldsymbol{v}, \boldsymbol{x} \rangle) : \boldsymbol{v} \in S_2^{d-1}\}$ is of finite dimension, it is closed in $L^2(\bP_X)$. Write $\phi = \phi_{\|} + \phi_\perp$, where $\phi_{\|}$ is the orthogonal projection (in the $L^2(\bP_X)$ sense) of $\phi$ onto $\operatorname{span}\{\boldsymbol{x} \mapsto \mathrm{He}_m(\langle \boldsymbol{v}, \boldsymbol{x} \rangle) : \boldsymbol{v} \in S_2^{d-1}\}$, and $\phi_\perp$ is orthogonal to $\operatorname{span}\{\boldsymbol{x} \mapsto \mathrm{He}_m(\langle \boldsymbol{v}, \boldsymbol{x} \rangle) : \boldsymbol{v} \in S_2^{d-1}\}$. Since $\phi,\phi_{\|}\in C_m$, $\phi_\perp\in C_m$ as well. Once we can prove that such $\phi_\perp=\vzero$, then necessarily $\forall \phi\in C_m$, there holds $\phi=\phi_{\|}\in \operatorname{span}\{\boldsymbol{x} \mapsto \mathrm{He}_m(\langle \boldsymbol{v}, \boldsymbol{x} \rangle) : \boldsymbol{v} \in S_2^{d-1}\}$.

Suppose $\phi_\perp(\boldsymbol{x}) = \sum_{|\boldsymbol{\alpha}|=m} c_{\boldsymbol{\alpha}} \mathrm{He}_{\boldsymbol{\alpha}}(\boldsymbol{x}) \in C_m$ is orthogonal to the set $\{\mathrm{He}_m(\langle \boldsymbol{v}, \cdot \rangle) : \boldsymbol{v} \in S_2^{d-1}\}$ in the sense of $L^2(\bP_X)$. For any $\boldsymbol{v} \in S_2^{d-1}$, using the orthogonality of multivariate Hermite polynomials, we compute the inner product:
\begin{align*}
\mathbb{E}[\phi_\perp(X)\mathrm{He}_m(\langle \boldsymbol{v}, X \rangle)] = \mathbb{E}\left[ \left(\sum_{|\boldsymbol{\beta}|=m} c_{\boldsymbol{\beta}} \mathrm{He}_{\boldsymbol{\beta}}(X)\right) \left(\sum_{|\boldsymbol{\alpha}|=m} \frac{m!}{\boldsymbol{\alpha}!} \boldsymbol{v}^{\boldsymbol{\alpha}} \mathrm{He}_{\boldsymbol{\alpha}}(X)\right) \right] = m! \sum_{|\boldsymbol{\alpha}|=m} c_{\boldsymbol{\alpha}} \boldsymbol{v}^{\boldsymbol{\alpha}} = 0.
\end{align*}
This implies that the homogeneous polynomial $\sum_{|\boldsymbol{\alpha}|=m} c_{\boldsymbol{\alpha}} \boldsymbol{v}^{\boldsymbol{\alpha}}$ is identically zero on the unit sphere $S_2^{d-1}$. By the homogeneity of multivariate polynomials, this polynomial must be identically zero on the entire $\mathbb{R}^d$, hence all combination coefficients $c_{\boldsymbol{\alpha}} = 0$. Thus $\phi_\perp = \vzero$, which completes the proof.
\endproof

Next, we provide the conditional expectation identity for $\mathrm{He}_k$ under different one-dimensional projection directions.

\begin{Lemma}\label{lemma:conditional_expectation_Hermite}
For any $\boldsymbol{v}, \boldsymbol{w} \in S_2^{d-1}$ and integer $k \ge 0$, we have
\begin{equation}
\mathbb{E}[\mathrm{He}_k(\langle \boldsymbol{v}, X \rangle) \mid \langle \boldsymbol{w}, X \rangle] = \langle \boldsymbol{v}, \boldsymbol{w} \rangle^k \mathrm{He}_k(\langle \boldsymbol{w}, X \rangle). \label{eq:hermite_regression}
\end{equation}
\end{Lemma}
\beginproof
Let $G_{\boldsymbol{w}} = \langle \boldsymbol{w}, X \rangle$, $G_{\boldsymbol{v}} = \langle \boldsymbol{v}, X \rangle$, and denote the inner product correlation coefficient as $\alpha = \langle \boldsymbol{v}, \boldsymbol{w} \rangle$.
If $\alpha = \pm 1$, then $\boldsymbol{v} = \pm \boldsymbol{w}$, and thus $G_{\boldsymbol{v}} = \pm G_{\boldsymbol{w}}$. By the parity of Hermite polynomials $\mathrm{He}_k(-x) = (-1)^k\mathrm{He}_k(x)$, we have $\mathbb{E}[\mathrm{He}_k(\pm G_{\boldsymbol{w}}) \mid G_{\boldsymbol{w}}] = (\pm 1)^k \mathrm{He}_k(G_{\boldsymbol{w}}) = \alpha^k \mathrm{He}_k(G_{\boldsymbol{w}})$, making the conclusion obvious.

Suppose $\alpha \neq \pm 1$. Since $X \sim N(0, I_d)$, the vector $(G_{\boldsymbol{w}}, G_{\boldsymbol{v}})^\top$ forms a centered joint Gaussian distribution with $\mathbb{E}[G_{\boldsymbol{w}}^2]=\mathbb{E}[G_{\boldsymbol{v}}^2]=1$ and covariance $\mathbb{E}[G_{\boldsymbol{w}} G_{\boldsymbol{v}}] = \alpha$.
We construct the random variable $Z = \frac{G_{\boldsymbol{v}} - \alpha G_{\boldsymbol{w}}}{\sqrt{1-\alpha^2}}$. We briefly compute its second moment and its covariance with $G_{\boldsymbol{w}}$: $\mathbb{E}[Z^2] = (1 - 2\alpha^2 + \alpha^2)/(1-\alpha^2) = 1$, and $\mathbb{E}[G_{\boldsymbol{w}} Z] = (\alpha - \alpha)/\sqrt{1-\alpha^2} = 0$. Since $(G_{\boldsymbol{w}}, Z)$ is jointly Gaussian, $\mathbb{E}[G_{\boldsymbol{w}} Z] = 0$ implies that $Z$ and $G_{\boldsymbol{w}}$ are independent. Then $G_{\boldsymbol{v}} = \alpha G_{\boldsymbol{w}} + \sqrt{1-\alpha^2}Z$, where $Z \sim N(0,1)$ and is independent of $G_{\boldsymbol{w}}$.

Using this decomposition, we compute the conditional expectation of the moment generating function:
\begin{align*}
\mathbb{E}\left[\exp\left(t G_{\boldsymbol{v}} - \frac{t^2}{2}\right) \mathrel{\Big|} G_{\boldsymbol{w}}\right] = \exp\left(\alpha t G_{\boldsymbol{w}} - \frac{\alpha^2 t^2}{2}\right) \mathbb{E}\left[\exp\left(t \sqrt{1-\alpha^2}Z - \frac{(1-\alpha^2)t^2}{2}\right) \mathrel{\Big|} G_{\boldsymbol{w}}\right].
\end{align*}
Since $Z$ is independent of $G_{\boldsymbol{w}}$, the conditional expectation on the right side degenerates to an unconditional expectation, which equals $1$. Writing both sides in series form and using Fubini's theorem to exchange the order of the series and the conditional expectation:
\begin{align*}
\sum_{k=0}^\infty \frac{\mathbb{E}[\mathrm{He}_k(G_{\boldsymbol{v}}) \mid G_{\boldsymbol{w}}]}{k!} t^k = \exp\left((\alpha t) G_{\boldsymbol{w}} - \frac{(\alpha t)^2}{2}\right) = \sum_{k=0}^\infty \left( \alpha^k \mathrm{He}_k(G_{\boldsymbol{w}}) \right) \frac{t^k}{k!}.
\end{align*}
Comparing the coefficients of $t^k$ on both sides directly yields Equation \eqref{eq:hermite_regression}.
\endproof

We use the following tensor notation, letting $\operatorname{Sym}^m(\mathbb{R}^d)$ be the space of order-$m$ symmetric tensors, equipped with the Frobenius inner product uniquely determined by $\langle \boldsymbol{u}^{\otimes m}, \boldsymbol{v}^{\otimes m} \rangle_F = \langle \boldsymbol{u}, \boldsymbol{v} \rangle^m$, see, for instance, \cite[pp. 137]{michalekInvitationNonlinearAlgebra2021}. Using the conditional expectation \eqref{eq:hermite_regression}, we can further reveal the isometric property between the space $C_m$ and the tensor space. Setting $k=m$ in \eqref{eq:hermite_regression}, multiplying both sides by $\mathrm{He}_m(\langle \boldsymbol{w}, X \rangle)$, and taking expectation using the law of total expectation:
\begin{align}\label{eq:inner_product_Hermite_polynomials}
    \begin{aligned}
        \mathbb{E}[\mathrm{He}_m(\langle \boldsymbol{w}, X \rangle)\mathrm{He}_m(\langle \boldsymbol{v}, X \rangle)] &= \mathbb{E}\Big[\mathrm{He}_m(\langle \boldsymbol{w}, X \rangle) \mathbb{E}[\mathrm{He}_m(\langle \boldsymbol{v}, X \rangle) \mid \langle \boldsymbol{w}, X \rangle]\Big] \\
        &= \alpha^m \mathbb{E}[\mathrm{He}_m(\langle \boldsymbol{w}, X \rangle)^2] = m! \langle \boldsymbol{w}^{\otimes m}, \boldsymbol{v}^{\otimes m} \rangle_F,
    \end{aligned}
\end{align}
where we used $\langle \boldsymbol{w}, X \rangle \sim N(0,1)$ to deduce $\mathbb{E}[\mathrm{He}_m(\langle \boldsymbol{w}, X \rangle)^2] = m!$, and $\alpha^m = \langle \boldsymbol{w}, \boldsymbol{v} \rangle^m = \langle \boldsymbol{w}^{\otimes m}, \boldsymbol{v}^{\otimes m} \rangle_F$. This core relation specifies an isometric isomorphism between the $m$-th homogeneous subspace $C_m$ and the $m$-th order symmetric tensor space $\operatorname{Sym}^m(\mathbb{R}^d)$ endowed with the Frobenius inner product. Specifically, this isomorphism is given by the following bidirectional mapping:
\begin{align}\label{eq:isomorphism_Hermite_tensors}
C_m \ni \frac{1}{\sqrt{m!}}\mathrm{He}_m(\langle \boldsymbol{v}, \cdot \rangle) \longleftrightarrow \boldsymbol{v}^{\otimes m} \in \operatorname{Sym}^m(\mathbb{R}^d), \quad \forall \boldsymbol{v} \in S_2^{d-1}
\end{align}where $C_m$ is equipped with the $L_2(\bP_X)$ inner product and $\operatorname{Sym}^m(\mathbb{R}^d)$ is equipped with its Frobenius inner product. Since these generators span the entire space and preserve their inner products, this explicitly establishes the isomorphism mapping relationship between the two Hilbert spaces.

Another result is a decomposition of the estimation error of single-index problem in terms of Hermite polynomials, that is, for $f_\varphi(\vx) = \int \sigma(\langle\vw,\vx\rangle)\mathrm{d}\varphi(\vw)$ and $f^\star(\vx) = \sigma(\langle\vw_\star,\vx\rangle)$. By Hermite expansion, $\sigma(\langle\vw,\vx\rangle) = \sum_{k\geq 0}\frac{b_k}{k!}\mathrm{He}_k(\langle\vw,\vx\rangle)$, where $b_k = \bE[\sigma(G)\mathrm{He}_k(G)]$, and $G\sim\cN(0,1)$. Then $f_\varphi(\vx) - f^\star(\vx) = \sum_{k\geq 1}\frac{b_k}{k!}\int \mathrm{He}_k(\langle\vw,\vx\rangle)\varphi(\mathrm{d}\vw) - \mathrm{He}_k(\langle\vw_\star,\vx\rangle)$. Since different $C_k$ are orthogonal in $L^2(\bP_X)$, $\|f_\varphi(\vx) - f^\star(\vx)\|_{L^2(\bP_X)}^2 = \sum_{k\geq 1}\frac{b_k^2}{(k!)^2}\|\int \mathrm{He}_k(\langle\vw,\vx\rangle)\varphi(\mathrm{d}\vw) - \mathrm{He}_k(\langle\vw_\star,\vx\rangle)\|_{L^2(\bP_X)}^2$. Moreover, by \eqref{eq:isomorphism_Hermite_tensors}, $\|\int \mathrm{He}_k(\langle\vw,\vx\rangle)\varphi(\mathrm{d}\vw) - \mathrm{He}_k(\langle\vw_\star,\vx\rangle)\|_{L^2(\bP_X)}^2 = k!\| \int \vw^{\otimes k}\varphi(\mathrm{d}\vw) - \vw_\star^{\otimes k} \|_F^2$, and consequently,
\begin{align}\label{eq:Hermite_expansion_estimation_error}
    \|f_\varphi - f^\star\|_{L^2(\bP_X)}^2 = \sum_{k\geq 1}\frac{b_k^2}{k!}\bigg\| \int \vw^{\otimes k}\varphi(\mathrm{d}\vw) - \vw_\star^{\otimes k} \bigg\|_F^2.
\end{align}

\subsection{The variational solution}\label{sec:variational_solution}

\begin{Lemma}\label{lemma:variation}
Let $q$ be a probability measure on $[-A, A]$ that admits a strictly positive pdf $q(a)$ with respect to the Lebesgue measure almost everywhere. For any $u \in (-A, A)$, define
\begin{equation}\label{eq:def_psi}
    \psi(u) = \inf \left( \int_{-A}^{A} p(a) \log \frac{p(a)}{q(a)} \,\mathrm{d}a : p \ge 0, \int_{-A}^{A} p(a) \,\mathrm{d}a = 1, \int_{-A}^{A} a p(a) \,\mathrm{d}a = u \right),
\end{equation}
and $\psi(u) = \infty$ if $u \notin (-A, A)$. Then the following hold:
\begin{enumerate}
    \item Let $Z : t \in \mathbb{R} \mapsto Z(t) = \int_{-A}^{A} q(a) \exp(ta) \,\mathrm{d}a$ be the partition function, and let $F(t) = \log(Z(t))$ be the free energy function. Then $F$ is convex, and for any $u \in (-A, A)$, there exists a unique $t(u) \in \mathbb{R}$ such that
    \begin{equation*}
        \psi(u) = \sup_{t \in \mathbb{R}} \{tu - F(t)\} = t(u)u - F(t(u)).
    \end{equation*}In particular, $\psi$ is convex, differentiable and $\psi'(u)=t(u)$.
    \item For any $u \in (-A, A)$, $\psi(u)$ is achieved in \eqref{eq:def_psi}  at a unique pdf $p_{t(u)}$ up to a null set, where for any $t \in \mathbb{R}$,
    \begin{equation*}
        p_t : a \in (-A, A) \mapsto p_t(a) = \frac{q(a)\exp(ta)}{Z(t)}\1_{[-A, A]}(a).
    \end{equation*}
    \item  The Bregman divergence of $\psi$: $D_{\psi}:(u, v)\in(-A, A)^2 \to \psi(u) - \psi(v) - \psi'(v)(u - v)$ satisfies: for any $u, v \in (-A, A)$,
    \begin{equation*}
        D_{\psi}(u, v) \ge \frac{1}{2A^2}(u - v)^2.
    \end{equation*}
\end{enumerate}
\end{Lemma}

\beginproof We first show that for any $u\in(-A, A)$ there exists a unique $t(u)\in\bR$ such that $\int_{-A}^A a p_{t(u)}(a) \,\mathrm{d}a = u$. By definition, $F:t\in\bR \to \log Z(t)=\log \bE \exp(tU)$ for $U\sim q$. Its first and second derivatives are given by: for all $t\in\bR$,
    \begin{align*}
        F'(t) &= \frac{Z'(t)}{Z(t)} = \frac{\int_{-A}^{A} a q(a) \exp(ta) \,\mathrm{d}a}{\int_{-A}^{A} q(a) \exp(ta) \,\mathrm{d}a} = \mathbb{E}_{p_t}[a],
        \end{align*}and
        \begin{align*}
        F''(t) &= \frac{Z''(t)Z(t) - (Z'(t))^2}{Z(t)^2} = \mathbb{E}_{p_t}[a^2] - (\mathbb{E}_{p_t}[a])^2 = \mathrm{Var}_{p_t}[a].
    \end{align*}
    Since $q(a) > 0$ almost everywhere, the support of $p_t$ is $[-A, A]$ for all $t \in \mathbb{R}$, implying that $p_t$ is not a Dirac measure. Consequently, $F''(t) = \mathrm{Var}_{p_t}[a] > 0$ for all $t \in \mathbb{R}$, which entails that $F'(t)$ is strictly monotonically increasing. 
    Since for any $t>0$, $a\mapsto \exp(ta)$ is increasing on $[-A,A]$,  there hold $Z(t)=\int_{-A}^Aq(a)\exp(ta)\mathrm{d}a\leq \exp(tA)\int_{-A}^Aq(a)\mathrm{d}a=\exp(tA)$ and, for any $0<\varepsilon<A$, $Z(t)\geq \int_{A-\varepsilon}^Aq(a)\exp(ta)\mathrm{d}a\geq \exp(t(A-\varepsilon))\int_{A-\varepsilon}^Aq(a)\mathrm{d}a$. Therefore for any $t>0$ and $0<\varepsilon<A$, $\exp(t(A-\varepsilon))\int_{A-\varepsilon}^Aq(a)\mathrm{d}a\leq Z(t)\leq \exp(tA)$. Taking logarithmic and dividing by $t$, we obtain $A-\varepsilon+\frac{1}{t}\log\int_{A-\varepsilon}^Aq(a)\mathrm{d}a\leq \frac{F(t)}{t}\leq A$, and consequently $\lim_{t\to\infty}\frac{F(t)}{t}=A$. Similarly, $\lim_{t \to -\infty} F'(t) = -A$. Since $F'$ is continuous and strictly increasing, the intermediate value theorem guarantees that for any $u \in (-A, A)$, there exists a unique $t(u) \in \mathbb{R}$ such that $F'(t(u)) = u$, which means that for any $u \in (-A, A)$, there exists a unique $t(u) \in \mathbb{R}$ such that $\mathbb{E}_{p_{t(u)}}[a] = u$.

\begin{enumerate}
    \item We first prove item~\emph{2.} Let $u\in(-A, A)$. Let $t(u) \in \mathbb{R}$ be the unique solution to $\mathbb{E}_{p_{t(u)}}[a] = u$.  Let us show that $p_{t(u)}$ is the unique global minimizer of the optimization problem \eqref{eq:def_psi} defining $\psi(u)$. Let $p$ be any other pdf satisfying the feasible constraints of \eqref{eq:def_psi}. We evaluate the difference in the objective functional: by $p_t(a) = \frac{q(a)\exp(ta)}{Z(t)} \1_{[-A, A]}(a)$, there holds $\log\frac{p_{t(u)}(a)}{q(a)}  = t(u)a-\log Z(t(u))$, hence
    \begin{align*}
    &\int_{-A}^{A} p(a) \log \frac{p(a)}{q(a)} \,\mathrm{d}a - \int_{-A}^{A} p_{t(u)}(a) \log \frac{p_{t(u)}(a)}{q(a)} \,\mathrm{d}a \\
    &= \int_{-A}^{A} p(a) \log \frac{p(a)}{p_{t(u)}(a)} \,\mathrm{d}a + \int_{-A}^{A} (p(a) - p_{t(u)}(a)) \log \frac{p_{t(u)}(a)}{q(a)} \,\mathrm{d}a \\
    &= \mathrm{KL}(p \| p_{t(u)}) + \int_{-A}^{A} (p(a) - p_{t(u)}(a)) (t(u)a - \log Z(t(u))) \,\mathrm{d}a \\
    &= \mathrm{KL}(p \| p_{t(u)}) + t(u) \bigg(\int_{-A}^{A} a p(a) \,\mathrm{d}a - \int_{-A}^{A} a p_{t(u)}(a) \,\mathrm{d}a\bigg)\\
    &\quad - \log Z(t(u)) \bigg(\int_{-A}^{A} p(a) \,\mathrm{d}a - \int_{-A}^{A} p_{t(u)}(a) \,\mathrm{d}a\bigg) \\
    &= \mathrm{KL}(p \| p_{t(u)}) + t(u)(u - u) - \log Z(t(u))(1 - 1) = \mathrm{KL}(p \| p_{t(u)}) \ge 0,
    \end{align*}with equality holding if and only if $p = p_{t(u)}$ almost everywhere. Thus, $p_{t(u)}$ is the unique global minimizer.

    \item We prove item~\emph{1.}  Substituting the optimal pdf $p_{t(u)}$ into the variational objective gives:
    \begin{equation*}
        \psi(u) = \int_{-A}^{A} p_{t(u)}(a) (t(u)a - \log Z(t(u))) \,\mathrm{d}a = t(u)u - F(t(u)).
    \end{equation*}
    To demonstrate strong duality, consider the function $g(t) = tu - F(t)$. Its derivative is $g'(t) = u - F'(t)$. The critical point condition $g'(t) = 0$ is exactly $F'(t) = u$, which is uniquely satisfied by $t = t(u)$. Furthermore, $g''(t) = -F''(t) < 0$, making $g(t)$ strictly concave. Thus, the supremum is uniquely attained at $t(u)$, establishing that $\psi(u) = \sup_{t \in \mathbb{R}} \{tu - F(t)\} = t(u)u - F(t(u))$, which identifies $\psi$ as the Fenchel-Legendre conjugate of $F$, in particular, it is convex. Moreover, since $\sup_{t \in \mathbb{R}} \{tu - F(t)\}$ has a unique solution given by $t(u)$ by Danskin theorem, $\psi$ is differentiable at any point $u$ and $\psi'(u)= t(u)$. 

    \item We prove item~\emph{3.} We deduce from Danskin theorem that $\psi'(u)=t(u)$. It can be also directly proved: from the duality identity $\psi(u) = t(u)u - F(t(u))$ and the relation $F'(t(u)) = u$, we differentiate $\psi(u)$ with respect to $u$ utilizing the chain rule:
    \begin{equation*}
        \psi'(u) = t(u) + t'(u)u - F'(t(u))t'(u) = t(u) + t'(u)u - u t'(u) = t(u).
    \end{equation*}
    Moreover, taking derivatives with respect to $u$ on both sides of $F'(t(u)) = u$ yields $t'(u) = \frac{1}{F''(t(u))}$, thus
    \begin{equation*}
        \psi''(u) = t'(u) = \frac{1}{F''(t(u))} > 0.
    \end{equation*}
    Recalling that $F''(t(u)) = \mathrm{Var}_{p_{t(u)}}[a]$, and noting that the random variable $a$ is almost surely bounded in $[-A, A]$, we can bound the variance:
    \begin{equation*}
        \mathrm{Var}_{p_{t(u)}}[a] = \mathbb{E}_{p_{t(u)}}[a^2] - u^2 \le A^2 - u^2 \le A^2.
    \end{equation*}
    Consequently, we obtain a uniform lower bound on the curvature of $\psi$: $\psi''(u) = \frac{1}{\mathrm{Var}_{p_{t(u)}}[a]} \ge \frac{1}{A^2}$ for all $u \in (-A, A)$.

    Finally, for any $u, v \in (-A, A)$, applying Taylor's theorem with the Lagrange remainder guarantees the existence of some $\xi$ strictly between $u$ and $v$ such that:
    \begin{equation*}
        \psi(u) = \psi(v) + \psi'(v)(u - v) + \frac{1}{2}\psi''(\xi)(u - v)^2.
    \end{equation*}
    Rearranging this expression immediately yields the required lower bound for the Bregman divergence: for all $u,v\in(-A, A)$,
    \begin{equation*}
        D_{\psi}(u, v) = \psi(u) - \psi(v) - \psi'(v)(u - v) = \frac{1}{2}\psi''(\xi)(u - v)^2 \ge \frac{1}{2A^2}(u - v)^2. \qedhere
    \end{equation*}
\end{enumerate}
\endproof

\begin{Remark}\label{eq:rem_psi_zero} 
When $q$ is symetric, for instance when $q:a\to (2A)^{-1}\1_{[-A, A]}(a)$ is the pdf of the uniform distribution over $[-A, A]$, then $Z(0)=1$, $p_0=q$ and $\int_{-A}^A a p_0(a)da=0$. As a consequence, by unicity of $t(u)$ for all $u\in(-A, A)$, we deduce that $t(0)=0$ and $\psi'(0)=t(0)=0$. 
\end{Remark}

\subsection{Chain rule of negative Shannon entropy}\label{sub:proof_chain_rule_negative_Shannon_entropy}

Let $Q$ be a probability measure on $\Theta = [-A, A] \times WB_2^d$ with a tensor product structure $Q(\,\mathrm{d}a, \,\mathrm{d}\boldsymbol{w}) = q(a)\,\mathrm{d}a \otimes Q_W(\,\mathrm{d}\boldsymbol{w})$, where $q(a)$ is a strictly positive pdf on $[-A, A]$ almost everywhere. For $\nu\in\cP(\Theta)$, we define the divergence functional $\mathcal{D}(\nu \| Q) := \int \nu \log(\nu/Q)$.
We prove the following identity:
\begin{align}\tag{\ref{eq:chain_rule_Shannon_entropy}}
    \mathcal{D}(\nu \| Q) = \mathcal{D}(\varphi \| Q_W) + \int_{WB_2^d} \left( \int_{-A}^A \nu(a|\boldsymbol{w}) \log \frac{\nu(a|\boldsymbol{w})}{q(a)} \,\mathrm{d}a \right) \,\mathrm{d}\varphi(\boldsymbol{w})
\end{align}where $\varphi$ is the marginal of $\nu$ on $W B_2^d$. Denote $\rho = \frac{\mathrm{d}\nu}{\mathrm{d}Q}$ as the Radon-Nikodym derivative of $\nu$ with respect to the reference measure $Q$. Since $Q(\,\mathrm{d}a, \,\mathrm{d}\boldsymbol{w}) = q(a)\,\mathrm{d}a \otimes Q_W(\,\mathrm{d}\boldsymbol{w})$, we can decompose the pdf as $\rho(a,\boldsymbol{w})=\rho(\boldsymbol{w})\rho(a|\boldsymbol{w})$, where $\rho(\boldsymbol{w}) = \frac{\mathrm{d}\varphi}{\mathrm{d}Q_W}(\boldsymbol{w})$ is the pdf of the marginal $\varphi$ with respect to $Q_W$, and $\rho(a|\boldsymbol{w})$ is the pdf of the conditional measure $\nu(\cdot|\boldsymbol{w})$ with respect to $q(a)\,\mathrm{d}a$, where $\nu(\mathrm{d}a|\vw)=\nu(a|\vw)\mathrm{d}a=\rho(a|\vw)q(a)\mathrm{d}a$. As a result,
\begin{align*}
    \mathcal{D}(\nu \| Q) &= \int_{WB_2^d}\int_{-A}^A \rho(\boldsymbol{w})\rho(a|\boldsymbol{w})\left(\log\rho(\boldsymbol{w})+\log\rho(a|\boldsymbol{w})\right) q(a)\,\mathrm{d}a\,\mathrm{d}Q_W(\boldsymbol{w})\\
    &=\int_{WB_2^d}\rho(\boldsymbol{w})\log\rho(\boldsymbol{w})\left(\int_{-A}^A\rho(a|\boldsymbol{w})q(a)\,\mathrm{d}a\right)\,\mathrm{d}Q_W(\boldsymbol{w}) \\
    &\quad + \int_{WB_2^d}\left(\int_{-A}^A \rho(a|\boldsymbol{w})\log\rho(a|\boldsymbol{w})q(a)\,\mathrm{d}a\right)\rho(\boldsymbol{w})\,\mathrm{d}Q_W(\boldsymbol{w}).
\end{align*}
Since $\int_{-A}^A \rho(a|\boldsymbol{w})q(a)\,\mathrm{d}a=1$, $\int_{-A}^A \rho(a|\boldsymbol{w})\log\rho(a|\boldsymbol{w})q(a)\,\mathrm{d}a = \int_{-A}^A \nu(a|\boldsymbol{w}) \log \frac{\nu(a|\boldsymbol{w})}{q(a)} \,\mathrm{d}a$, and $\,\mathrm{d}\varphi(\boldsymbol{w}) = \rho(\boldsymbol{w})\,\mathrm{d}Q_W(\boldsymbol{w})$, we obtain the chain rule of relative entropy, that is, \eqref{eq:chain_rule_Shannon_entropy}.

\section{The Feature Learning Property of the LASSO}

In this section, we show that the LASSO \cite{tibshirani1996regression} satisfies the Feature Learning Property introduced in Definition~\ref{def:feature_learning_general}. Our objective is not to re-derive the standard $\mathcal{O}(s\log(d)/N)$ convergence rate of the LASSO via its polar decomposition under a feature learning perspective. Rather, our goal is to demonstrate that feature learning in the LASSO can be naturally identified with its support recovery capability. Specifically, under the assumption of exact support recovery, we prove that the LASSO satisfies the Feature Learning Property. However, because Definition~\ref{def:feature_learning_general} extends beyond feature learning per se (Items~3 and 4) to encompass both feature evolution (Item~1) and feature utilization by the linear head (Items~2 and 5), additional evolution and estimation conditions are required to verify all components of Definition~\ref{def:feature_learning_general}.

Let $(\boldsymbol{e}_1,\cdots,\boldsymbol{e}_d)$ be the canonical basis of $\bR^d$. Let $X$ be an isotropic random vector, that is, $\bE X=0$ and $\Sigma = \bE[XX^\top]=I_d$. Let $Y=\langle X,\boldsymbol{\beta}^\star\rangle + \xi$, where $\boldsymbol{\beta}^\star\in\bR^d$ -- so that $f^\star:\vx\mapsto \langle\vx,\bbeta^\star\rangle$ --  and denote $S^\star = \supp(\boldsymbol{\beta}^\star)$ and $s=|S^\star|$. Let $\hat{\boldsymbol{\beta}}\in\argmin(P_N\ell_{\boldsymbol{\beta}}+\lambda\|\boldsymbol{\beta}\|_1)$, where $\ell$ is the squared loss, and $\lambda\geq 0$ is some tuning parameter. The LASSO with parameter $\lambda$ is defined by $\hat f_N:\vx\mapsto \langle\vx,\hat\bbeta\rangle$. 

Let $\hat S = \supp(\hat{\boldsymbol{\beta}})$ be the support of the LASSO. From feature learning perspective, it is natural to look at $\hat S$ as the set of the coordinates of $\bR^d$ that the LASSO decides to be relevant for the prediction task: they are the '\textit{features learned}', even though, no features have been actually constructed, they have been only selected from a pre-existing (large) set of features $\boldsymbol{e}_1,\cdots,\boldsymbol{e}_d$. We say that support recovery holds when $\hat S= S^\star$. A necessary and sufficient condition for support recovery of the LASSO has been identified by \cite{zhao2006model}.    

Let us now identity the feature space $\cH_\fea$ through which we factorize the LASSO (and $f^\star$) in order to prove LASSO's  feature learning property. The most natural choice is to consider the space spanned by the features selected by the LASSO, that is $\cH_\fea = {\rm span}(\ve_j:j\in\hat S) = \bR^{\hat S}$. By doing so, we take $(\vx_i,y_i)_{i=1}^N\mapsto K_{\fea}(x,x'):=\langle x_{\hat S},x'_{\hat S}\rangle$ as the generating rule of the reproducing kernel, its canonical feature map $\phi_\fea:\vx\in\bR^d\to \vx_{\hat S}\in\bR^{\hat S}$ and $\Sigma_\fea = I_{\bR^{\hat S}}$ is the identity operator on $\bR^{\hat S}$. The head estimator is $\hat g_N:\vx\in\bR^{\hat S}\to \inr{\hat\vbeta_{\hat S}, \vx}$ and, for this choice, Figure~\ref{fig:estimator_factorization}'s factorization holds for the LASSO: for all $\vx\in\bR^d$,
\begin{equation*}
  \hat f_N(\vx) = \inr{\vx, \hat\vbeta} = \inr{\vx_{\hat S}, \hat\vbeta_{\hat S}} = \hat g_N(\vx_{\hat S}) =   \hat g_N(\phi_\fea(\vx)).
\end{equation*}In particular, \textit{Item~2} from Definition~\ref{def:feature_learning_general} holds.
Regarding $f^\star$'s factorization via $\cH_\fea$ as depicted in Figure~\ref{fig:oracle_factorization}, we take $g_\fea$ to be the oracle of $f^\star$ in $\cH_\fea$, i.e. the closest element in $\cH_\fea$ to $f^\star$ w.r.t. the $L^2(\bP_X)$ metric: $g_\fea:\vx\in\bR^{\hat S}\to \inr{\vbeta^\star_{\hat S}, \vx}$. We have 
\begin{equation*}
    \norm{f^\star - g_\fea(\phi_\fea)}_{L^2(\bP_X)}^2 = \norm{\vbeta^\star_{\hat S \Delta S}} 
\end{equation*}where $\hat S \Delta S = (\hat S \cup S)\backslash (\hat S \cap S)$ is the symmetric difference between  $\hat S$ and $S$. \textit{Item~3} from Definition~\ref{def:feature_learning_general} requires this approximation term to tend to zero. Under support recovery, this term equals to zero, and so $f^\star$'s factorization is exact in that case -- support recovery is a sufficient condition (even though not necessary) to guarantee \textit{Item~3} and \textit{Item~4}  from Definition~\ref{def:feature_learning_general}.

Regarding feature evolution from \textit{Item~1} in Definition~\ref{def:feature_learning_general}, an algorithm like LARS \cite{efron2004least} starts from $\boldsymbol{\beta}^{(0)}=\boldsymbol{0}$ and so $K_0\equiv 0$, and $\cH_0=\{\boldsymbol{0}\}$. In order to fulfill \textit{Item~1}, we need $\cH_\fea=\bR^{\hat S}$ be go away from $\cH_0=\{\boldsymbol{0}\}$. This requires $\hat S\neq \emptyset$ and so, under support recovery, $\vbeta^\star\neq 0$. Hence,  feature evolution from \textit{Item~1} requires $\vbeta^\star\neq 0$. Indeed, let $X,X'\sim\bP_X$ with $X$ independent of $X'$. Then 
$$\|K_\fea - K_0\|_{L^2(\bP_X\otimes\bP_X)}^2 = \|K_{\fea}\|_{L^2(\bP_X\otimes\bP_X)}^2 = \bE[ \langle X_{\hat S},X_{\hat S}'\rangle^2 |\hat S] = \sum_{i,j\in\hat S}(\bE[x_ix_j])^2 = |\hat S|$$ because $\bE[x_j x_i]=\1_{i=j}$ for all $i,j\in[d]$.  In the case $\vbeta^\star= 0$, starting from $\cH_0=\{\boldsymbol{0}\}$ is the best possible choice, since $\cH_0$ contains a unique element which is the signal itself. Hence, there is nothing more to learn since $\cH_0$ perfectly identifies the best possible space approaching $f^\star$. 

Finally, under support recovery, there is no tail energy of $\vbeta^\star$ in the feature space and so  \textit{Item~5} follows from the consistency of the LASSO. As we said previously, this item is not strictly speaking about feature learning than their utilization. That is why Definition~\ref{def:feature_learning_general} contains a notion of statistical efficiency through its \textit{Item~5} that is implied by the consistency of the estimator studied here.

To conclude, the LASSO satisfies the feature learning property as introduced in Definition~\ref{def:feature_learning_general} under support recovery, consistency and for $\vbeta^\star\neq 0$ (to allow for feature evolution for an algorithm implementing the LASSO and starting at $\boldsymbol{\beta}^{(0)}=\boldsymbol{0}$). LASSO's polar decomposition and $f^\star$'s factorization go through the learned feature space $\bR^{\hat S}$ with the learned features  $\{\ve_j:j\in\hat S\}$.

\section{Self Regularization (implicit bias) of MFLD}\label{app:mfld-foundations}

In this section, we prove a self-regularization property of MFLD on the output layer. The main point is that, once the hidden layer distribution achieves stationarity, the output layer solves a strongly convex regularized problem: the head estimator $\hat g_N$ is a RERM with strongly convex regularization function.

Throughout this section, $L^2(\hat\varphi_\lambda)$ denotes the learned coefficient Hilbert space, and define the prediction operator
\[
T_{\fea}:L^2(\hat\varphi_\lambda)\to L^2(\mathbb P_X),\qquad (T_{\fea}a)(x):=\langle a,\varphi_{\neu}(x)\rangle_{L^2(\hat\varphi_\lambda)}
\]where $\varphi_{\neu}(x)(w)=\sigma(\inr{x,w})$.
By Remark~\ref{remark:scaling}, we take the output-weight reference pdf to be the uniform pdf $q(\alpha)=1/(2A)$ on $[-A,A]$ in this section. This does not change the minimizer.

\begin{Proposition}\label{prop:hat_g_N}
Define $\psi(u):=\inf\{\mathcal D(\rho\|q):\rho\in\mathcal P([-A,A]), \int \alpha\,\mathrm{d}\rho(\alpha)=u\}$ for $u\in(-A,A)$, and set $\psi(u)=+\infty$ otherwise. For all $\lambda>0$, define the extended-valued convex functional $\Psi:{L^2(\hat\varphi_\lambda)}\to\mathbb R\cup\{+\infty\}$ by $\Psi(a):=\int\psi(a(\boldsymbol w))\,\mathrm{d}\hat\varphi_\lambda(\boldsymbol w)$, with the convention that $\Psi(a)=+\infty$ if $a(\boldsymbol w)\notin(-A,A)$ on a set of positive $\hat\varphi_\lambda$-measure. Then, $\mathbb P^{\otimes N}$-almost surely, $\hat a_N \in{L^2(\hat\varphi_\lambda)}$ satisfies
\[
\forall\lambda>0,\quad\hat a_N\in\arg\min_{a\in{L^2(\hat\varphi_\lambda)}}\bigg\{\frac1N\sum_{i=1}^N(Y_i-(T_{\fea}a)(X_i))^2+\lambda\Psi(a)\bigg\}.
\]
Moreover, for any $a_1,a_2\in\operatorname{dom}(\Psi)$ and any $\zeta_2\in\partial\Psi(a_2)$, $\Psi(a_1)-\Psi(a_2)-\langle\zeta_2,a_1-a_2\rangle_{{L^2(\hat\varphi_\lambda)}}\ge (2A^2)^{-1}\|a_1-a_2\|_{{L^2(\hat\varphi_\lambda)}}^2$. Consequently, for
\begin{align*}
    \Psi_{\cH_\fea}:g\in\cH_\fea\mapsto \inf\bigg\{\Psi(a):\, a\in L^2(\hat\varphi_\lambda),\, g\circ\phi_\fea(\cdot) = \langle a ,\varphi_{\neu}(\cdot)\rangle_{L^2(\hat\varphi_\lambda)}\bigg\},
\end{align*}there hold
\begin{align*}
    \forall\lambda>0,\quad \hat g_N \in \argmin_{g\in\cH_\fea}\left\{ \frac{1}{N}\sum_{i=1}^N\left(Y_i - g\circ\phi_\fea(X_i)\right)^2 + \lambda\Psi_{\cH_\fea}(g) \right\},
\end{align*}and for any $g_1,g_2\in\mathrm{dom}(\Psi_{\cH_\fea})$ and any $\eta_2\in\partial\Psi_{\cH_\fea}(g_2)$,
\begin{align*}
    \Psi_{\cH_\fea}(g_1) - \Psi_{\cH_\fea}(g_2) - \langle\eta_2,g_1-g_2\rangle_{\cH_\fea} \geq \frac{1}{2A^2}\norm{g_1-g_2}_{\cH_\fea}^2.
\end{align*}
\end{Proposition}
\beginproof
Let $Q$ be the probability measure on $\Theta = [-A, A] \times WB_2^d$ with a tensor product structure $Q(\,\mathrm{d}\alpha, \,\mathrm{d}\boldsymbol{w}) = q(\alpha)\,\mathrm{d}\alpha \otimes Q_W(\,\mathrm{d}\boldsymbol{w})$, where $q(\alpha) = \frac{1}{2A}$ for all $-A\leq \alpha\leq A$ and $Q_W$ is the uniform distribution over $WB_2^d$. In that case, $Q=\alpha {\rm Leb}$ where $\alpha=(2AW^d {\rm Vol}(B_2^d))^{-1}$ and so,  $\mathcal{D}(\nu \| Q) = \int \nu \log(\nu/Q)=\Ent^-(\nu) + c$ where $c$ is an absolute constant that is independent of $\nu$. Hence, according to Remark~\ref{remark:scaling}, one can replace $\Ent^-(\nu)$ by $\mathcal{D}(\nu \| Q)$ in \eqref{eq:def_hat_nu}  while keeping unchanged the minimizer $\hat\nu_\lambda$. Since $\lambda>0$, we only consider $\nu\in\cP(\Theta)$ such that $\Ent^-(\nu)<\infty$, that is, $\nu\in\cP_{\mathrm{ac}}(\Theta)$.

For any $\nu\in\cP_{\mathrm{ac}}(\Theta)$, let $\varphi:=(P_W)_\sharp\nu$ and  $\nu(\mathrm{d}a|\vw):=\nu(a|\boldsymbol{W} = \vw)\mathrm{d}a=\rho(a|\boldsymbol{W} = \vw)q(a)\mathrm{d}a$ as the conditional probability measure. By the chain rule of negative Shannon entropy (see Section~\ref{sub:proof_chain_rule_negative_Shannon_entropy}):
\begin{align}\label{eq:chain_rule_Shannon_entropy}
\mathcal{D}(\nu \| Q) = \mathcal{D}(\varphi \| Q_W) + \int_{WB_2^d}\bigg(\int_{-A}^A \nu(\alpha|\boldsymbol{w}) \log \frac{\nu(\alpha|\boldsymbol{w})}{q(\alpha)}\,\mathrm{d}\alpha\bigg)\,\mathrm{d}\varphi(\boldsymbol{w}).
\end{align}
Denote $a:\boldsymbol{w}\in WB_2^d\mapsto \int_{-A}^A \alpha\,\mathrm{d}\nu(\alpha|\boldsymbol{w})=\bE_{\nu}[\alpha|\boldsymbol{W} =\boldsymbol{w}]$ where $(\alpha, \boldsymbol{W})\sim \nu$. Since $\nu\in\cP_{\mathrm{ac}}(\Theta)$, $a(\vw)\in(-A,A)$, $\varphi$-a.s.. Then
\begin{align*}
P_N\ell_\nu &= \frac{1}{N}\sum_{i=1}^N\bigg( Y_i - \int_{WB_2^d}\sigma(\langle\boldsymbol{w},X_i\rangle)\bigg[\int_{-A}^A \alpha\,\mathrm{d}\nu(\alpha|\boldsymbol{W} = \boldsymbol{w})\bigg] \,\mathrm{d}\varphi(\boldsymbol{w}) \bigg)^2 \\
&= \frac{1}{N}\sum_{i=1}^N\bigg(Y_i - \int_{WB_2^d} \sigma(\langle\boldsymbol{w},X_i\rangle)a(\boldsymbol{w})\,\mathrm{d}\varphi(\boldsymbol{w}) \bigg)^2,
\end{align*}
which depends only on $(\varphi,a)$, that is, the base point $\varphi$ and the element $a$ in its fiber $L^2(\varphi)$. Hence, we write $P_N\ell_{\varphi,a}$ instead of $P_N\ell_\nu$ in the following: $P_N \ell_{\varphi, a} = \frac{1}{N} \sum_{i=1}^N (Y_i - \bE_{\boldsymbol{W}\sim\varphi}[a(\boldsymbol{W}) \sigma(\langle \boldsymbol{W}, X_i \rangle)])^2$, which implies that the empirical risk depends only on the marginal distribution $\varphi$ and the conditional expectation $a$. Plug \eqref{eq:chain_rule_Shannon_entropy} back to \eqref{eq:def_hat_nu}, we obtain
\begin{align*}
P_N\ell_\nu^\lambda = P_N\ell_{\varphi,a} + \lambda\mathcal{D}(\varphi \| Q_W) + \lambda\int_{WB_2^d}\bigg( \int_{-A}^A \nu(\alpha|\boldsymbol{w}) \log \frac{\nu(\alpha|\boldsymbol{w})}{q(\alpha)}\,\mathrm{d}\alpha \bigg)\,\mathrm{d}\varphi(\boldsymbol{w}).
\end{align*}
The key point here is that the minimization of the regularized empirical risk over the feasible set $\nu\in\cP(\Theta)$ can be decomposed into a two-layer optimization process involving $(\varphi,a)$ and $\nu(\cdot|\boldsymbol{w})$ (the conditional distribution of $\alpha|\boldsymbol{W}=\vw$). Note that $P_N\ell_{\varphi,a}$ and $\lambda\mathcal{D}(\varphi \| Q_W)$ both depend only on $(\varphi,a)$ and not on $\nu(\cdot|\boldsymbol{w})$. Therefore,
\begin{align*}
\min_{\nu\in\cP(\Theta)}P_N\ell_\nu^\lambda &= \min_{(\varphi,a)}\bigg[ P_N\ell_{\varphi,a} + \lambda\mathcal{D}(\varphi \| Q_W) + \lambda\min_{\nu(\cdot|\cdot)\in \cK(\varphi,a)}\int_{WB_2^d}\bigg( \int_{-A}^A \nu(\alpha|\boldsymbol{w}) \log \frac{\nu(\alpha|\boldsymbol{w})}{q(\alpha)}\,\mathrm{d}\alpha \bigg)\,\mathrm{d}\varphi(\boldsymbol{w}) \bigg]
\end{align*}
where
\begin{align*}
\cK(\varphi,a) := \bigg\{\nu(\cdot|\cdot):\, \mbox{for }\varphi-\mathrm{a.s. }\,\boldsymbol{w}, \nu(\cdot|\boldsymbol{w})\in\cP([-A,A]),\mbox{ and } \int_{-A}^A \alpha\,\mathrm{d}\nu(\alpha|\boldsymbol{w})=a(\boldsymbol{w}) \bigg\}.
\end{align*}

We first prove that we can swap the inner minimum and the integral. Once we prove that, then for any $(\varphi,a)$, there holds
\begin{align*}
\min_{\nu(\cdot|\cdot)\in\cK(\varphi,a)}\int_{WB_2^d}\bigg( \int_{-A}^A \nu(\alpha|\boldsymbol{w}) \log \frac{\nu(\alpha|\boldsymbol{w})}{q(\alpha)}\,\mathrm{d}\alpha \bigg)\,\mathrm{d}\varphi(\boldsymbol{w}) = \int_{WB_2^d}\psi(a(\boldsymbol{w}))\,\mathrm{d}\varphi(\boldsymbol{w}).
\end{align*}

It is easy to prove the left-hand-side is not smaller than the right-hand-side. In fact, for any $\nu(\cdot|\cdot)\in\cK(\varphi,a)$, by the definition of $\psi$, there holds $\int_{-A}^A \nu(\alpha|\boldsymbol{w}) \log \frac{\nu(\alpha|\boldsymbol{w})}{q(\alpha)}\,\mathrm{d}\alpha \geq \psi(a(\boldsymbol{w}))$ for $\varphi$-almost all $\vw$. Taking integral under $\varphi$ gives the desired inequality. We now prove the other side. By Lemma~\ref{lemma:variation}, for any $u\in(-A,A)$, there exists a unique $p_{t(u)}\in\cP([-A,A])$, such that $\int_{-A}^A \alpha p_{t(u)}(\alpha)\,\mathrm{d}\alpha = u$, and $\psi(u) = \int p_{t(u)}(\alpha)\log \frac{p_{t(u)}(\alpha)}{q(\alpha)}\,\mathrm{d}\alpha$. Take $\nu^\star(\,\mathrm{d}\alpha|\boldsymbol{w})=p_{t(a(\boldsymbol{w}))}(\alpha)\,\mathrm{d}\alpha$. Then, for $\varphi$-almost all $\vw$, $\int \alpha\,\nu^\star(\,\mathrm{d}\alpha|\boldsymbol{w})=a(\boldsymbol{w})$, $\nu^\star(\cdot|\boldsymbol{w})\in\cP([-A,A])$ and the relative entropy is $\psi(a(\boldsymbol{w}))$. Then $\nu^\star(\cdot|\cdot)\in\cK(\varphi,a)$. Taking integral of $\int_{-A}^A\nu^\star(\alpha|\vw)\log\frac{\nu^\star(\alpha|\vw)}{q(\alpha)}\,\mathrm{d}\alpha=\psi(a(\vw))$ with respect to $\varphi$ gives the desired reverse inequality. Therefore, combining the two inequalities, the claimed equality holds.

By Lemma~\ref{lemma:variation}, $a\mapsto \int_{WB_2^d}\psi(a(\boldsymbol{w}))\,\mathrm{d}\varphi(\boldsymbol{w})$ is a strongly convex function. Since $P_N\ell_\nu^\lambda$ is a convex function of $\nu$, by \cite[Appendix A, 1.3, pp. 387]{hiriart-urrutyConvexAnalysisMinimization1993}, $(\hat\varphi_\lambda,\vw\to \mathbb E[\hat \alpha\mid \hat W=\boldsymbol w])$, where $(\hat \alpha, \hat \vW)\sim \hat\nu_\lambda$, is the global minimizer of $F(\varphi,a) = P_N\ell_{\varphi,a} + \lambda\mathcal{D}(\varphi \| Q_W)+\lambda\int_{WB_2^d}\psi(a)\,\mathrm{d}\varphi$.

Finally, after fixing the learned marginal $\hat\varphi_\lambda$, the term $\lambda\mathcal D(\hat\varphi_\lambda\|Q_W)$ is constant in the optimization problem over $a$. Hence the conditional mean $\vw\to \hat a_N(\vw)=\mathbb E[\hat A\mid \hat W=\boldsymbol w]\in{L^2(\hat\varphi_\lambda)}$ satisfies $\hat a_N\in\arg\min_{a\in{L^2(\hat\varphi_\lambda)}}\{N^{-1}\sum_{i=1}^N(Y_i-(T_{\fea}a)(X_i))^2+\lambda\Psi(a)\}$.

Thus $\hat a_N$ is a convex regularized M-estimator on the learned coefficient Hilbert space ${L^2(\hat\varphi_\lambda)}$. By Lemma~\ref{lemma:variation}, for any $a_1,a_2\in\operatorname{dom}(\Psi)$ and any $\zeta_2\in\partial\Psi(a_2)$, $\Psi(a_1)-\Psi(a_2)-\langle\zeta_2,a_1-a_2\rangle_{{L^2(\hat\varphi_\lambda)}}\ge (2A^2)^{-1}\int(a_1-a_2)^2\,\mathrm{d}\hat\varphi_\lambda=(2A^2)^{-1}\|a_1-a_2\|_{{L^2(\hat\varphi_\lambda)}}^2$.

For any \(0<s<1\) and \(\varepsilon>0\), choose \(a_1,a_2\in L^2(\hat\varphi_\lambda)\) such that \(g_j(\phi_{\mathrm{feat}}(\cdot))=\langle a_j,\phi_{\mathrm{neur}}(\cdot)\rangle_{L^2(\hat\varphi_\lambda)}\) and \(\Psi(a_j)\leq \Psi_{\mathcal H_{\mathrm{feat}}}(g_j)+\varepsilon\) for \(j=1,2\). Then \((1-s)a_2+sa_1\) is a coefficient representation of \((1-s)g_2+sg_1\). By the \(1/(2A^2)\)-strong convexity of \(\Psi\) in \(L^2(\hat\varphi_\lambda)\), \(\Psi_{\mathcal H_{\mathrm{feat}}}\big((1-s)g_2+sg_1\big)\leq \Psi\big((1-s)a_2+sa_1\big)\leq (1-s)\Psi(a_2)+s\Psi(a_1)-\frac{s(1-s)}{2A^2}\|a_1-a_2\|_{L^2(\hat\varphi_\lambda)}^2\leq (1-s)\Psi_{\mathcal H_{\mathrm{feat}}}(g_2)+s\Psi_{\mathcal H_{\mathrm{feat}}}(g_1)+\varepsilon-\frac{s(1-s)}{2A^2}\|g_1-g_2\|_{\mathcal H_{\mathrm{feat}}}^2\), where the last inequality follows from \(\|g_1-g_2\|_{\mathcal H_{\mathrm{feat}}}\leq \|a_1-a_2\|_{L^2(\hat\varphi_\lambda)}\), because \(a_1-a_2\) is a coefficient representation of \(g_1-g_2\). Letting \(\varepsilon\downarrow0\), we obtain \(\Psi_{\mathcal H_{\mathrm{feat}}}\big((1-s)g_2+sg_1\big)\leq (1-s)\Psi_{\mathcal H_{\mathrm{feat}}}(g_2)+s\Psi_{\mathcal H_{\mathrm{feat}}}(g_1)-\frac{s(1-s)}{2A^2}\|g_1-g_2\|_{\mathcal H_{\mathrm{feat}}}^2\). Now let \(\eta_2\in\partial\Psi_{\mathcal H_{\mathrm{feat}}}(g_2)\), then \(\Psi_{\mathcal H_{\mathrm{feat}}}\big((1-s)g_2+sg_1\big)\geq \Psi_{\mathcal H_{\mathrm{feat}}}(g_2)+s\langle \eta_2,g_1-g_2\rangle_{\mathcal H_{\mathrm{feat}}}\). Combining the last two inequalities gives \(\Psi_{\mathcal H_{\mathrm{feat}}}(g_2)+s\langle \eta_2,g_1-g_2\rangle_{\mathcal H_{\mathrm{feat}}}\leq (1-s)\Psi_{\mathcal H_{\mathrm{feat}}}(g_2)+s\Psi_{\mathcal H_{\mathrm{feat}}}(g_1)-\frac{s(1-s)}{2A^2}\|g_1-g_2\|_{\mathcal H_{\mathrm{feat}}}^2\). After rearranging and dividing by \(s\), \(\Psi_{\mathcal H_{\mathrm{feat}}}(g_1)-\Psi_{\mathcal H_{\mathrm{feat}}}(g_2)-\langle \eta_2,g_1-g_2\rangle_{\mathcal H_{\mathrm{feat}}}\geq \frac{1-s}{2A^2}\|g_1-g_2\|_{\mathcal H_{\mathrm{feat}}}^2\). Letting \(s\downarrow0\), we conclude the proof.
\endproof

\section{Proof of the existence of feature-learning for spherical MFLD}\label{app:feature-learning-proof}

In this section, we prove the feature-learning property of the spherical MFLD in the Gaussian index models. The proofs show that the learned measure $\hat\varphi_\lambda$ concentrates near the true directions and that this localization / low-dimensional structure leads to sharp estimation rates.

\subsection{Proof of Theorem~\ref{thm:fixed-output-general-ie} via feature learning}\label{sec:proof_thm_fixed-output-general-ie}

Recall that for $\tau$ the uniform measure over $S_2^{d-1}$,
\begin{align*}
\Ent_\tau^-(\varphi):=
\begin{cases}
\displaystyle \int_{S_2^{d-1}}\log\!\left(\frac{\mathrm{d}\varphi}{\mathrm{d}\tau}\right)\mathrm{d}\varphi,& \varphi\ll\tau,\\[0.7em]
+\infty,&\text{otherwise}.
\end{cases}
\end{align*}
For $\lambda>0$, define the fixed-output entropy-regularized empirical minimizer by
\begin{align}
\label{eq:7.1}
\hat\varphi_\lambda\in\arg\min_{\varphi\in\mathcal P(S_2^{d-1})}\left\{P_N(Y-f_\varphi(X))^2+\lambda\Ent_\tau^-(\varphi)\right\}.
\end{align}
For $\lambda=0$, we define $\hat\varphi_0\in\arg\min_{\varphi\in\mathcal P(S_2^{d-1})}\left\{P_N(Y-f_\varphi(X))^2\right\}$.

We use the following tensor notation. Let $\operatorname{Sym}^{\mathrm{IE}(\sigma)}(\mathbb R^d)$ be the space of order-$\mathrm{IE}(\sigma)$ symmetric tensors, equipped with the Frobenius inner product uniquely determined by $\langle \boldsymbol{u}^{\otimes \mathrm{IE}(\sigma)},\boldsymbol{v}^{\otimes \mathrm{IE}(\sigma)}\rangle_F=\langle \boldsymbol{u},\boldsymbol{v}\rangle^{\mathrm{IE}(\sigma)}.$
For a finite signed measure $\nu$ on $S_2^{d-1}$, define its $\mathrm{IE}(\sigma)$-th moment tensor by
\begin{equation}
\label{eq:def_TIE_sigma}
\mathcal T_{\mathrm{IE}(\sigma)}(\nu):=\int_{S_2^{d-1}}\boldsymbol{w}^{\otimes \mathrm{IE}(\sigma)}\nu(\mathrm{d}\boldsymbol{w}).
\end{equation}
For a probability measure $\varphi$, we write $\mathcal T_{\mathrm{IE}(\sigma)}(\varphi)$, and for the Dirac mass at $\boldsymbol{w}_\star$, $\mathcal T_{\mathrm{IE}(\sigma)}(\delta_{\boldsymbol{w}_\star})=\boldsymbol{w}_\star^{\otimes \mathrm{IE}(\sigma)}$. Let $B_Y = B_\sigma + B_\xi$.

\begin{Lemma}
\label{lem:existence-general-ie}
Grant Assumption~\ref{ass:fixed-output-general-ie}. For every $\lambda\ge 0$, the minimization problem in \eqref{eq:7.1} admits a minimizer.
\end{Lemma}

\beginproof
If $\lambda=0$, the objective is the empirical loss, and if $\lambda>0$, the objective is the sum of the weakly continuous empirical loss and the weakly lower semicontinuous entropy term.

The space $\mathcal P(S_2^{d-1})$, endowed with weak convergence, is compact and metrizable because $S_2^{d-1}$ is compact. For fixed data $(X_i,Y_i)_{i=1}^N$, the map $\varphi\mapsto P_N(Y-f_\varphi(X))^2$ is weakly continuous: for every $i$, the function $\boldsymbol{w}\mapsto\sigma(\langle \boldsymbol{w},X_i\rangle)$ is continuous and bounded on $S_2^{d-1}$, hence $\varphi\mapsto f_\varphi(X_i)$ is weakly continuous. The entropy is weakly lower semicontinuous by the variational representation
\begin{align*}
\Ent_\tau^-(\varphi)=\sup_{\psi\in C(S_2^{d-1})}\left\{\int_{S_2^{d-1}}\psi\,\mathrm{d}\varphi-\log\int_{S_2^{d-1}}e^\psi\,\mathrm{d}\tau\right\}.
\end{align*}
Thus the objective in \eqref{eq:7.1} is lower semicontinuous on a compact set. Since $\varphi=\tau$ has finite objective value, the minimum is attained.
\endproof

\begin{Lemma}
\label{lemma:proj-ridge-feature}
Let $C_{\mathrm{IE}(\sigma)}$ denote the $\mathrm{IE}(\sigma)$-th homogeneous Wiener chaos in $L^2(\bP_X)$ defined in Lemma~\ref{lemma:C_m_is_Span_Hermite}. For any $\boldsymbol{w} \in S_2^{d-1}$, let $\mathrm{Proj}_{\mathrm{IE}(\sigma)} \sigma(\langle \boldsymbol{w}, \cdot \rangle)$ be the orthogonal projection of the function $\sigma(\langle \boldsymbol{w}, \cdot \rangle)$ onto $C_{\mathrm{IE}(\sigma)}$, then $\mathrm{Proj}_{\mathrm{IE}(\sigma)} \sigma(\langle \boldsymbol{w}, \cdot \rangle) = \frac{b_{\mathrm{IE}(\sigma)}}{\mathrm{IE}(\sigma)!} \mathrm{He}_{\mathrm{IE}(\sigma)}(\langle \boldsymbol{w}, \cdot \rangle).$
\end{Lemma}
\beginproof
Fix $\boldsymbol{w}, \boldsymbol{v} \in S_2^{d-1}$, and let $G_{\boldsymbol{w}} = \langle \boldsymbol{w}, X \rangle$, $\alpha = \langle \boldsymbol{w}, \boldsymbol{v} \rangle$. By Lemma~\ref{lemma:conditional_expectation_Hermite}, we directly obtain
\begin{align*}
\mathbb{E}[\mathrm{He}_{\mathrm{IE}(\sigma)}(\langle \boldsymbol{v}, X \rangle) \mid \langle \boldsymbol{w}, X \rangle] = \alpha^{\mathrm{IE}(\sigma)} \mathrm{He}_{\mathrm{IE}(\sigma)}(\langle \boldsymbol{w}, X \rangle).
\end{align*}
Using this identity, we compute the inner product:
\begin{align*}
\mathbb{E}[\sigma(G_{\boldsymbol{w}})\mathrm{He}_{\mathrm{IE}(\sigma)}(G_{\boldsymbol{v}})] &= \mathbb{E}[\sigma(G_{\boldsymbol{w}})\mathbb{E}[\mathrm{He}_{\mathrm{IE}(\sigma)}(G_{\boldsymbol{v}}) \mid G_{\boldsymbol{w}}]] \\
&= \alpha^{\mathrm{IE}(\sigma)} \mathbb{E}[\sigma(G_{\boldsymbol{w}})\mathrm{He}_{\mathrm{IE}(\sigma)}(G_{\boldsymbol{w}})] = b_{\mathrm{IE}(\sigma)} \langle \boldsymbol{w}, \boldsymbol{v} \rangle^{\mathrm{IE}(\sigma)}.
\end{align*}
Meanwhile, from the isometric isomorphism established earlier, we know
\begin{align*}
\mathbb{E}[\mathrm{He}_{\mathrm{IE}(\sigma)}(G_{\boldsymbol{w}})\mathrm{He}_{\mathrm{IE}(\sigma)}(G_{\boldsymbol{v}})] = \mathrm{IE}(\sigma)! \langle \boldsymbol{w}, \boldsymbol{v} \rangle^{\mathrm{IE}(\sigma)}.
\end{align*}
Recalling that $G_{\boldsymbol{w}}=\langle\boldsymbol{w},X\rangle$, $G_{\boldsymbol{v}}=\langle\boldsymbol{v},X\rangle$ and $\bP_X$ is standard Gaussian distribution, combining the two equations above reveals that 
\begin{align*}
    \forall \boldsymbol{v}\in S_2^{d-1},\quad \left< \sigma(\langle\boldsymbol{w},\cdot,\rangle) - \frac{b_{\mathrm{IE}(\sigma)}}{\mathrm{IE}(\sigma)!}\mathrm{He}_{\mathrm{IE}(\sigma)}(\langle\boldsymbol{w},\cdot\rangle),\, \mathrm{He}_{\mathrm{IE}(\sigma)}(\langle\boldsymbol{v},\cdot\rangle) \right>_{L^2(\bP_X)}=0,
\end{align*}
By Lemma~\ref{lemma:C_m_is_Span_Hermite}, which states that the family $\{\mathrm{He}_{\mathrm{IE}(\sigma)}(\langle \boldsymbol{v}, \cdot \rangle) : \boldsymbol{v} \in S_2^{d-1}\}$ completely spans the space $C_{\mathrm{IE}(\sigma)}$, the conclusion holds.
\endproof

\begin{Proposition}
\label{prop:m-chaos-localization}
Grant Assumption~\ref{ass:fixed-output-general-ie}. For $\varphi\in\mathcal P(S_2^{d-1})$ and $\mathcal T_{\mathrm{IE}(\sigma)}(\varphi)$  defined in \eqref{eq:def_TIE_sigma}, we have
\begin{align}\label{eq:tensor-localization}
\|\mathcal T_{\mathrm{IE}(\sigma)}(\varphi)-\boldsymbol{w}_\star^{\otimes \mathrm{IE}(\sigma)}\|_F \le \frac{\sqrt{\mathrm{IE}(\sigma)!}}{|b_{\mathrm{IE}(\sigma)}|}\|f_\varphi-f^\star\|_{L^2(\bP_X)}.
\end{align}
Moreover, define
\begin{align*}
d_{\mathrm{IE}(\sigma)}^2(\boldsymbol{w},\boldsymbol{w}_\star) := \|\boldsymbol{w}-\operatorname{sgn}(\langle \boldsymbol{w},\boldsymbol{w}_\star\rangle)^{\mathrm{IE}(\sigma)+1}\boldsymbol{w}_\star\|_2^2,
\end{align*}
where $\operatorname{sgn}(0)=1$, and
\begin{align*}
S_{\mathrm{IE}(\sigma)}(\varphi):=\int_{S_2^{d-1}}d_{\mathrm{IE}(\sigma)}^2(\boldsymbol{w},\boldsymbol{w}_\star)\varphi(\mathrm{d}\boldsymbol{w}).
\end{align*}
Let
\begin{align*}
\kappa_{\mathrm{IE}(\sigma)}:=
\begin{cases}
\displaystyle \min_{t\in[-1,1]}\frac{1-t^{\mathrm{IE}(\sigma)}}{1-t},& \mathrm{IE}(\sigma)\ \mathrm{odd},\\[1em]
1,& \mathrm{IE}(\sigma)\ \mathrm{even},
\end{cases}
\end{align*}
where the ratio at $t=1$ is understood as its continuous extension, equal to $\mathrm{IE}(\sigma)$. Then $\kappa_{\mathrm{IE}(\sigma)}>0$, and
\begin{align}\label{eq:geom-localization}
S_{\mathrm{IE}(\sigma)}(\varphi)\le \frac{2\sqrt{\mathrm{IE}(\sigma)!}}{\kappa_{\mathrm{IE}(\sigma)} |b_{\mathrm{IE}(\sigma)}|}\|f_\varphi-f^\star\|_{L^2(\bP_X)}.
\end{align}
Thus the prediction error localizes $\varphi$ around $\boldsymbol{w}_\star$ when $\mathrm{IE}(\sigma)$ is odd, and around the projective direction $\{\pm \boldsymbol{w}_\star\}$ when $\mathrm{IE}(\sigma)$ is even.
\end{Proposition}

\beginproof
We first prove the tensor localization bound. Let $\nu:=\varphi-\delta_{\boldsymbol{w}_\star}$. By the linearity of the projection and Lemma~\ref{lemma:proj-ridge-feature}, we have
\begin{align*}
\mathrm{Proj}_{\mathrm{IE}(\sigma)}(f_\varphi-f^\star)(\boldsymbol{x}) =\frac{b_{\mathrm{IE}(\sigma)}}{\mathrm{IE}(\sigma)!}\int_{S_2^{d-1}}\mathrm{He}_{\mathrm{IE}(\sigma)}(\langle \boldsymbol{w},\boldsymbol{x}\rangle)\nu(\mathrm{d}\boldsymbol{w}).
\end{align*}
Using Fubini's theorem and \eqref{eq:inner_product_Hermite_polynomials}, we compute the squared $L^2(\bP_X)$ norm of this projection:
\begin{align*}
&\left\|\int \mathrm{He}_{\mathrm{IE}(\sigma)}(\langle \boldsymbol{w},\cdot\rangle)\nu(\mathrm{d}\boldsymbol{w})\right\|_{L^2(\bP_X)}^2
= \iint\mathbb E[\mathrm{He}_{\mathrm{IE}(\sigma)}(\langle \boldsymbol{w},X\rangle)\mathrm{He}_{\mathrm{IE}(\sigma)}(\langle \boldsymbol{v},X\rangle)]\nu(\mathrm{d}\boldsymbol{w})\nu(\mathrm{d}\boldsymbol{v}) \\
&= \mathrm{IE}(\sigma)!\iint\langle \boldsymbol{w},\boldsymbol{v}\rangle^{\mathrm{IE}(\sigma)}\nu(\mathrm{d}\boldsymbol{w})\nu(\mathrm{d}\boldsymbol{v})= \mathrm{IE}(\sigma)! \left\langle \int \boldsymbol{w}^{\otimes \mathrm{IE}(\sigma)}\nu(\mathrm{d}\boldsymbol{w}), \int \boldsymbol{v}^{\otimes \mathrm{IE}(\sigma)}\nu(\mathrm{d}\boldsymbol{v}) \right\rangle_F \\
&= \mathrm{IE}(\sigma)!\left\|\int \boldsymbol{w}^{\otimes \mathrm{IE}(\sigma)}\nu(\mathrm{d}\boldsymbol{w})\right\|_F^2.
\end{align*}
Recall that $\nu = \varphi-\delta_{\boldsymbol{w}_\star}$, then $\int \boldsymbol{w}^{\otimes \mathrm{IE}(\sigma)}\nu(\mathrm{d}\boldsymbol{w}) = \mathcal T_{\mathrm{IE}(\sigma)}(\varphi)-\boldsymbol{w}_\star^{\otimes \mathrm{IE}(\sigma)}$, and consequently,
\begin{align*}
\|\mathrm{Proj}_{\mathrm{IE}(\sigma)}(f_\varphi-f^\star)\|_{L^2(\bP_X)} = \frac{|b_{\mathrm{IE}(\sigma)}|}{\sqrt{\mathrm{IE}(\sigma)!}}\|\mathcal T_{\mathrm{IE}(\sigma)}(\varphi)-\boldsymbol{w}_\star^{\otimes \mathrm{IE}(\sigma)}\|_F.
\end{align*}
Since orthogonal projection is a contraction in $L^2(\bP_X)$, $\|f_\varphi-f^\star\|_{L^2(\bP_X)}\ge\|\mathrm{Proj}_{\mathrm{IE}(\sigma)}(f_\varphi-f^\star)\|_{L^2(\bP_X)}$, and \eqref{eq:tensor-localization} follows.

We now prove the concentration of the stationary measure.
\begin{enumerate}
    \item When $\mathrm{IE}(\sigma)$ is odd, $\operatorname{sgn}(\langle \boldsymbol{w},\boldsymbol{w}_\star\rangle)^{\mathrm{IE}(\sigma)+1}=1$, so $d_{\mathrm{IE}(\sigma)}^2(\boldsymbol{w},\boldsymbol{w}_\star)=\|\boldsymbol{w}-\boldsymbol{w}_\star\|_2^2=2(1-\langle \boldsymbol{w},\boldsymbol{w}_\star\rangle)$. For $t\in[-1,1]$, we have the algebraic identity: $1-t^{\mathrm{IE}(\sigma)}=(1-t)\sum_{\ell=0}^{\mathrm{IE}(\sigma)-1}t^\ell.$
Recall that when $\mathrm{IE}(\sigma)$ is odd, $\kappa_{\mathrm{IE}(\sigma)} = \min_{t \in [-1, 1]} \frac{1 - t^{\mathrm{IE}(\sigma)}}{1 - t}$. The continuous extension of the function $t \mapsto \frac{1 - t^{\mathrm{IE}(\sigma)}}{1 - t}$ at $t=1$ is $\lim_{t \to 1} \sum_{\ell=0}^{\mathrm{IE}(\sigma)-1} t^\ell = \mathrm{IE}(\sigma) > 0$. Since $\mathrm{IE}(\sigma)$ is odd, $\frac{1 - t^{\mathrm{IE}(\sigma)}}{1 - t} > 0$ for all $t < 1$. Hence $\kappa_{\mathrm{IE}(\sigma)}>0$, and $1-t\le \kappa_{\mathrm{IE}(\sigma)}^{-1}(1-t^{\mathrm{IE}(\sigma)})$ holds uniformly on $[-1,1]$. Applying this to $t = \langle \boldsymbol{w}, \boldsymbol{w}_\star \rangle$ yields
\begin{align}\label{eq:Sm-bound-1}
S_{\mathrm{IE}(\sigma)}(\varphi)=2\int(1-\langle \boldsymbol{w},\boldsymbol{w}_\star\rangle)\varphi(\mathrm{d}\boldsymbol{w}) \le \frac{2}{\kappa_{\mathrm{IE}(\sigma)}}\int(1-\langle \boldsymbol{w},\boldsymbol{w}_\star\rangle^{\mathrm{IE}(\sigma)})\varphi(\mathrm{d}\boldsymbol{w}).
\end{align}

\item If $\mathrm{IE}(\sigma)$ is even, $\operatorname{sgn}(\langle \boldsymbol{w},\boldsymbol{w}_\star\rangle)^{\mathrm{IE}(\sigma)+1}=\operatorname{sgn}(\langle \boldsymbol{w},\boldsymbol{w}_\star\rangle)$, which means $d_{\mathrm{IE}(\sigma)}^2(\boldsymbol{w},\boldsymbol{w}_\star)=\|\boldsymbol{w}-\operatorname{sgn}(\langle \boldsymbol{w},\boldsymbol{w}_\star\rangle)\boldsymbol{w}_\star\|_2^2=2(1-|\langle \boldsymbol{w},\boldsymbol{w}_\star\rangle|)$. Since $0\le |\langle \boldsymbol{w},\boldsymbol{w}_\star\rangle|\le1$, we have $1-|\langle \boldsymbol{w},\boldsymbol{w}_\star\rangle|\le1-|\langle \boldsymbol{w},\boldsymbol{w}_\star\rangle|^{\mathrm{IE}(\sigma)}=1-\langle \boldsymbol{w},\boldsymbol{w}_\star\rangle^{\mathrm{IE}(\sigma)}$. Recall that $\kappa_{\mathrm{IE}(\sigma)}=1$ when $\mathrm{IE}(\sigma)$ is even, so the inequality \eqref{eq:Sm-bound-1} also holds in the even case.
\end{enumerate}

Finally, by the linearity of the tensor Frobenius inner product,
\begin{align*}
\int \langle \boldsymbol{w},\boldsymbol{w}_\star\rangle^{\mathrm{IE}(\sigma)}\varphi(\mathrm{d}\boldsymbol{w}) = \int\langle \boldsymbol{w}^{\otimes \mathrm{IE}(\sigma)},\boldsymbol{w}_\star^{\otimes \mathrm{IE}(\sigma)}\rangle_F\varphi(\mathrm{d}\boldsymbol{w}) = \left\langle \mathcal T_{\mathrm{IE}(\sigma)}(\varphi),\boldsymbol{w}_\star^{\otimes \mathrm{IE}(\sigma)}\right\rangle_F.
\end{align*}
Since $\|\boldsymbol{w}_\star^{\otimes \mathrm{IE}(\sigma)}\|_F=1$, we can write
\begin{align*}
\int(1-\langle \boldsymbol{w},\boldsymbol{w}_\star\rangle^{\mathrm{IE}(\sigma)})\varphi(\mathrm{d}\boldsymbol{w}) &= \langle \boldsymbol{w}_\star^{\otimes \mathrm{IE}(\sigma)}, \boldsymbol{w}_\star^{\otimes \mathrm{IE}(\sigma)} \rangle_F - \langle \mathcal T_{\mathrm{IE}(\sigma)}(\varphi),\boldsymbol{w}_\star^{\otimes \mathrm{IE}(\sigma)}\rangle_F \\
&= \left\langle \boldsymbol{w}_\star^{\otimes \mathrm{IE}(\sigma)}-\mathcal T_{\mathrm{IE}(\sigma)}(\varphi),\boldsymbol{w}_\star^{\otimes \mathrm{IE}(\sigma)}\right\rangle_F.
\end{align*}
Applying the Cauchy-Schwarz inequality, this is bounded by
\begin{align*}
\left\langle \boldsymbol{w}_\star^{\otimes \mathrm{IE}(\sigma)}-\mathcal T_{\mathrm{IE}(\sigma)}(\varphi),\boldsymbol{w}_\star^{\otimes \mathrm{IE}(\sigma)}\right\rangle_F \le \|\mathcal T_{\mathrm{IE}(\sigma)}(\varphi)-\boldsymbol{w}_\star^{\otimes \mathrm{IE}(\sigma)}\|_F \|\boldsymbol{w}_\star^{\otimes \mathrm{IE}(\sigma)}\|_F = \|\mathcal T_{\mathrm{IE}(\sigma)}(\varphi)-\boldsymbol{w}_\star^{\otimes \mathrm{IE}(\sigma)}\|_F.
\end{align*}
Substituting this bound into \eqref{eq:Sm-bound-1} and invoking the tensor localization bound \eqref{eq:tensor-localization}, we obtain \eqref{eq:geom-localization}.
\endproof

\begin{Lemma}
\label{lem:second-order-multiplier-general-ie}
Grant Assumption~\ref{ass:fixed-output-general-ie}. Let $(X_i)_{i=1}^N$ be independent copies of $X$, and let $(\varepsilon_i)_{i=1}^N$ be independent Rademacher variables independent of $(X_i)_{i=1}^N$. For $\boldsymbol{u}\in3B_2^d$, define
\begin{align*}
A_N(\boldsymbol{u}):=\frac1N\sum_{i=1}^N\varepsilon_i\sigma''(\langle \boldsymbol{u},X_i\rangle)X_iX_i^\top.
\end{align*}
There exists a constant $C_{\rm sec}\ge1$, depending only on $(M_\sigma,T_\sigma)$, such that
\begin{align}
\label{eq:19.m.1}
\mathbb E\sup_{\|\boldsymbol{u}\|_2\le3}\|A_N(\boldsymbol{u})\|_{\rm op} \le C_{\rm sec}\left(\sqrt{\frac{D_{d,N}}{N}}+\frac{D_{d,N}}{N}\right), \qquad D_{d,N}:=d\log(edN).
\end{align}
Consequently, for
\begin{align*}
\mathcal Q_1:=\left\{\boldsymbol{x}\mapsto \sigma''(\langle \boldsymbol{u},\boldsymbol{x}\rangle)\langle \boldsymbol{e},\boldsymbol{x}\rangle^2:\|\boldsymbol{u}\|_2\le1,\ \boldsymbol{e}\in S_2^{d-1}\right\},
\end{align*}
one has
\begin{align}
\label{eq:19.m.2}
\mathbb E\sup_{q\in\mathcal Q_1}\left|\frac1N\sum_{i=1}^N\varepsilon_iq(X_i)\right| \le C_{\rm sec}\left(\sqrt{\frac{D_{d,N}}{N}}+\frac{D_{d,N}}{N}\right).
\end{align}
\end{Lemma}

\beginproof
Fix $\boldsymbol{u}\in3B_2^d$. Let $\mathcal V$ be a $1/4$-net of $S_2^{d-1}$ with $|\mathcal V|\le 9^d$. Since $A_N(\boldsymbol{u})$ is symmetric,
\begin{align}
\label{eq:19.m.3}
\|A_N(\boldsymbol{u})\|_{\rm op}\le2\sup_{\boldsymbol{v}\in\mathcal V}|\boldsymbol{v}^\top A_N(\boldsymbol{u})\boldsymbol{v}|.
\end{align}
For fixed $\boldsymbol{u}$ and $\boldsymbol{v}$ in $3B_2^d$, define $Z_i(\boldsymbol{u},\boldsymbol{v}):=\varepsilon_i\sigma''(\langle \boldsymbol{u},X_i\rangle)\langle \boldsymbol{v},X_i\rangle^2$. The variables $Z_i(\boldsymbol{u},\boldsymbol{v})$ are independent and centered. Moreover, $|Z_i(\boldsymbol{u},\boldsymbol{v})|\le M_\sigma\langle \boldsymbol{v},X_i\rangle^2$, and $\langle \boldsymbol{v},X_i\rangle^2$ is a sub-exponential with a  $\psi_1$-Orlicz norm constant independent of $\vv$. Therefore $\|Z_i(\boldsymbol{u},\boldsymbol{v})\|_{\psi_1}\le CM_\sigma$. Bernstein's inequality for centered sub-exponential variables gives, for every $s\ge1$,
\begin{align*}
\mathbb P\left(\left|\frac1N\sum_{i=1}^NZ_i(\boldsymbol{u},\boldsymbol{v})\right|>CM_\sigma\left(\sqrt{\frac{s}{N}}+\frac{s}{N}\right)\right)\le2e^{-s}.
\end{align*}
Applying the union bound over $\mathcal V$ and using \eqref{eq:19.m.3}, we get, for every fixed $\boldsymbol{u}$ and every $s\ge1$,
\begin{align}
\label{eq:19.m.5}
\mathbb P\left(\|A_N(\boldsymbol{u})\|_{\rm op}>CM_\sigma\left(\sqrt{\frac{d+s}{N}}+\frac{d+s}{N}\right)\right)\le2e^{-s}.
\end{align}

Let $\mathcal N_\eta$ be an $\eta$-net of $3B_2^d$, with $\eta=(edN)^{-2}$ and $|\mathcal N_\eta|\le(C/\eta)^d\le \exp(Cd\log(edN))$. Applying \eqref{eq:19.m.5} on this net and taking a union bound yields, for every $s\ge1$, with probability at least $1-2e^{-s}$,
\begin{align*}
\sup_{\boldsymbol{u}\in\mathcal N_\eta}\|A_N(\boldsymbol{u})\|_{\rm op} \le C\left(\sqrt{\frac{D_{d,N}+s}{N}}+\frac{D_{d,N}+s}{N}\right).
\end{align*}
Integrating this tail bound gives
\begin{align}
\label{eq:19.m.7}
\mathbb E\sup_{\boldsymbol{u}\in\mathcal N_\eta}\|A_N(\boldsymbol{u})\|_{\rm op} \le C\left(\sqrt{\frac{D_{d,N}}{N}}+\frac{D_{d,N}}{N}\right).
\end{align}

It remains to pass from $\mathcal N_\eta$ to the whole ball $3B_2^d$. For $\boldsymbol{u},\boldsymbol{v}\in3B_2^d$,
\begin{align*}
\|A_N(\boldsymbol{u})-A_N(\boldsymbol{v})\|_{\rm op} \le \frac1N\sum_{i=1}^N|\sigma''(\langle \boldsymbol{u},X_i\rangle)-\sigma''(\langle \boldsymbol{v},X_i\rangle)|\,\|X_i\|_2^2 \le T_\sigma\|\boldsymbol{u}-\boldsymbol{v}\|_2\frac1N\sum_{i=1}^N\|X_i\|_2^3.
\end{align*}
Since $\mathbb E\|X\|_2^3\le C d^{3/2}$, we have
\begin{align}
\label{eq:19.m.9}
\mathbb E\sup_{\|\boldsymbol{u}\|_2\le3}\inf_{\boldsymbol{v}\in\mathcal N_\eta}\|A_N(\boldsymbol{u})-A_N(\boldsymbol{v})\|_{\rm op} \le CT_\sigma\eta d^{3/2}\le CN^{-2}.
\end{align}
Combining \eqref{eq:19.m.7} and \eqref{eq:19.m.9}, and enlarging the constant, proves \eqref{eq:19.m.1}.

For $q(\boldsymbol{x})=\sigma''(\langle \boldsymbol{u},\boldsymbol{x}\rangle)\langle \boldsymbol{e},\boldsymbol{x}\rangle^2\in\mathcal Q_1$, $\frac1N\sum_{i=1}^N\varepsilon_iq(X_i)=\boldsymbol{e}^\top A_N(\boldsymbol{u})\boldsymbol{e},$
with $\|\boldsymbol{u}\|_2\le1$. Hence the absolute value is at most $\sup_{\|\boldsymbol{u}\|_2\le1}\|A_N(\boldsymbol{u})\|_{\rm op}$, which is bounded by the left-hand side of \eqref{eq:19.m.1}. This proves \eqref{eq:19.m.2}.
\endproof

\begin{Proposition}
\label{prop:parametric-rademacher-general-ie}
Grant Assumption~\ref{ass:fixed-output-general-ie}. For $r>0$, define
\begin{align*}
(\mathcal{F}-f^\star)\cap B_{L^2(\bP_X)}(0;r):=\left\{f_\varphi-f^\star:\varphi\in\mathcal P(S_2^{d-1}),\ \|f_\varphi-f^\star\|_{L^2(\bP_X)}\le r\right\}.
\end{align*}
Then there exists a constant $C_{\rm rad}\ge1$, depending only on $(\mathrm{IE}(\sigma),L_\sigma,M_\sigma,T_\sigma,|b_{\mathrm{IE}(\sigma)}|^{-1})$, such that
\begin{align}
\label{eq:20.m.1}
\mathbb E\sup_{g\in(\mathcal{F}-f^\star)\cap B_{L^2(\bP_X)}(0;r)}\left|\frac1N\sum_{i=1}^N\varepsilon_ig(X_i)\right| \le C_{\rm rad}r\left(\sqrt{\frac{D_{d,N}}{N}}+\frac{D_{d,N}}{N}\right).
\end{align}
Moreover, define
\begin{align*}
\mathcal L(r):=\left\{(\boldsymbol{x},y)\mapsto (y-f_\varphi(\boldsymbol{x}))^2-(y-f^\star(\boldsymbol{x}))^2:f_\varphi-f^\star\in(\mathcal{F}-f^\star)\cap B_{L^2(\bP_X)}(0;r)\right\}.
\end{align*}
Then there exists a constant $C_{\rm rad}'\ge1$, depending only on $(\mathrm{IE}(\sigma),B_\sigma,L_\sigma,M_\sigma,T_\sigma,B_\xi,|b_{\mathrm{IE}(\sigma)}|^{-1})$, such that
\begin{align}
\label{eq:20.m.2}
\mathbb E\sup_{h\in\mathcal L(r)}\left|\frac1N\sum_{i=1}^N\varepsilon_ih(X_i,Y_i)\right| \le C_{\rm rad}'r\left(\sqrt{\frac{D_{d,N}}{N}}+\frac{D_{d,N}}{N}\right).
\end{align}
\end{Proposition}

\beginproof
Throughout the proof, all suprema are understood over separable versions of the processes.
Fix $\varphi\in\cP(S_2^{d-1})$ and write $g_\varphi:=f_\varphi-f^\star$. Let
\begin{align*}
s(\boldsymbol{w}):= \operatorname{sgn}(\langle \boldsymbol{w},\boldsymbol{w}_\star\rangle)^{\mathrm{IE}(\sigma)+1},
\end{align*}
with the convention $\operatorname{sgn}(0)=1$. Recalling $d_{\mathrm{IE}(\sigma)}^2(\boldsymbol{w},\boldsymbol{w}_\star) = \|\boldsymbol{w}-s(\boldsymbol{w})\boldsymbol{w}_\star\|_2^2$, Proposition~\ref{prop:m-chaos-localization} gives
\begin{align}
\label{eq:20.m.3}
S_\varphi:=\int_{S_2^{d-1}}\|\boldsymbol{w}-s(\boldsymbol{w})\boldsymbol{w}_\star\|_2^2\varphi(\mathrm{d}\boldsymbol{w}) \le \frac{2\sqrt{\mathrm{IE}(\sigma)!}}{\kappa_{\mathrm{IE}(\sigma)} |b_{\mathrm{IE}(\sigma)}|}\|f_\varphi-f^\star\|_{L^2(\bP_X)}.
\end{align}

For any fixed $\boldsymbol{w} \in S_2^{d-1}$, we decompose it around the base point $s(\boldsymbol{w})\boldsymbol{w}_\star$. If $\boldsymbol{w} \neq s(\boldsymbol{w})\boldsymbol{w}_\star$, we define the unit direction vector $\boldsymbol{e}_{\boldsymbol{w}} := (\boldsymbol{w} - s(\boldsymbol{w})\boldsymbol{w}_\star) / \|\boldsymbol{w} - s(\boldsymbol{w})\boldsymbol{w}_\star\|_2$. If $\boldsymbol{w} = s(\boldsymbol{w})\boldsymbol{w}_\star$, we arbitrarily choose any unit vector $\boldsymbol{e}_{\boldsymbol{w}} \in S_2^{d-1}$. In both cases, taking the inner product with $\boldsymbol{x}$ gives
\begin{align*}
\langle \boldsymbol{w},\boldsymbol{x}\rangle=s(\boldsymbol{w})\langle \boldsymbol{w}_\star,\boldsymbol{x}\rangle+\|\boldsymbol{w}-s(\boldsymbol{w})\boldsymbol{w}_\star\|_2\langle \boldsymbol{e}_{\boldsymbol{w}},\boldsymbol{x}\rangle.
\end{align*}

Using Taylor's formula with an integral remainder at the base point $s(\boldsymbol{w})\langle \boldsymbol{w}_\star, \boldsymbol{x} \rangle$, we expand $\sigma$:
\begin{align}
\label{eq:20.m.4}
&\sigma(\langle \boldsymbol{w},\boldsymbol{x}\rangle)-\sigma(s(\boldsymbol{w})\langle \boldsymbol{w}_\star,\boldsymbol{x}\rangle) = \sigma'(s(\boldsymbol{w})\langle \boldsymbol{w}_\star,\boldsymbol{x}\rangle)\langle \boldsymbol{w}-s(\boldsymbol{w})\boldsymbol{w}_\star,\boldsymbol{x}\rangle \nonumber \\
&\quad+ \|\boldsymbol{w}-s(\boldsymbol{w})\boldsymbol{w}_\star\|_2^2 \int_0^1(1-t)\sigma''(\langle (1-t)s(\boldsymbol{w})\boldsymbol{w}_\star+t\boldsymbol{w},\boldsymbol{x}\rangle)\langle \boldsymbol{e}_{\boldsymbol{w}},\boldsymbol{x}\rangle^2\mathrm{d}t.
\end{align}

We rewrite the integral part as $q_{\boldsymbol{w}}(\boldsymbol{x})$ by explicitly introducing the probability density $2(1-t)$ on $[0,1]$:
\begin{align*}
q_{\boldsymbol{w}}(\boldsymbol{x}) := \frac{1}{2} \int_0^1 2(1-t)\sigma''(\langle (1-t)s(\boldsymbol{w})\boldsymbol{w}_\star+t\boldsymbol{w},\boldsymbol{x}\rangle)\langle \boldsymbol{e}_{\boldsymbol{w}},\boldsymbol{x}\rangle^2\mathrm{d}t.
\end{align*}
For any fixed $t \in [0,1]$, the vector $(1-t)s(\boldsymbol{w})\boldsymbol{w}_\star+t\boldsymbol{w}$ is a convex combination of two unit vectors, meaning its $\ell_2$-norm is at most $1$. Thus, the integrand $\boldsymbol{x} \mapsto \sigma''(\langle (1-t)s(\boldsymbol{w})\boldsymbol{w}_\star+t\boldsymbol{w},\boldsymbol{x}\rangle)\langle \boldsymbol{e}_{\boldsymbol{w}},\boldsymbol{x}\rangle^2$ belongs to the function class $\mathcal Q_1 := \left\{\boldsymbol{x}\mapsto\sigma''(\langle \boldsymbol{u},\boldsymbol{x}\rangle)\langle \boldsymbol{e},\boldsymbol{x}\rangle^2:\|\boldsymbol{u}\|_2\le1,\ \boldsymbol{e}\in S_2^{d-1}\right\}$. Integrating over the probability density $2(1-t)$ forms a convex combination of elements in $\mathcal Q_1$. Scaling by the $1/2$ factor outside, we deduce that
\begin{align*}
q_{\boldsymbol{w}}\in\frac12\overline{\operatorname{conv}}(\mathcal Q_1),\mbox{ where } \overline{\conv}\mbox{ is the closure of convex hull in }L^2(\bP_X).
\end{align*}
Integrating \eqref{eq:20.m.4} with respect to $\varphi$ gives $g_\varphi(\boldsymbol{x}) = f_\varphi(\boldsymbol{x}) - f^\star(\boldsymbol{x}) = \int_{S_2^{d-1}} (\sigma(\langle \boldsymbol{w}, \boldsymbol{x} \rangle) - \sigma(\langle \boldsymbol{w}_\star, \boldsymbol{x} \rangle)) \varphi(\mathrm{d}\boldsymbol{w})$. We decompose $g_\varphi=H_\varphi+R_\varphi$, where the main term $H_\varphi$ collects the zero-th and first-order terms:
\begin{align}
\label{eq:20.m.7}
H_\varphi(\boldsymbol{x}):= \int_{S_2^{d-1}}\left\{\sigma(s(\boldsymbol{w})\langle \boldsymbol{w}_\star,\boldsymbol{x}\rangle)-\sigma(\langle \boldsymbol{w}_\star,\boldsymbol{x}\rangle) + \sigma'(s(\boldsymbol{w})\langle \boldsymbol{w}_\star,\boldsymbol{x}\rangle)\langle \boldsymbol{w}-s(\boldsymbol{w})\boldsymbol{w}_\star,\boldsymbol{x}\rangle\right\}\varphi(\mathrm{d}\boldsymbol{w}),
\end{align}
and the remainder term is $R_\varphi(\boldsymbol{x}) := \int_{S_2^{d-1}} \|\boldsymbol{w}-s(\boldsymbol{w})\boldsymbol{w}_\star\|_2^2 q_{\boldsymbol{w}}(\boldsymbol{x})\varphi(\mathrm{d}\boldsymbol{w})$. Suppose $S_\varphi>0$
By normalizing with $S_\varphi$, we rewrite $R_\varphi(\boldsymbol{x}) = S_\varphi \int_{S_2^{d-1}} q_{\boldsymbol{w}}(\boldsymbol{x}) \frac{\|\boldsymbol{w} - s(\boldsymbol{w})\boldsymbol{w}_\star\|_2^2}{S_\varphi} \varphi(\mathrm{d}\boldsymbol{w})$. Since the integral is over a valid probability measure and $q_{\boldsymbol{w}} \in \frac{1}{2}\overline{\operatorname{conv}}(\mathcal Q_1)$, the result remains in the same closed convex set. Multiplying back by $S_\varphi$, we obtain
\begin{align}
\label{eq:20.m.8}
R_\varphi\in \frac{S_\varphi}{2}\overline{\operatorname{conv}}(\mathcal Q_1).
\end{align}
When $S_\varphi=0$, then $R_\varphi=0$.
We next show $H_\varphi$ belongs to a specific finite-dimensional space. Define
\begin{align*}
\mathcal V:= \operatorname{span}\left\{ \boldsymbol{x}\mapsto \sigma(-\langle \boldsymbol{w}_\star,\boldsymbol{x}\rangle)-\sigma(\langle \boldsymbol{w}_\star,\boldsymbol{x}\rangle),\ \boldsymbol{x}\mapsto\sigma'(\langle \boldsymbol{w}_\star,\boldsymbol{x}\rangle)x_j,\ \boldsymbol{x}\mapsto\sigma'(-\langle \boldsymbol{w}_\star,\boldsymbol{x}\rangle)x_j:\ 1\le j\le d \right\}.
\end{align*}
Clearly, $\dim(\mathcal V)\le2d+1$. 
If $\mathrm{IE}(\sigma)$ is odd, then $s(\boldsymbol{w})\equiv1$. The term $\sigma(s(\boldsymbol{w})\langle \boldsymbol{w}_\star,\boldsymbol{x}\rangle)-\sigma(\langle \boldsymbol{w}_\star,\boldsymbol{x}\rangle)$ identically vanishes. Thus $H_\varphi$ simplifies to
\begin{align*}
H_\varphi(\boldsymbol{x}) = \sigma'(\langle \boldsymbol{w}_\star,\boldsymbol{x}\rangle)\left\langle\int(\boldsymbol{w}-\boldsymbol{w}_\star)\varphi(\mathrm{d}\boldsymbol{w}),\boldsymbol{x}\right\rangle = \sum_{j=1}^d \left( \int (\boldsymbol{w}_j-\boldsymbol{w}_{\star,j})\varphi(\mathrm{d}\boldsymbol{w}) \right) \sigma'(\langle \boldsymbol{w}_\star,\boldsymbol{x}\rangle)x_j.
\end{align*}
This is exactly a linear combination of the basis functions $\boldsymbol{x} \mapsto \sigma'(\langle \boldsymbol{w}_\star,\boldsymbol{x}\rangle)x_j$, hence $H_\varphi\in\mathcal V$.

If $\mathrm{IE}(\sigma)$ is even, $s(\boldsymbol{w})=\operatorname{sgn}(\langle \boldsymbol{w},\boldsymbol{w}_\star\rangle)$. We split $S_2^{d-1}=A_+\cup A_-$ where $A_+=\{\boldsymbol{w}:\langle \boldsymbol{w},\boldsymbol{w}_\star\rangle\ge0\}$ and $A_-=\{\boldsymbol{w}:\langle \boldsymbol{w},\boldsymbol{w}_\star\rangle<0\}$. 
On $A_+$, $s(\boldsymbol{w})=1$, so the integrand of $H_\varphi(\boldsymbol{x})$ is $\sigma'(\langle \boldsymbol{w}_\star,\boldsymbol{x}\rangle)\langle \boldsymbol{w}-\boldsymbol{w}_\star,\boldsymbol{x}\rangle$. 
On $A_-$, $s(\boldsymbol{w})=-1$, so the integrand becomes $(\sigma(-\langle \boldsymbol{w}_\star,\boldsymbol{x}\rangle)-\sigma(\langle \boldsymbol{w}_\star,\boldsymbol{x}\rangle)) + \sigma'(-\langle \boldsymbol{w}_\star,\boldsymbol{x}\rangle)\langle \boldsymbol{w}+\boldsymbol{w}_\star,\boldsymbol{x}\rangle$. 
Integrating over these two regions separately and summing the results up, we obtain
\begin{align*}
H_\varphi(\boldsymbol{x}) &= \bigg[ \varphi(A_-)(\sigma(-\langle \boldsymbol{w}_\star,\boldsymbol{x}\rangle)-\sigma(\langle \boldsymbol{w}_\star,\boldsymbol{x}\rangle)) \\
&\quad+\sigma'(\langle \boldsymbol{w}_\star,\boldsymbol{x}\rangle)\left\langle\int_{A_+}(\boldsymbol{w}-\boldsymbol{w}_\star)\varphi(\mathrm{d}\boldsymbol{w}),\boldsymbol{x}\right\rangle +\sigma'(-\langle \boldsymbol{w}_\star,\boldsymbol{x}\rangle)\left\langle\int_{A_-}(\boldsymbol{w}+\boldsymbol{w}_\star)\varphi(\mathrm{d}\boldsymbol{w}),\boldsymbol{x}\right\rangle \bigg].
\end{align*}
The first term is a scalar multiple of $\sigma(-\langle \boldsymbol{w}_\star,\boldsymbol{x}\rangle)-\sigma(\langle \boldsymbol{w}_\star,\boldsymbol{x}\rangle)$, and the remaining inner products can be expanded into linear combinations of $\boldsymbol{x} \mapsto \sigma'(\pm\langle \boldsymbol{w}_\star,\boldsymbol{x}\rangle)x_j$ just as in the odd case. Thus, again $H_\varphi\in\mathcal V$.

Notice that $\sup(\|q\|_{L^2(\bP_X)}: q\in\cQ_1) \le M_\sigma\|\langle \boldsymbol{e},X\rangle^2\|_{L^2} = \sqrt3M_\sigma.$
Therefore \eqref{eq:20.m.8} and \eqref{eq:20.m.3} imply
\begin{align*}
\|R_\varphi\|_{L^2(\bP_X)} \le \frac{\sqrt3M_\sigma}{2}S_\varphi \le \frac{\sqrt{3\mathrm{IE}(\sigma)!}M_\sigma}{\kappa_{\mathrm{IE}(\sigma)}|b_{\mathrm{IE}(\sigma)}|}\|f_\varphi-f^\star\|_{L^2(\bP_X)}.
\end{align*}
Since $g_\varphi=H_\varphi+R_\varphi$, if $g_\varphi\in(\mathcal{F}-f^\star)\cap B_{L^2(\bP_X)}(0;r)$, then
\begin{align*}
H_\varphi\in\mathcal V \mbox{ and } \|H_\varphi\|_{L^2(\bP_X)} \le \|g_\varphi\|_{L^2(\bP_X)}+\|R_\varphi\|_{L^2(\bP_X)} \le C_{V,\mathrm{IE}(\sigma)}r,
\end{align*}
where $C_{V,\mathrm{IE}(\sigma)}$ depends only on $(\mathrm{IE}(\sigma),M_\sigma,|b_{\mathrm{IE}(\sigma)}|^{-1})$. Also, by \eqref{eq:20.m.8} and \eqref{eq:20.m.3},
\begin{align}
\label{eq:20.m.13}
R_\varphi\in \frac{\sqrt{\mathrm{IE}(\sigma)!}}{\kappa_{\mathrm{IE}(\sigma)}|b_{\mathrm{IE}(\sigma)}|}r\,\overline{\operatorname{conv}}(\mathcal Q_1) \qquad\text{whenever }g_\varphi\in(\mathcal{F}-f^\star)\cap B_{L^2(\bP_X)}(0;r).
\end{align}

We now bound the expected supremum of the empirical process for $H_\varphi$. Because the set $\{H_\varphi : g_\varphi \in (\mathcal{F}-f^\star)\cap B_{L^2(\bP_X)}(0;r)\}$ strictly falls into the deterministic bounded ball $\{h \in \mathcal V: \|h\|_{L^2} \le C_{V,\mathrm{IE}(\sigma)}r\}$, we can upper bound the supremum over $g_\varphi$ by taking the supremum over this entire ball. Let $\psi_1,\ldots,\psi_q$ be an $L^2(\bP_X)$-orthonormal basis of $\mathcal V$, with $q\le2d+1$. This enlargement of the index set yields:
\begin{align}\label{eq:20.m.14}
\begin{aligned}
    &\mathbb E\sup_{g_\varphi \in (\mathcal{F}-f^\star)\cap B_{L^2(\bP_X)}(0;r)} \left| \frac{1}{N} \sum_{i=1}^N \varepsilon_i H_\varphi(X_i) \right| \le \mathbb E\sup_{\substack{h\in\mathcal V\\ \|h\|_{L^2}\le C_{V,\mathrm{IE}(\sigma)}r}} \left|\frac1N\sum_{i=1}^N\varepsilon_ih(X_i)\right| \\
&\le C_{V,\mathrm{IE}(\sigma)}r \left(\sum_{\ell=1}^q \mathbb E\left[\frac1N\sum_{i=1}^N\varepsilon_i\psi_\ell(X_i)\right]^2\right)^{1/2} = C_{V,\mathrm{IE}(\sigma)}r\sqrt{\frac qN} \le Cr\sqrt{\frac dN}.
\end{aligned}
\end{align}
For the remainder, \eqref{eq:20.m.13} and Lemma~\ref{lem:second-order-multiplier-general-ie} give
\begin{align}
\label{eq:20.m.15}
\mathbb E\sup_{g_\varphi\in(\mathcal{F}-f^\star)\cap B_{L^2(\bP_X)}(0;r)} \left|\frac1N\sum_{i=1}^N\varepsilon_iR_\varphi(X_i)\right| &\le \frac{\sqrt{\mathrm{IE}(\sigma)!}}{\kappa_{\mathrm{IE}(\sigma)}|b_{\mathrm{IE}(\sigma)}|}r\, \mathbb E\sup_{q\in\mathcal Q_1} \left|\frac1N\sum_{i=1}^N\varepsilon_iq(X_i)\right| \nonumber \\
&\le Cr\left(\sqrt{\frac{D_{d,N}}{N}}+\frac{D_{d,N}}{N}\right).
\end{align}
Combining the decomposition $g_\varphi = H_\varphi + R_\varphi$, \eqref{eq:20.m.14}, \eqref{eq:20.m.15}, and $d\le D_{d,N}$ proves \eqref{eq:20.m.1}.

For the loss class, write $g=f_\varphi-f^\star$. Since $Y=f^\star(X)+\xi$, then $(Y-f_\varphi(X))^2-(Y-f^\star(X))^2 = g(X)^2-2\xi g(X).$ Also $|g(X)|\le |f_\varphi(X)|+|f^\star(X)|\le2 B_\sigma$. Conditionally on the data, the map $u\mapsto u^2-2\xi_i u$ is $(4 B_\sigma+2B_\xi)$-Lipschitz on $[-2 B_\sigma,2 B_\sigma]$ and vanishes at $u=0$. By the contraction principle, see, for instance, \cite[Theorem 6.7.1]{vershyninHighDimensionalProbabilityIntroduction2018},
\begin{align*}
\mathbb E\sup_{h\in\mathcal L(r)} \left|\frac1N\sum_{i=1}^N\varepsilon_ih(X_i,Y_i)\right| \le C(4 B_\sigma+2B_\xi) \mathbb E\sup_{g\in(\mathcal{F}-f^\star)\cap B_{L^2(\bP_X)}(0;r)} \left|\frac1N\sum_{i=1}^N\varepsilon_ig(X_i)\right|.
\end{align*}
Using \eqref{eq:20.m.1} proves \eqref{eq:20.m.2}.
\endproof

\begin{Lemma}
\label{lem:localized-isomorphism-general-ie}
Grant Assumption~\ref{ass:fixed-output-general-ie}. There exists a constant $C_{\rm iso}\ge1$, depending only on $\mathrm{IE}(\sigma)$, $B_\sigma$, $L_\sigma$, $M_\sigma$, $T_\sigma$, $B_\xi,$ and $|b_{\mathrm{IE}(\sigma)}|^{-1}$, such that, for every $x\ge1$ and $N\ge2$,
\begin{align}
\label{eq:RIP_single_index}
\bP\left( \forall\varphi\in\cP(S_2^{d-1}),\, \left| (P-P_N)\cL_\varphi\right| \le \frac14\|f_\varphi-f^\star\|_{L^2(\bP_X)}^2 + C_{\rm iso}\frac{D_{d,N}+x}{N}\right)\geq 1-4\exp(-x),
\end{align}where $\cL_\varphi:(\boldsymbol{x},y)\in\cX\times\bR \mapsto (y-f_\varphi(\boldsymbol{x}))^2 - (y-f^\star(\boldsymbol{x}))^2$.
\end{Lemma}

\beginproof
Since $Y=f^\star(X)+\xi$, for every $\varphi\in\cP(S_2^{d-1})$,
\begin{align*}
\cL_\varphi(X,Y)=(Y-f_\varphi(X))^2-(Y-f^\star(X))^2=(f_\varphi(X)-f^\star(X))^2-2\xi(f_\varphi(X)-f^\star(X)).
\end{align*}
Using $\mathbb E[\xi\mid X]=0$, we get $P\cL_\varphi=\|f_\varphi-f^\star\|_{L^2(\bP_X)}^2$. Put $U=8B_\sigma^2+8B_\sigma B_\xi$ and $V=2B_\sigma+2B_\xi$. Since $|f_\varphi-f^\star|\le2B_\sigma$, one has $|\cL_\varphi|\le U/2$, $|\cL_\varphi-P\cL_\varphi|\le U$, and, whenever $\|f_\varphi-f^\star\|_{L^2(\bP_X)}\le r$, $P\cL_\varphi^2\le V^2r^2$. Let $C_{\mathcal L}\ge1$ be such that, for every $r>0$, symmetrization and Proposition~\ref{prop:parametric-rademacher-general-ie} give
\begin{align}\label{eq:upper_Rademacher_excess_risk_single_index}
\mathbb E\sup\{|(P-P_N)\cL_\varphi|:\cL_\varphi\in\mathcal L(r)\}\le C_{\mathcal L}r\left(\sqrt{\frac{D_{d,N}}{N}}+\frac{D_{d,N}}{N}\right).
\end{align}
Set $C_1=2C_{\mathcal L}+\sqrt2V+\frac{4U}{3}$, $C_2=(256C_1)^2,$ and $ C_{\rm iso}=\max\{U,C_2/64\}.$
For $r>0$, write $Z_r:=\sup\{|(P-P_N)\cL_\varphi|:\cL_\varphi\in\mathcal L(r)\}$. If $(D_{d,N}+x)/N\ge1$, then $|(P-P_N)\cL_\varphi|\le U\le C_{\rm iso}(D_{d,N}+x)/N$, and \eqref{eq:RIP_single_index} follows. We therefore assume $(D_{d,N}+x)/N<1$. Then $D_{d,N}/N\le\sqrt{D_{d,N}/N}$, and the choice of $C_{\mathcal L}$ gives $\mathbb E Z_r\le C_{\mathcal L}r\sqrt{D_{d,N}/N}$. Applying Bousquet's version of Talagrand's inequality to the symmetric class $\mathcal L(r)\cup(-\mathcal L(r))$, whose centered envelope is bounded by $U$ and whose variance proxy is bounded by $V^2r^2$, yields that, for every $u\ge1$, with probability at least $1-e^{-u}$, $Z_r\le \mathbb E Z_r+\sqrt{\frac{2u}{N}\big(V^2r^2+2U\mathbb E Z_r\big)}+\frac{Uu}{3N}.$
Using $\sqrt{a+b}\le\sqrt a+\sqrt b$ and $2\sqrt{ab}\le a+b$, this implies that,
\begin{align*}
\bP\left( Z_r\le C_1\left[r\sqrt{\frac{D_{d,N}}{N}}+r\sqrt{\frac{u}{N}}+\frac{u}{N}\right]\right)\geq 1-\exp(-u).
\end{align*}
Let $r_0^2=C_2\frac{D_{d,N}+x}{N}$, $r_j=2^jr_0$, $u_j=x+j+1, j\ge0.$
Applying the last display with $r=r_j$ and $u=u_j$, and taking a union bound, gives an event $\mathcal E_x$ satisfying $\mathbb P(\mathcal E_x)\ge1-\sum_{j=0}^\infty e^{-u_j}\ge1-4e^{-x}$. On $\mathcal E_x$, for every $j\ge0$,
\begin{align*}
Z_{r_j}\le C_1\left[r_j\sqrt{\frac{D_{d,N}}{N}}+r_j\sqrt{\frac{x+j+1}{N}}+\frac{x+j+1}{N}\right].
\end{align*}
Since $x\ge1$ and $D_{d,N}\ge1$, one has $x+j+1\le4^j(D_{d,N}+x)$ for every $j\ge0$. Because $r_j^2=4^jC_2(D_{d,N}+x)/N$, the three terms on the right-hand side are respectively bounded by $C_1C_2^{-1/2}r_j^2$, $C_1C_2^{-1/2}r_j^2$, and $C_1C_2^{-1}r_j^2$. The choice $C_2=(256C_1)^2$ gives for any $j\geq 0$, $Z_{r_j}\le\frac1{64}r_j^2$
on $\mathcal E_x$. Now fix $\varphi\in\cP(S_2^{d-1})$. If $\|f_\varphi-f^\star\|_{L^2(\bP_X)}\le r_0$, then $\cL_\varphi\in\mathcal L(r_0)$, and therefore $|(P-P_N)\cL_\varphi|\le\frac1{64}r_0^2\le C_{\rm iso}\frac{D_{d,N}+x}{N}.$
If $r_{j-1}<\|f_\varphi-f^\star\|_{L^2(\bP_X)}\le r_j$ for some $j\ge1$, then $\cL_\varphi\in\mathcal L(r_j)$ and $r_j^2\le4\|f_\varphi-f^\star\|_{L^2(\bP_X)}^2$, so $|(P-P_N)\cL_\varphi|\le\frac1{64}r_j^2\le\frac1{16}\|f_\varphi-f^\star\|_{L^2(\bP_X)}^2.$
Thus, on $\mathcal E_x$, inequality \eqref{eq:RIP_single_index} holds simultaneously for all $\varphi\in\cP(S_2^{d-1})$ with the constant $C_{\rm iso}$ fixed above.
\endproof

\begin{Lemma}
\label{lem:smoothed-dirac-general-ie}
Grant Assumption~\ref{ass:fixed-output-general-ie}. There exists a constant $C_{\rm app}\ge1$, depending only on $(L_\sigma,M_\sigma)$, such that, whenever $0\le \lambda d\le1$,
\begin{align*}
\inf_{\varphi\in\mathcal P(S_2^{d-1})} \left\{ \|f_\varphi-f^\star\|_{L^2(\bP_X)}^2+\lambda\Ent_\tau^-(\varphi) \right\} \le C_{\rm app}\psi(\lambda d),\mbox{ where }\psi(t)=t(1+\log(e/t)).
\end{align*}
\end{Lemma}

\beginproof
If $\lambda=0$, take $\varphi=\delta_{\boldsymbol{w}^\star}$. Then $f_\varphi=f^\star$, so the left-hand side is zero.

Assume $\lambda>0$. For $0<\phi\le1$, let $B_S(\boldsymbol{w}_\star,\phi):=\{\boldsymbol{w}\in S_2^{d-1}:\|\boldsymbol{w}-\boldsymbol{w}_\star\|_2\le\phi\}$ and let $\varphi_\phi:=\tau(\cdot\mid B_S(\boldsymbol{w}_\star,\phi))$. If $\boldsymbol{W}_\phi\sim\varphi_\phi$, write $\boldsymbol{\Delta}_\phi=\boldsymbol{W}_\phi-\boldsymbol{w}_\star$. For any orthogonal transformation $O$ such that $O\boldsymbol{w}_\star=\boldsymbol{w}_\star$, the cap $B_S(\boldsymbol{w}_\star,\phi)$ and the uniform measure on it are invariant, which implies $O(\mathbb{E}\boldsymbol{W}_\phi)=\mathbb{E}[O\boldsymbol{W}_\phi]=\mathbb{E}\boldsymbol{W}_\phi$. The only vectors invariant under all such rotations are scalar multiples of $\boldsymbol{w}_\star$, so $\mathbb{E}\boldsymbol{W}_\phi=c\boldsymbol{w}_\star$ for some scalar $c$. Consequently, $\mathbb{E}\boldsymbol{\Delta}_\phi=\mathbb{E}\boldsymbol{W}_\phi-\boldsymbol{w}_\star=(c-1)\boldsymbol{w}_\star$. By defining $\alpha_\phi=1-c$, we can write $\mathbb{E}\boldsymbol{\Delta}_\phi=-\alpha_\phi \boldsymbol{w}_\star$. To explicitly find $c$, taking the inner product of both sides of $\mathbb{E}\boldsymbol{W}_\phi=c\boldsymbol{w}_\star$ with $\boldsymbol{w}_\star$ yields $\langle \mathbb{E}\boldsymbol{W}_\phi,\boldsymbol{w}_\star\rangle=c\|\boldsymbol{w}_\star\|_2^2=c$. Since the expectation is linear, this implies $c=\mathbb{E}\langle \boldsymbol{W}_\phi,\boldsymbol{w}_\star\rangle$. Substituting this back into the definition of $\alpha_\phi$ gives $\alpha_\phi=1-\mathbb{E}\langle \boldsymbol{W}_\phi,\boldsymbol{w}_\star\rangle$. Since $\boldsymbol{W}_\phi,\boldsymbol{w}_\star\in S_2^{d-1}$, expanding the squared distance gives $\mathbb{E}\|\boldsymbol{W}_\phi-\boldsymbol{w}_\star\|_2^2=2-2\mathbb{E}\langle \boldsymbol{W}_\phi,\boldsymbol{w}_\star\rangle=2\alpha_\phi$, which immediately establishes $\alpha_\phi=\frac{1}{2}\mathbb{E}\|\boldsymbol{W}_\phi-\boldsymbol{w}_\star\|_2^2\ge0$. Finally, since $\|\boldsymbol{W}_\phi-\boldsymbol{w}_\star\|_2\le\phi$ almost surely on the support, we directly obtain $\alpha_\phi\le\phi^2/2$.
For fixed $\boldsymbol{x}$,
\begin{align*}
\sigma(\langle \boldsymbol{W}_\phi,\boldsymbol{x}\rangle)-\sigma(\langle \boldsymbol{w}_\star,\boldsymbol{x}\rangle) = \sigma'(\langle \boldsymbol{w}_\star,\boldsymbol{x}\rangle)\langle\boldsymbol{\Delta}_\phi,\boldsymbol{x}\rangle + \int_0^1(1-t)\sigma''(\langle \boldsymbol{w}_\star+t\boldsymbol{\Delta}_\phi,\boldsymbol{x}\rangle)\langle\boldsymbol{\Delta}_\phi,\boldsymbol{x}\rangle^2\mathrm{d}t.
\end{align*}
Taking expectation in $\boldsymbol{W}_\phi$, then the $L^2(\bP_X)$ norm in $X$, gives
\begin{align*}
\left\|\mathbb E\left[\sigma'(\langle \boldsymbol{w}_\star,X\rangle)\langle\boldsymbol{\Delta}_\phi,X\rangle\right]\right\|_{L^2(\bP_X)} \le L_\sigma\|\mathbb E\boldsymbol{\Delta}_\phi\|_2 \le \frac{L_\sigma}{2}\phi^2.
\end{align*}
For the second-order term, Jensen's inequality and $\mathbb E\langle \boldsymbol{u},X\rangle^4=3\|\boldsymbol{u}\|_2^4$ give
\begin{align*}
\left\| \mathbb E\int_0^1(1-t)\sigma''(\langle \boldsymbol{w}_\star+t\boldsymbol{\Delta}_\phi,X\rangle)\langle\boldsymbol{\Delta}_\phi,X\rangle^2\mathrm{d}t \right\|_{L^2(\bP_X)} \le \frac{\sqrt3M_\sigma}{2}\mathbb E\|\boldsymbol{\Delta}_\phi\|_2^2 \le \frac{\sqrt3M_\sigma}{2}\phi^2.
\end{align*}
Squaring the above directly, we get $\|f_{\varphi_\phi}-f^\star\|_{L^2(\bP_X)}^2\le C\phi^4.$
Moreover, $\Ent_\tau^-(\varphi_\phi)=-\log\tau(B_S(\boldsymbol{w}_\star,\phi)).$
A standard spherical cap lower bound gives, for $0<\phi\le1$, $\tau(B_S(\boldsymbol{w}_\star,\phi))\ge\left(\frac{\phi}{C_{\rm cap}}\right)^{d-1}$
with a universal $C_{\rm cap}>1$. Hence $\Ent_\tau^-(\varphi_\phi)\le(d-1)\log\frac{C_{\rm cap}}{\phi}.$
Choosing $\phi=(\lambda d)^{1/4}\le1$ completes the proof.
\endproof

\begin{Theoremrewrite}
{thm:fixed-output-general-ie}
Grant Assumption~\ref{ass:fixed-output-general-ie}. There exist constants $c_0\in(0,1)$, $C\ge1$, and $C_0\ge e$, depending only on $(\mathrm{IE}(\sigma),B_\sigma,L_\sigma,M_\sigma,T_\sigma,B_\xi,|b_{\mathrm{IE}(\sigma)}|^{-1})$, such that the following holds. Let $x\ge1$, $N\ge2$, and let $\lambda\ge 0$ satisfy $\lambda d\le c_0$. Then, with probability at least $1-4e^{-x}$,
\begin{align}
\label{eq:objective_single_index_estimation_error}
\|f_{\hat\varphi_\lambda}-f^\star\|_{L^2(\bP_X)}^2 + \lambda\Ent_\tau^-(\hat\varphi_\lambda) \le C\left[ \frac{d\log(C_0dN)+x}{N} + \psi(\lambda d) \right],
\end{align}and
\begin{align}\label{eq:proof_estimation_parameter_single}
    \norm{\bE[\hat{\boldsymbol{W}}^{\otimes \mathrm{IE}(\sigma)}|(X_i,Y_i)_{i=1}^N]-(\boldsymbol{w}^*)^{\otimes \mathrm{IE}(\sigma)}}_F^2\leq \frac{\mathrm{IE}(\sigma)!}{b_{\mathrm{IE}(\sigma)}^2}C\left[ \frac{d\log(C_0dN)+x}{N} + \psi(\lambda d) \right].
\end{align}
\end{Theoremrewrite}

\beginproof
Existence of $\hat\varphi_\lambda$ follows from Lemma~\ref{lem:existence-general-ie}. Let $\ell_\varphi(\boldsymbol{x},y):=(y-f_\varphi(\boldsymbol{x}))^2,$ and $\ell_\star(\boldsymbol{x},y):=(y-f^\star(\boldsymbol{x}))^2.$
We work on the random event provided by Lemma~\ref{lem:localized-isomorphism-general-ie}. By optimality of $\hat\varphi_\lambda$, the following inequality holds. If $\lambda\geq 0$, it holds for every $\varphi$ with finite entropy: $P_N\ell_{\hat\varphi_\lambda}+\lambda\Ent_\tau^-(\hat\varphi_\lambda) \le P_N\ell_\varphi+\lambda\Ent_\tau^-(\varphi).$
Subtract $P_N\ell_\star$ from both sides. Let $\Delta_N = C_{\mathrm{iso}}\frac{D_{d,N}+x}{N}$. Since $P(\ell_\varphi-\ell_\star)=\|f_\varphi-f^\star\|_{L^2(\bP_X)}^2$, there holds
\begin{align*}
&\|f_{\hat\varphi_\lambda}-f^\star\|_{L^2(\bP_X)}^2 +\lambda\Ent_\tau^-(\hat\varphi_\lambda)\\
&\le P_N(\ell_{\hat\varphi_\lambda}-\ell_\star) +\lambda\Ent_\tau^-(\hat\varphi_\lambda) +\frac14\|f_{\hat\varphi_\lambda}-f^\star\|_{L^2(\bP_X)}^2+\Delta_N\\
&\le P_N(\ell_\varphi-\ell_\star) +\lambda\Ent_\tau^-(\varphi) +\frac14\|f_{\hat\varphi_\lambda}-f^\star\|_{L^2(\bP_X)}^2+\Delta_N\\
&\le P(\ell_\varphi-\ell_\star) +\frac14\|f_\varphi-f^\star\|_{L^2(\bP_X)}^2 +\lambda\Ent_\tau^-(\varphi) +\frac14\|f_{\hat\varphi_\lambda}-f^\star\|_{L^2(\bP_X)}^2+2\Delta_N\\
&= \frac54\|f_\varphi-f^\star\|_{L^2(\bP_X)}^2 +\lambda\Ent_\tau^-(\varphi) +\frac14\|f_{\hat\varphi_\lambda}-f^\star\|_{L^2(\bP_X)}^2+2\Delta_N.
\end{align*}
Therefore,
\begin{align*}
\|f_{\hat\varphi_\lambda}-f^\star\|_{L^2(\bP_X)}^2 + \lambda\Ent_\tau^-(\hat\varphi_\lambda) \lesssim \left(A_\lambda+\Delta_N\right),\mbox{ where }
A_\lambda:= \inf_{\varphi} \left\{ \|f_\varphi-f^\star\|_{L^2(\bP_X)}^2+\lambda\Ent_\tau^-(\varphi) \right\}.
\end{align*}
By Lemma~\ref{lem:smoothed-dirac-general-ie}, if $\lambda d\le c_0\le1$, then $A_\lambda \le C_{\rm app}\psi(\lambda d).$
Using $D_{d,N}=d\log(edN)\le d\log(C_0dN)$ after increasing $C_0$, proves \eqref{eq:objective_single_index_estimation_error}. Finally, \eqref{eq:proof_estimation_parameter_single} follows from \eqref{eq:objective_single_index_estimation_error} and \eqref{eq:tensor-localization}.
\endproof

\begin{Corollaryrewrite}{prop:concentration_Levy_Milman}
Grant Assumption~\ref{ass:fixed-output-general-ie}. Let $\varepsilon = \frac{2\sqrt{\mathrm{IE}(\sigma)!}}{\kappa_{\mathrm{IE}(\sigma)}|b_{\mathrm{IE}(\sigma)}|} r_*$. With $\bP^{\otimes N}$ probability at least $1-4\exp(-x)$, the following hold.
\begin{enumerate}
    \item If $\mathrm{IE}(\sigma)$ is odd. For any $1$-Lipschitz function $F:S_2^{d-1}\to\bR$ (with respect to $\|\cdot\|_2$),
    \begin{align*}
        \forall t > 2\sqrt{2\varepsilon},\quad  \hat\varphi_\lambda\left( \left| F(\boldsymbol{Z}) - \bE_{Z\sim \hat\varphi_\lambda}[F(\boldsymbol{Z})] \right| \geq t \right)\leq \frac{8\varepsilon}{t^2}.
    \end{align*}Moreover, for any $\rho>0$, $\hat\varphi_\lambda(S_2^{d-1}\backslash B(\boldsymbol{w}_\star;\rho))\leq \frac{r_*}{\rho^2}$, where $B(\boldsymbol{w}_\star;\rho)=\{\boldsymbol{w}:\|\boldsymbol{w}-\boldsymbol{w}_\star\|_2\leq \rho\}$.

    \item If $\mathrm{IE}(\sigma)$ is even. Let $[\boldsymbol{w}] = \{\boldsymbol{w},-\boldsymbol{w}\}$, called  the antipodal equivalence class. Let $\bR\bP^{d-1}=S_2^{d-1}/\{\pm 1\}$ be the equivalence class modulo sign, with quotient map $\pi:S_2^{d-1}\to\bR\bP^{d-1}$. Define $d_{\bR\bP}([\boldsymbol{u}],[\boldsymbol{v}])=\min\{\|\boldsymbol{u}-\boldsymbol{v}\|_2,\|\boldsymbol{u}+\boldsymbol{v}\|_2\}$ be the projective metric. Let $\bar\varphi_\lambda = \pi_\sharp\hat\varphi_\lambda$, then for any $1$-Lipschitz function $G:\bR\bP^{d-1}:\bR$ with respect to metric $d_{\bR\bP}$,
    \begin{align*}
        \forall t>2\sqrt{2\varepsilon},\quad \bar\varphi_\lambda\left( \left| G(Z) - \bE_{Z\sim\bar\varphi_\lambda}[G(Z)] \right| \geq t \right)\leq \frac{8\varepsilon}{t^2}.
    \end{align*}Moreover, for any $\rho>0$, $\bar\varphi_\lambda(\bR\bP^{d-1}\backslash B_{\bR\bP}([\boldsymbol{w}_\star];\rho))\leq \frac{r_*}{\rho^2}$ and $\hat\varphi_\lambda(S_2^{d-1}\backslash (B(\boldsymbol{w}_\star;\rho)\cup B(-\boldsymbol{w}_\star;\rho)))\leq \frac{r_*}{\rho^2}$, where $B_{\bR\bP}([\boldsymbol{w}_\star];\rho) = \{\boldsymbol{w}\in\bR\bP^{d-1}:d_{\bR\bP}(\boldsymbol{w};[\boldsymbol{w}_\star])\leq \rho\}$.
\end{enumerate}Moreover, for $\mu = \frac{1}{2}\delta_{\boldsymbol{w}_\star} + \frac{1}{2}\delta_{-\boldsymbol{w}_\star}$, and $1$-Lipschitz function $G(\boldsymbol{w})=\langle\boldsymbol{w},\boldsymbol{w}_\star\rangle$ with respect to the Euclidean metric, there hold $\bE_{\boldsymbol{Z}\sim\mu}[G(\boldsymbol{Z})]=0$ and $\mu(|G(\boldsymbol{Z})-\bE[G(\boldsymbol{Z})]|\geq 1) = 1$.
\end{Corollaryrewrite}
\beginproof For item~\emph{1.}, we use Lemma~\ref{lemma:levy_milman_concentration} applied to $z=\boldsymbol{w}_\star$. Then the claim follows by Proposition~\ref{prop:m-chaos-localization} and Theorem~\ref{thm:fixed-output-general-ie}. For item~\emph{2.}, we use Lemma~\ref{lemma:levy_milman_concentration} again, to $z=[\boldsymbol{w}_\star]$. The claim follows by Proposition~\ref{prop:m-chaos-localization} and Theorem~\ref{thm:fixed-output-general-ie}. The results on $\hat\varphi_\lambda(S_2^{d-1}\backslash B(\boldsymbol{w}_\star;\rho))$, $\bar\varphi_\lambda(\bR\bP^{d-1}\backslash B_{\bR\bP}([\boldsymbol{w}_\star];\rho))\leq \frac{r_*}{\rho^2}$, and $\hat\varphi_\lambda(S_2^{d-1}\backslash (B(\boldsymbol{w}_\star;\rho)\cup B(-\boldsymbol{w}_\star;\rho)))\leq \frac{r_*}{\rho^2}$ follow from Markov's inequality directly. The last claim follows by the observation that $\int d_{\bR\bP}^2([\boldsymbol{w}],[\boldsymbol{w}_\star])\mathrm{d}\mu(\boldsymbol{w})=0$.
\endproof

\subsection{Proof of Theorem~\ref{thm:multi-index-fixed-output}}\label{sec:proof_thm_multi-index-fixed-output}

\begin{Proposition}
\label{prop:multi-index-localization}
Suppose there exist $\boldsymbol{w}_1^\star,\cdots,\boldsymbol{w}_M^\star\in S_2^{d-1}$ and $a_1^\star,\cdots,a_M^\star>0$ such that $\sum_{j=1}^M a_j^\star =1$. Suppose $\bP_X\sim\cN(\boldsymbol{0}, I_d)$. Then
\begin{equation}\label{eq:tensor_loc}
\forall \varphi\in\mathcal P(S_2^{d-1}),\quad \forall 1\leq m\leq M,\,\qquad \|\mathcal T_m(\varphi)-\mathcal T_m(\varphi_\star)\|_F \le \frac{\sqrt{m!}}{|b_m|}\,\|f_\varphi-f^\star\|_{L^2(\bP_X)}.
\end{equation}
Moreover, if $\Delta_\star = \min(1-\langle\boldsymbol{w}_i^\star,\boldsymbol{w}_j^\star\rangle:i\neq j)>0$, there is a scalar constant $C_{\rm loc}$, depending only on $M$, $\Delta_\star^{-1}$, and $\max_{1\le m\le M}|b_m|^{-1}$, such that
\begin{equation}\label{eq:geo_loc}
S_\star(\varphi):=\int_{S_2^{d-1}}\min_{1\le j\le M}\|\boldsymbol{w}-\boldsymbol{w}_j^\star\|_2^2\,\varphi(\mathrm{d}\boldsymbol{w}) \le C_{\rm loc}\|f_\varphi-f^\star\|_{L^2(\bP_X)}.
\end{equation}
\end{Proposition}

\beginproof
The tensor bound \eqref{eq:tensor_loc} follows from the same Hermite--tensor projection identity used in the single-index finite-information-index case, applied to the signed measure $\nu:=\varphi-\varphi_\star$. Namely, by Lemma ~\ref{lemma:proj-ridge-feature}, for every $1\le m\le M$, $\mathrm{Proj}_m(f_\varphi-f^\star)=\frac{b_m}{m!}\int \mathrm{He}_m(\langle \boldsymbol{w},\cdot\rangle)(\varphi-\varphi_\star)(\mathrm{d}\boldsymbol{w}),$
and
\[
\left\|\int \mathrm{He}_m(\langle \boldsymbol{w},\cdot\rangle)\nu(\mathrm{d}\boldsymbol{w})\right\|_{L^2(\bP_X)}^2 = m!\|\mathcal T_m(\nu)\|_F^2.
\]
Thus $\|\mathrm{Proj}_m(f_\varphi-f^\star)\|_{L^2(\bP_X)}=(|b_m|/\sqrt{m!})\|\mathcal T_m(\varphi)-\mathcal T_m(\varphi_\star)\|_F$. Since $\mathrm{Proj}_m$ is an orthogonal projection, \eqref{eq:tensor_loc} follows.

It remains to prove the measure concentration. Define the nonnegative polynomial
\[
Q_\star(\boldsymbol{w}):=\prod_{j=1}^M\bigl(1-\langle \boldsymbol{w},\boldsymbol{w}_j^\star\rangle\bigr),\qquad \boldsymbol{w}\in S_2^{d-1}.
\]
Each factor is nonnegative on the sphere, and $Q_\star(\boldsymbol{w})=0$ if and only if $\boldsymbol{w}\in\{\boldsymbol{w}_1^\star,\ldots,\boldsymbol{w}_M^\star\}$. Let $\delta(\boldsymbol{w}):=\min_j\|\boldsymbol{w}-\boldsymbol{w}_j^\star\|_2^2=2\min_j(1-\langle \boldsymbol{w},\boldsymbol{w}_j^\star\rangle)$. We first claim that
\begin{equation}\label{eq:poly_lower_bound}
\delta(\boldsymbol{w})\le C_{\rm geo}Q_\star(\boldsymbol{w}),\qquad \boldsymbol{w}\in S_2^{d-1},
\end{equation}
for a constant $C_{\rm geo}=C_{\rm geo}(M,\Delta_\star^{-1})$.

Fix $\boldsymbol{w}$, and choose $j=j(\boldsymbol{w})$ attaining the minimum in $\delta(\boldsymbol{w})$. Then $1-\langle \boldsymbol{w},\boldsymbol{w}_j^\star\rangle=\delta(\boldsymbol{w})/2$. Suppose first that $\delta(\boldsymbol{w})\le \Delta_\star^2/4$. For every $\ell\neq j$,
\begin{align*}
1-\langle \boldsymbol{w},\boldsymbol{w}_\ell^\star\rangle
&= 1-\langle \boldsymbol{w}_j^\star,\boldsymbol{w}_\ell^\star\rangle+\langle \boldsymbol{w}_j^\star-\boldsymbol{w},\boldsymbol{w}_\ell^\star\rangle \ge \Delta_\star-\|\boldsymbol{w}-\boldsymbol{w}_j^\star\|_2 \ge \Delta_\star/2.
\end{align*}
Here the first inequality uses $1-\langle \boldsymbol{w}_j^\star,\boldsymbol{w}_\ell^\star\rangle\ge\Delta_\star$ and Cauchy's inequality $\langle \boldsymbol{w}_j^\star-\boldsymbol{w},\boldsymbol{w}_\ell^\star\rangle\ge-\|\boldsymbol{w}_j^\star-\boldsymbol{w}\|_2\|\boldsymbol{w}_\ell^\star\|_2=-\|\boldsymbol{w}_j^\star-\boldsymbol{w}\|_2$. Hence, in this case,
\[
Q_\star(\boldsymbol{w})\ge \frac{\delta(\boldsymbol{w})}2\left(\frac{\Delta_\star}{2}\right)^{M-1}.
\]
If instead $\delta(\boldsymbol{w})>\Delta_\star^2/4$, then every factor satisfies $1-\langle \boldsymbol{w},\boldsymbol{w}_i^\star\rangle\ge\delta(\boldsymbol{w})/2>\Delta_\star^2/8$, so $Q_\star(\boldsymbol{w})\ge(\Delta_\star^2/8)^M$. Since $\delta(\boldsymbol{w})\le4$ on the unit sphere, \eqref{eq:poly_lower_bound} holds, for example, with
\[
C_{\rm geo}:=\max\left\{2(2/\Delta_\star)^{M-1},\,4(8/\Delta_\star^2)^M\right\}.
\]

Next we expand $Q_\star$. For $J=\{j_1,\ldots,j_m\}\subseteq[M]$, let
$
W_J^\star:=\operatorname{Sym}(\boldsymbol{w}_{j_1}^\star\otimes\cdots\otimes \boldsymbol{w}_{j_m}^\star)\in\operatorname{Sym}^m(\mathbb R^d),
$
where $\operatorname{Sym}$ denotes symmetrization. Since $\boldsymbol{w}^{\otimes m}$ is symmetric,
$
\langle W_J^\star,\boldsymbol{w}^{\otimes m}\rangle_F=\prod_{j\in J}\langle \boldsymbol{w},\boldsymbol{w}_j^\star\rangle.
$
Expanding the product defining $Q_\star$ gives
\begin{align*}
Q_\star(\boldsymbol{w}) &= \prod_{j=1}^M \bigl(1-\langle \boldsymbol{w},\boldsymbol{w}_j^\star\rangle\bigr) = 1 + \sum_{m=1}^M (-1)^m \sum_{\substack{J\subseteq[M]\\ |J|=m}} \prod_{j\in J} \langle \boldsymbol{w},\boldsymbol{w}_j^\star\rangle = 1 + \sum_{m=1}^M \left\langle (-1)^m \sum_{\substack{J\subseteq[M]\\ |J|=m}} W_J^\star, \boldsymbol{w}^{\otimes m}\right\rangle_F.
\end{align*}
Let us define
\begin{equation}\label{eq:coeff_tensor_Am}
A_m := (-1)^m \sum_{\substack{J\subseteq[M]\\ |J|=m}} W_J^\star,
\end{equation}
so that $Q_\star(\boldsymbol{w}) = 1 + \sum_{m=1}^M \langle A_m, \boldsymbol{w}^{\otimes m}\rangle_F$. The tensors $A_m$ are deterministic coefficient tensors. They enter the proof only through their scalar norms. Since symmetrization is a contraction and the target vectors have unit norm, $\|W_J^\star\|_F\le1$. By the triangle inequality, we have $\|A_m\|_F\le\binom Mm$.

Because $\varphi_\star$ is supported on the zeros of $Q_\star$, $\int Q_\star \mathrm{d}\varphi_\star=0$. Therefore $\int Q_\star\,\mathrm{d}\varphi=\int Q_\star\,\mathrm{d}(\varphi-\varphi_\star).$
Thus, substituting the expansion of $Q_\star$ \eqref{eq:coeff_tensor_Am} into the integral, we obtain
\begin{align*}
\int Q_\star\,\mathrm{d}\varphi
&= \int \left( 1 + \sum_{m=1}^M \langle A_m, \boldsymbol{w}^{\otimes m}\rangle_F \right) \mathrm{d}(\varphi-\varphi_\star)(\boldsymbol{w}) = \int 1 \,\mathrm{d}(\varphi-\varphi_\star) + \sum_{m=1}^M \left\langle A_m, \int \boldsymbol{w}^{\otimes m} \mathrm{d}(\varphi-\varphi_\star)(\boldsymbol{w}) \right\rangle_F.
\end{align*}
The constant term disappears because both $\varphi$ and $\varphi_\star$ are probability measures: $\int 1\,\mathrm{d}(\varphi-\varphi_\star)=\varphi(S_2^{d-1})-\varphi_\star(S_2^{d-1})=0$. Hence,
\[
\int Q_\star\,\mathrm{d}\varphi
=
\sum_{m=1}^M\left\langle A_m,\mathcal T_m(\varphi)-\mathcal T_m(\varphi_\star)\right\rangle_F.
\]
By Cauchy's inequality and \eqref{eq:tensor_loc},
\[
\int Q_\star\,\mathrm{d}\varphi
\le
\sum_{m=1}^M \|A_m\|_F\,\|\mathcal T_m(\varphi)-\mathcal T_m(\varphi_\star)\|_F
\le C_{\rm tens}\|f_\varphi-f^\star\|_{L^2(\bP_X)},
\]
where
\[
C_{\rm tens}:=\sum_{m=1}^M\frac{\sqrt{m!}\,\|A_m\|_F}{|b_m|}
\le
\sum_{m=1}^M\frac{\sqrt{m!}\binom Mm}{|b_m|}.
\]
This is a scalar constant depending only on $M$ and the Hermite coefficients $b_1,\ldots,b_M$. Finally, by \eqref{eq:poly_lower_bound},
\[
S_\star(\varphi)=\int\delta(\boldsymbol{w})\varphi(\mathrm{d}\boldsymbol{w})\le C_{\rm geo}\int Q_\star(\boldsymbol{w})\varphi(\mathrm{d}\boldsymbol{w})\le C_{\rm geo}C_{\rm tens}\|f_\varphi-f^\star\|_{L^2(\bP_X)}.
\]
This proves \eqref{eq:geo_loc} with $C_{\rm loc}:=C_{\rm geo}C_{\rm tens}$.
\endproof

\begin{Proposition}
\label{prop:multi-index-voronoi-concentration}
Let $V_1,\ldots,V_M$ be the Voronoi partition generated by $\boldsymbol{w}_1^\star,\ldots,\boldsymbol{w}_M^\star$, that is, $V_j = \{\boldsymbol{w}\in S_2^{d-1}: j=\min(\argmin_{1\leq\ell\leq M}\|\boldsymbol{w}-\boldsymbol{w}_\ell^\star\|_2)\}$. There is a constant $C_{\rm vor}$, depending only on $M$, $\Delta_\star^{-1}$, and $\max_{1\le k\le M}|b_k|^{-1}$, such that for every $\varphi\in\mathcal P(S_2^{d-1})$,
\begin{align*}
&\max_{1\le j\le M}\bigg|\int_{V_j}\varphi(\mathrm{d}\boldsymbol{w})-a_j^\star\bigg|+\max_{1\le j\le M}\bigg\|\int_{V_j}(\boldsymbol{w}-\boldsymbol{w}_j^\star)\varphi(\mathrm{d}\boldsymbol{w})\bigg\|_2 + \max_{1\leq j\leq M}\int_{V_j}\|\vw-\vw_j^\star\|_2^2\varphi(\mathrm{d}\vw)\\
&\qquad\le C_{\rm vor}\|f_\varphi-f^\star\|_{L^2(\bP_X)}.
\end{align*}
\end{Proposition}

\beginproof
For every polynomial $q$ on $\mathbb R^d$ of degree at most $M$, Proposition~\ref{prop:multi-index-localization} gives
\begin{align*}
&\bigg|\int_{S_2^{d-1}}q(\boldsymbol{w})\varphi(\mathrm{d}\boldsymbol{w})-\int_{S_2^{d-1}}q(\boldsymbol{w})\varphi^\star(\mathrm{d}\boldsymbol{w})\bigg|\le C_q\|f_\varphi-f^\star\|_{L^2(\bP_X)}.
\end{align*}
Taylor's formula for $q(\boldsymbol{w}) = q(\boldsymbol{w}_i^\star) + \langle\boldsymbol{w}-\boldsymbol{w}_i^\star,\nabla q(\boldsymbol{w}_i^\star)\rangle+R_i(\boldsymbol{w})$ on the cells $V_1,\ldots,V_M$ gives
\begin{align*}
    &\int_{S_2^{d-1}}q(\boldsymbol{w})\mathrm{d}\varphi(\boldsymbol{w}) - \int_{S_2^{d-1}}q(\boldsymbol{w})\mathrm{d}\varphi^\star(\boldsymbol{w})\\
    &= \sum_{i=1}^M\left( \int_{V_i}\mathrm{d}\varphi(\boldsymbol{w}) - a_i^\star \right)q(\boldsymbol{w}_i^\star) + \sum_{i=1}^M\bigg\langle \int_{V_i}(\boldsymbol{w}-\boldsymbol{w}_i^\star)\mathrm{d}\varphi(\boldsymbol{w}),\nabla q(\boldsymbol{w}_i^\star) \bigg\rangle + \sum_{i=1}^M \int_{V_i}R_i(\boldsymbol{w})\mathrm{d}\varphi(\boldsymbol{w}),
\end{align*}
where $|R_i(\boldsymbol{w})|\leq C_q\|\boldsymbol{w}-\boldsymbol{w}_i^\star\|_2^2$ for some absolute constant $C_q$. Taking absolute value on both sides and using the upper bound from Proposition~\ref{prop:multi-index-localization} together with $\sum_{i=1}^M\int_{V_i}\|\boldsymbol{w}-\boldsymbol{w}_i^\star\|_2^2\mathrm{d}\varphi(\boldsymbol{w})=S_\star(\varphi)\leq C_{\mathrm{loc}}\|f_\varphi - f^\star\|_{L^2(\bP_X)}$ yields
\begin{align}\label{eq:taylor_for_polynomial_q}
&\bigg|\sum_{i=1}^M\bigg(\int_{V_i}\varphi(\mathrm{d}\boldsymbol{w})-a_i^\star\bigg)q(\boldsymbol{w}_i^\star)+\sum_{i=1}^M\bigg\langle\int_{V_i}(\boldsymbol{w}-\boldsymbol{w}_i^\star)\varphi(\mathrm{d}\boldsymbol{w}),\nabla q(\boldsymbol{w}_i^\star)\bigg\rangle\bigg|\le C_q\|f_\varphi-f^\star\|_{L^2(\bP_X)}.
\end{align}
\begin{enumerate}
    \item We first prove the concentration for the local barycenters. Fix $1\le j\le M$ and $\boldsymbol{e}\in S_2^{d-1}$. Apply the last display to the polynomial
\begin{align*}
&q(\boldsymbol{w})=\langle \boldsymbol{w}-\boldsymbol{w}_j^\star,\boldsymbol{e}\rangle\prod_{\ell\ne j}\big(1-\langle \boldsymbol{w},\boldsymbol{w}_\ell^\star\rangle\big).
\end{align*}
In \eqref{eq:taylor_for_polynomial_q}, $q(\boldsymbol{w}_i^\star)=0$ for any $1\leq i\leq M$, hence the first term of the left-hand-side of \eqref{eq:taylor_for_polynomial_q} vanishes, hence $|\sum_{i=1}^M\langle\int_{V_i}(\boldsymbol{w}-\boldsymbol{w}_i^\star)\varphi(\mathrm{d}\boldsymbol{w}),\nabla q(\boldsymbol{w}_i^\star)\rangle|\leq C_q\|f_\varphi - f^\star\|_{L^2(\bP_X)}$. We separate the sum
\begin{align*}
    &\sum_{i=1}^M\langle\int_{V_i}(\boldsymbol{w}-\boldsymbol{w}_i^\star)\varphi(\mathrm{d}\boldsymbol{w}),\nabla q(\boldsymbol{w}_i^\star)\rangle \\
    &= \bigg\langle\int_{V_j}(\boldsymbol{w}-\boldsymbol{w}_j^\star)\mathrm{d}\varphi(\boldsymbol{w}),\nabla q(\boldsymbol{w}_j^\star)\bigg\rangle + \sum_{i\neq j}\bigg\langle \int_{V_i}(\boldsymbol{w}-\boldsymbol{w}_i^\star)\mathrm{d}\varphi(\boldsymbol{w}),\nabla q(\boldsymbol{w}_i^\star)\bigg\rangle.
\end{align*}
For the first term, by computation, $\nabla q(\boldsymbol{w}_j^\star) = \boldsymbol{e}\prod_{\ell\neq j}(1-\langle\boldsymbol{w}_j^\star,\boldsymbol{w}_\ell^\star\rangle)$, together with the definition of $\Delta_\star$ in \eqref{eq:target_separation},
\begin{align*}
    &\left|\bigg\langle\int_{V_j}(\boldsymbol{w}-\boldsymbol{w}_j^\star)\mathrm{d}\varphi(\boldsymbol{w}),\nabla q(\boldsymbol{w}_j^\star)\bigg\rangle\right| \geq \Delta_\star^{M-1}\left|\bigg\langle \int_{V_j}(\boldsymbol{w}-\boldsymbol{w}_j^\star)\mathrm{d}\varphi(\boldsymbol{w}) ,\boldsymbol{e}\bigg\rangle\right|.
\end{align*}
Next we deal with the second term.
By computation, $\nabla q(\boldsymbol{w}_i^\star) = -\langle\boldsymbol{w}_i^\star-\boldsymbol{w}_j^\star,\boldsymbol{e}\rangle\prod_{\ell\neq j,i}(1-\langle\boldsymbol{w}_i^\star,\boldsymbol{w}_\ell^\star\rangle)\boldsymbol{w}_i^\star$; hence
\begin{align*}
&\bigg|\bigg\langle\int_{V_i}(\boldsymbol{w}-\boldsymbol{w}_i^\star)\varphi(\mathrm{d}\boldsymbol{w}),\nabla q(\boldsymbol{w}_i^\star)\bigg\rangle\bigg| =\bigg|\langle\boldsymbol{w}_i^\star-\boldsymbol{w}_j^\star,\boldsymbol{e}\rangle\prod_{\ell\ne j,i}\bigl(1-\langle\boldsymbol{w}_i^\star,\boldsymbol{w}_\ell^\star\rangle\bigr)\int_{V_i}\langle\boldsymbol{w}-\boldsymbol{w}_i^\star,\boldsymbol{w}_i^\star\rangle\varphi(\mathrm{d}\boldsymbol{w})\bigg|\\
&=\frac{1}{2}\bigg|\langle\boldsymbol{w}_i^\star-\boldsymbol{w}_j^\star,\boldsymbol{e}\rangle\prod_{\ell\ne j,i}\bigl(1-\langle\boldsymbol{w}_i^\star,\boldsymbol{w}_\ell^\star\rangle\bigr)\bigg|\int_{V_i}\|\boldsymbol{w}-\boldsymbol{w}_i^\star\|_2^2\varphi(\mathrm{d}\boldsymbol{w})\le C_q\int_{V_i}\|\boldsymbol{w}-\boldsymbol{w}_i^\star\|_2^2\varphi(\mathrm{d}\boldsymbol{w}),
\end{align*}where we have used the identity $\langle\boldsymbol{w}_i^\star,\boldsymbol{w}-\boldsymbol{w}_i^\star\rangle = -\frac{1}{2}\|\boldsymbol{w}-\boldsymbol{w}_i^\star\|_2^2$ in the second equality. Summing over $i\neq j$, we obtain $|\sum_{i\neq j}\langle \int_{V_i}(\boldsymbol{w}-\boldsymbol{w}_i^\star)\mathrm{d}\varphi(\boldsymbol{w}),\nabla q(\boldsymbol{w}_i^\star)\rangle|\leq C_q\sum_{i\neq j}\int_{V_i}\|\boldsymbol{w}-\boldsymbol{w}_i^\star\|_2^2\mathrm{d}\varphi(\boldsymbol{w})\leq C\|f_\varphi - f^\star\|_{L^2(\bP_X)}^2$.

Take $\ve = -\vw_j^\star$; then $\langle\int_{V_j}(\vw-\vw_j^\star)\varphi(\mathrm d\vw),\ve\rangle = \frac{1}{2}\int_{V_j}\|\vw-\vw_j^\star\|_2^2\varphi(\mathrm d\vw)$, hence by triangular inequality, $\int_{V_j}\|\boldsymbol{w}-\boldsymbol{w}_j^\star\|_2^2\varphi(\mathrm{d}\boldsymbol{w})\le C_{\rm vor}\|f_\varphi-f^\star\|_{L^2(\bP_X)}.$ Taking supreme over $\ve\in S_2^{d-1}$ together with the triangular inequality yields $\|\int_{V_j}\boldsymbol{w}-\boldsymbol{w}_j^\star\varphi(\mathrm{d}\boldsymbol{w})\|_2\le C_{\rm vor}\|f_\varphi-f^\star\|_{L^2(\bP_X)}.$

\item It remains to control the mass. Apply the same display to $q(\boldsymbol{w})=\prod_{\ell\ne j}\big(1-\langle \boldsymbol{w},\boldsymbol{w}_\ell^\star\rangle\big).$
For $i\ne j$, $q(\boldsymbol{w}_i^\star)=0$ and for $i=j$, $q(\boldsymbol{w}_i^\star) = \prod_{\ell\neq j}(1-\langle\boldsymbol{w}_j^\star,\boldsymbol{w}_\ell^\star\rangle)$. Therefore, the first term of the left-hand-side of \eqref{eq:taylor_for_polynomial_q} is $(\int_{V_j}\mathrm{d}\varphi(\boldsymbol{w})-a_j^\star)\prod_{\ell\neq j}(1-\langle\boldsymbol{w}_j^\star,\boldsymbol{w}_\ell^\star\rangle)$. For $i\neq j$, $\nabla q(\boldsymbol{w}_i^\star) = -\boldsymbol{w}_i^\star\prod_{\ell\neq j,i}(1-\langle\boldsymbol{w}_i^\star,\boldsymbol{w}_\ell^\star\rangle)$. Using again that $\langle\boldsymbol{w}-\boldsymbol{w}_i^\star,\boldsymbol{w}_i^\star\rangle = -\frac{1}{2}\|\boldsymbol{w}-\boldsymbol{w}_i^\star\|_2^2$ and taking sum over $i\neq j$, $|\sum_{i\neq j}\langle\int_{V_i}(\boldsymbol{w}-\boldsymbol{w}_i^\star)\mathrm{d}\varphi(\boldsymbol{w}),\nabla q(\boldsymbol{w}_i^\star)\rangle|\leq C_q\|f_\varphi - f^\star\|_{L^2(\bP_X)}$. Therefore, by triangular inequality and \eqref{eq:taylor_for_polynomial_q},
\begin{align*}
&\bigg|\bigg(\int_{V_j}\varphi(\mathrm{d}\boldsymbol{w})-a_j^\star\bigg)\prod_{\ell\ne j}\big(1-\langle \boldsymbol{w}_j^\star,\boldsymbol{w}_\ell^\star\rangle\big)+\bigg\langle\int_{V_j}(\boldsymbol{w}-\boldsymbol{w}_j^\star)\varphi(\mathrm{d}\boldsymbol{w}),\nabla q(\boldsymbol{w}_j^\star)\bigg\rangle\bigg| \lesssim\|f_\varphi-f^\star\|_{L^2(\bP_X)}.
\end{align*}
Using again $\prod_{\ell\ne j}\big(1-\langle \boldsymbol{w}_j^\star,\boldsymbol{w}_\ell^\star\rangle\big)\ge \Delta_\star^{M-1}$ together with the conclusion of item~\emph{1.} proves the proposition.
\end{enumerate}
\endproof

\begin{Proposition}
\label{prop:multi-index-rademacher}
Let $D_{d,N}:=d\log(edN)$ and $\mathfrak D_{M,d,N}:=Md+D_{d,N}$. For $r>0$, define
\begin{align*}
&(\mathcal{F}-f^\star)\cap B_{L^2(\bP_X)}(0;r):=\{f_\varphi-f^\star:\varphi\in\mathcal P(S_2^{d-1}),\ \|f_\varphi-f^\star\|_{L^2(\bP_X)}\le r\}.
\end{align*}
There is a constant $C_{\rm rad}$, depending only on $M$, $\Delta_\star^{-1}$, $L_\sigma$, $M_\sigma$, $T_\sigma$, and $\max_{1\le k\le M}|b_k|^{-1}$, such that
\begin{align*}
&\mathbb E\sup_{g\in(\mathcal{F}-f^\star)\cap B_{L^2(\bP_X)}(0;r)}\bigg|\frac1N\sum_{i=1}^N\varepsilon_ig(X_i)\bigg|\le C_{\rm rad}r\bigg(\sqrt{\frac{\mathfrak D_{M,d,N}}{N}}+\frac{\mathfrak D_{M,d,N}}{N}\bigg).
\end{align*}
For the localized loss class
\begin{align*}
&\mathcal L(r):=\{(\boldsymbol{x},y)\mapsto (y-f_\varphi(\boldsymbol{x}))^2-(y-f^\star(\boldsymbol{x}))^2:f_\varphi-f^\star\in(\mathcal{F}-f^\star)\cap B_{L^2(\bP_X)}(0;r)\},
\end{align*}
there is a constant $C'_{\rm rad}$, depending additionally on $B_\sigma$ and $B_\xi$, such that
\begin{align*}
&\mathbb E\sup_{h\in\mathcal L(r)}\bigg|\frac1N\sum_{i=1}^N\varepsilon_ih(X_i,Y_i)\bigg|\le C'_{\rm rad}r\bigg(\sqrt{\frac{\mathfrak D_{M,d,N}}{N}}+\frac{\mathfrak D_{M,d,N}}{N}\bigg).
\end{align*}
\end{Proposition}

\beginproof
By Proposition~\ref{prop:multi-index-voronoi-concentration}, for every $\varphi$ such that $\|f_\varphi-f^\star\|_{L^2(\bP_X)}\le r$,
\begin{align*}
&\max_{1\le j\le M}\bigg|\int_{V_j}\varphi(\mathrm{d}\boldsymbol{w})-a_j^\star\bigg|+\max_{1\le j\le M}\bigg\|\int_{V_j}(\boldsymbol{w}-\boldsymbol{w}_j^\star)\varphi(\mathrm{d}\boldsymbol{w})\bigg\|_2\le C_{\rm vor}r.
\end{align*}
Set
\begin{align*}
&(T_\varphi^{(1)}\1)(\boldsymbol{x})=\sum_{j=1}^M\bigg(\int_{V_j}\varphi(\mathrm{d}\boldsymbol{w})\bigg)\sigma(\langle\boldsymbol{w}_j^\star,\boldsymbol{x}\rangle)+\sum_{j=1}^M\sigma'(\langle\boldsymbol{w}_j^\star,\boldsymbol{x}\rangle)\bigg\langle\int_{V_j}(\boldsymbol{w}-\boldsymbol{w}_j^\star)\varphi(\mathrm{d}\boldsymbol{w}),\boldsymbol{x}\bigg\rangle.
\end{align*}
Taylor's formula gives, for $\boldsymbol{w}\in V_j$,
\begin{align*}
&\sigma(\langle \boldsymbol{w},\boldsymbol{x}\rangle)-\sigma(\langle \boldsymbol{w}_j^\star,\boldsymbol{x}\rangle)-\sigma'(\langle \boldsymbol{w}_j^\star,\boldsymbol{x}\rangle)\langle \boldsymbol{w}-\boldsymbol{w}_j^\star,\boldsymbol{x}\rangle\\
&=\|\boldsymbol{w}-\boldsymbol{w}_j^\star\|_2^2\int_0^1(1-t)\sigma''(\langle \boldsymbol{w}_j^\star+t(\boldsymbol{w}-\boldsymbol{w}_j^\star),\boldsymbol{x}\rangle)\bigg\langle\frac{\boldsymbol{w}-\boldsymbol{w}_j^\star}{\|\boldsymbol{w}-\boldsymbol{w}_j^\star\|_2},\boldsymbol{x}\bigg\rangle^2\mathrm{d}t,
\end{align*}
with $\frac{\boldsymbol{w}-\boldsymbol{w}_j^\star}{\|\boldsymbol{w}-\boldsymbol{w}_j^\star\|_2}$ replaced by an arbitrary unit vector in the last display when $\boldsymbol{w}=\boldsymbol{w}_j^\star$. Since $\boldsymbol{w}_j^\star+t(\boldsymbol{w}-\boldsymbol{w}_j^\star)\in B_2^d$, the integral in the last display belongs to $\frac12\overline{\operatorname{conv}}(\mathcal Q_1)$, where
\begin{align*}
&\mathcal Q_1:=\{\boldsymbol{x}\mapsto\sigma''(\langle \boldsymbol{u},\boldsymbol{x}\rangle)\langle \boldsymbol{e},\boldsymbol{x}\rangle^2:\|\boldsymbol{u}\|_2\le1,\ \boldsymbol{e}\in S_2^{d-1}\}.
\end{align*}
Integrating over the Voronoi partition yields $(T_\varphi-T_\varphi^{(1)})\1\in \frac{S_\star(\varphi)}2\overline{\operatorname{conv}}(\mathcal Q_1)\subset C_{\mathrm{rad}}r\overline{\conv}(\cQ_1)$ by
Proposition~\ref{prop:multi-index-localization}.
Moreover, since $f^\star(\boldsymbol{x}) = \sum_{j=1}^M a_j^\star\sigma(\langle\boldsymbol{w}_j^\star,\boldsymbol{x}\rangle)$, $T_\varphi^{(1)}\1-f^\star(\boldsymbol{x}) = \sum_{j=1}^M(\int_{V_j}\mathrm{d}\varphi(\boldsymbol{w}) - a_j^\star)\sigma(\langle\boldsymbol{w}_j^\star,\boldsymbol{x}\rangle)+\sum_{j=1}^M\sigma'(\langle\boldsymbol{w}_j^\star,\boldsymbol{x}\rangle)\langle\int_{V_j}(\boldsymbol{w}-\boldsymbol{w}_j^\star)\mathrm{d}\varphi(\boldsymbol{w}),\boldsymbol{x}\rangle$, which implies
\begin{align*}
&T_\varphi^{(1)}\1-f^\star\in \operatorname{span}\{\boldsymbol{x}\mapsto\sigma(\langle \boldsymbol{w}_j^\star,\boldsymbol{x}\rangle),\ \boldsymbol{x}\mapsto\sigma'(\langle \boldsymbol{w}_j^\star,\boldsymbol{x}\rangle)x_\ell:1\le j\le M,\ 1\le \ell\le d\}.
\end{align*}
Thus every $g\in(\mathcal F-f^\star)\cap B_{L^2(\bP_X)}(0;r)$ is contained in the sum of the $C_{\rm rad}r$-ball of an $M(d+1)$-dimensional subspace and $C_{\rm rad}r\,\overline{\operatorname{conv}}(\mathcal Q_1)$. The remaining part of the proof is the same as in that of Proposition~\ref{prop:parametric-rademacher-general-ie}, thus omitted.
\endproof

\begin{Lemma}
\label{lem:multi-index-isomorphism}
Under the assumptions above, there is a constant $C_{\rm iso}$, depending only on $M$, $\Delta_\star^{-1}$, $B_\sigma$, $L_\sigma$, $M_\sigma$, $T_\sigma$, $B_\xi$, and $\max_{1\le k\le M}|b_k|^{-1}$, such that, for every $x\ge1$ and $N\ge2$,
\[
\bP\left(\forall\varphi\in\mathcal P(S_2^{d-1}),\, \left|(P-P_N)\cL_\varphi\right|\le\frac14\|f_\varphi-f^\star\|_{L^2(\bP_X)}^2 + C_{\rm iso}\frac{\mathfrak D_{M,d,N}+x}{N}\right)\geq 1-4\exp(-x).
\]
\end{Lemma}
The proof is exactly the same as in that of Lemma~\ref{lem:localized-isomorphism-general-ie}, thus omitted.


\begin{Lemma}
\label{lem:multi-dirac-competitor}
Grant Assumption~\ref{ass:multi_index}. For any $d\in \N$, there is a constant $C_{\rm app}$ depending only on $L_\sigma$ and $M_\sigma$ such that, whenever $0\le \lambda d\le1$, $\inf_{\varphi\in\mathcal P(S_2^{d-1})}\left\{\|f_\varphi-f^\star\|_{L^2(\bP_X)}^2+\lambda\Ent_\tau^-(\varphi)\right\}\le C_{\rm app}\psi(\lambda d).$
\end{Lemma}
We omit its proof since it is similar to that of Lemma~\ref{lem:smoothed-dirac-general-ie}.


\begin{Theoremrewrite}{thm:multi-index-fixed-output}
    Grant Assumption~\ref{ass:multi_index}.
    Then there exist constants $c_0\in(0,1)$, $C\ge1$, and $C_0\ge e$, depending only on $M$, $\Delta_\star^{-1}$, $B_\sigma$, $L_\sigma$, $M_\sigma$, $T_\sigma$, $B_\xi$, and $\max_{1\le k\le M}|b_k|^{-1}$, such that for any $x\ge1$, $N\ge2$, and $\lambda d\le c_0$, for $r_*$ in Theorem~\ref{thm:multi-index-fixed-output},    with probability at least $1-4e^{-x}$, for any $1\leq m\leq M$,
\begin{align*}
    &\|f_{\hat\varphi_\lambda}-f^\star\|_{L^2(\bP_X)}^2+\lambda\Ent_\tau^-(\hat\varphi_\lambda) \le r_*^2,\mbox{ and } \norm{\bE\left[(\hat{\boldsymbol{W}})^{\otimes m}\big|(X_i,Y_i)_{i=1}^N\right]-\sum_{j=1}^M a_j^\star(\boldsymbol{w}_j^\star)^{\otimes m}}_F^2\leq \frac{m!}{b_m^2}r_*^2.
\end{align*}
\end{Theoremrewrite}
Its proof is almost identical to that of Theorem~\ref{thm:fixed-output-general-ie}, and is therefore omitted.

\subsection{Proof of Theorem~\ref{thm:rip-multi-index-fixed-output}}

The proof strategy in this subsection is the same as in
Section~\ref{sec:proof_thm_multi-index-fixed-output}. The essential differences are that the empirical comparison is carried out with a different low-degree polynomial $Q_{\mathrm{rip}}$, as constructed in Lemma~\ref{lem:rip-low-degree-localization}.

\begin{Lemma}
\label{lem:rip-low-degree-localization}
Grant Assumption~\ref{ass:rip_multi_index}.  For every $\varphi\in\mathcal P(S_2^{d-1})$,
\[
S_\star(\varphi) := \int_{S_2^{d-1}} \min_{1\le j\le M} \|\boldsymbol{w}-\boldsymbol{w}_j^\star\|_2^2 \,\varphi(\mathrm{d}\boldsymbol{w}) \le C_{\eta,\sigma}\sqrt M \|f_\varphi-f^\star\|_{L^2(\mathbb P_X)}.
\]
\end{Lemma}

\beginproof Define $\boldsymbol{v}_j^\star:=W_\star G_\star^{-1}\boldsymbol{e}_j$ and $P_\star:=W_\star G_\star^{-1}W_\star^\top$. Then $P_\star$ is the Euclidean projection onto $\Range(W_\star)$, and $\langle\boldsymbol{v}_j^\star,\boldsymbol{w}_k^\star\rangle=\delta_{jk}$. For $\boldsymbol{w}\in S_2^{d-1}$, we have $P_\star\boldsymbol{w}=\sum_{j=1}^M\langle\boldsymbol{v}_j^\star,\boldsymbol{w}\rangle\boldsymbol{w}_j^\star$, $\boldsymbol{w}=P_\star\boldsymbol{w}+(I-P_\star)\boldsymbol{w}$, $(I-P_\star)\boldsymbol{w}\perp\Range(W_\star)$, and since $\|\boldsymbol{w}\|_2=1$, we have $\|P_\star\boldsymbol{w}\|_2^2+\|(I-P_\star)\boldsymbol{w}\|_2^2=1$.

Define the polynomial $Q_{\rm rip}(\boldsymbol{w}):=\|(I-P_\star)\boldsymbol{w}\|_2^2+\sum_{j=1}^M\langle\boldsymbol{v}_j^\star,\boldsymbol{w}\rangle^2(1-\langle\boldsymbol{v}_j^\star,\boldsymbol{w}\rangle)^2$. Applying Lemma~\ref{lem:rip-coordinate-rounding} to the coordinates $(\langle\boldsymbol{v}_j^\star,\boldsymbol{w}\rangle)_{j=1}^M$ and to $(I-P_\star)\boldsymbol{w}$ gives $\min_{1\le j\le M}\|\boldsymbol{w}-\boldsymbol{w}_j^\star\|_2^2\le C_\eta Q_{\rm rip}(\boldsymbol{w})$. Thus $S_\star(\varphi)\le C_\eta\int Q_{\rm rip}(\boldsymbol{w})\,\varphi(\mathrm{d}\boldsymbol{w})$.

We now expand $Q_{\rm rip}$. Since $\|(I-P_\star)\boldsymbol{w}\|_2^2=1-\langle P_\star,\boldsymbol{w}^{\otimes2}\rangle_F$ and $\langle\boldsymbol{v}_j^\star,\boldsymbol{w}\rangle^2(1-\langle\boldsymbol{v}_j^\star,\boldsymbol{w}\rangle)^2=\langle\boldsymbol{v}_j^\star,\boldsymbol{w}\rangle^2-2\langle\boldsymbol{v}_j^\star,\boldsymbol{w}\rangle^3+\langle\boldsymbol{v}_j^\star,\boldsymbol{w}\rangle^4$, we have $Q_{\rm rip}(\boldsymbol{w})=1+\sum_{m=2}^4\langle A_m,\boldsymbol{w}^{\otimes m}\rangle_F$, where $A_2=-P_\star+\sum_{j=1}^M(\boldsymbol{v}_j^\star)^{\otimes2}$, $A_3=-2\sum_{j=1}^M(\boldsymbol{v}_j^\star)^{\otimes3}$, and $A_4=\sum_{j=1}^M(\boldsymbol{v}_j^\star)^{\otimes4}$.

We claim $\|A_m\|_F\le C_\eta\sqrt M$ for $m=2,3,4$. Indeed, $\|P_\star\|_F=\sqrt M$. Also, $\langle\boldsymbol{v}_i^\star,\boldsymbol{v}_j^\star\rangle=\boldsymbol{e}_i^\top G_\star^{-1}\boldsymbol{e}_j$ and $\|G_\star^{-1}\|_{\rm op}\le(1-\eta)^{-1}$. For $m=2,3,4$, $\Big\|\sum_{j=1}^M(\boldsymbol{v}_j^\star)^{\otimes m} \Big\|_F^2 = \sum_{i,j=1}^M \langle \boldsymbol{v}_i^\star,\boldsymbol{v}_j^\star\rangle^m \le \sum_{i,j=1}^M |\langle \boldsymbol{v}_i^\star,\boldsymbol{v}_j^\star\rangle|^m \le C_{\eta,m} \sum_{i,j=1}^M |\langle \boldsymbol{v}_i^\star,\boldsymbol{v}_j^\star\rangle|^2 = C_{\eta,m} \|G_\star^{-1}\|_F^2 \le C_{\eta,m}M.$ Thus $\|A_m\|_F\le C_\eta\sqrt M$.

Since $Q_{\rm rip}(\boldsymbol{w}_j^\star)=0$ for every $j$, we have $\int Q_{\rm rip}\,\mathrm{d}\varphi_\star=0$, where $\varphi_\star = \sum_{j=1}^M a_j^\star\delta_{\vw_j^\star}$. Also $\varphi$ and $\varphi_\star$ both have total mass one, so the constant term $1$ cancels: $\int Q_{\rm rip}\,\mathrm{d}\varphi=\int Q_{\rm rip}\,\mathrm{d}(\varphi-\varphi_\star)=\sum_{m=2}^4\left\langle A_m,\mathcal T_m(\varphi)-\mathcal T_m(\varphi_\star)\right\rangle_F.$ Therefore, $\int Q_{\rm rip}\,\mathrm{d}\varphi \le \sum_{m=2}^4 \|A_m\|_F \left\| \mathcal T_m(\varphi)-\mathcal T_m(\varphi_\star) \right\|_F \le C_{\eta,\sigma}\sqrt M \|f_\varphi-f^\star\|_{L^2(\mathbb P_X)}.$
Combining this with $S_\star(\varphi)\le C_\eta\int Q_{\rm rip}\,\mathrm{d}\varphi$ proves the claim.
\endproof

The proof of the remaining part is almost identical to the previous argument. We therefore state only the conclusion and omit the proof.
\begin{Lemma}
\label{lem:bernstein-net-chaining-psi1}
Let $\mathcal G$ be a class of measurable functions on a probability space,
and let $d_{\psi_1}(g,h):=\|g-h\|_{\psi_1}.$
Assume that $0\in\mathcal G$, that $\sup_{g\in\mathcal G}\|g\|_{\psi_1}\le K,$
and that for some $A\ge e$ and $v\ge1$, $N(\mathcal G,d_{\psi_1},\varepsilon)\le \left(\frac{AK}{\varepsilon}\right)^v$ for $0<\varepsilon\le K.$ Let $X_1,\ldots,X_N$ be i.i.d. and let $\varepsilon_1,\ldots,\varepsilon_N$ be independent Rademacher variables. Then
\[
\mathbb E\sup_{g\in\mathcal G}\left|\frac1N\sum_{i=1}^N\varepsilon_i g(X_i)\right|\le C K\left(\sqrt{\frac{v\log(AN)}{N}}+\frac{v\log(AN)}{N}\right).
\]
\end{Lemma}



\begin{Lemma}[Projected second-order multiplier]
\label{lem:projected-second-order-multiplier}
Grant Assumption~\ref{ass:rip_multi_index}. Let
$$\mathcal V_\star := \operatorname{span} \left\{ \sigma(\langle \boldsymbol{w}_j^\star,\cdot\rangle), \ \boldsymbol{x}\mapsto \sigma'(\langle \boldsymbol{w}_j^\star,\boldsymbol{x}\rangle)x_\ell : 1\le j\le M,\ 1\le\ell\le d \right\}.$$
Let $\Pi_\star$ be the $L^2(\mathbb P_X)$-orthogonal projection onto
$\mathcal V_\star$. Define
$$\mathcal Q_1 := \left\{ \boldsymbol{x}\mapsto \sigma''(\langle \boldsymbol{u},\boldsymbol{x}\rangle)\langle \boldsymbol{e},\boldsymbol{x}\rangle^2 : \boldsymbol{u}\in B_2^d,\ \boldsymbol{e}\in S_2^{d-1} \right\}.$$
Then
\[
\mathbb E\sup_{q\in\mathcal Q_1} \left| \frac1N\sum_{i=1}^N \varepsilon_i(I-\Pi_\star)q(X_i) \right| \le C_{\eta,\sigma} \left( \sqrt{\frac{d\log(C_0dN)}{N}} + \frac{d\log(C_0dN)}{N} \right).
\]
\end{Lemma}

\begin{Proposition}[Localized Rademacher bound under the frame assumption]
\label{prop:rip-multi-index-rademacher}
Grant Assumption~\ref{ass:rip_multi_index}. 
Then, for every $r>0$,
\[
\mathbb E\sup_{\|f_\varphi-f^\star\|_{L^2(\mathbb P_X)}\le r} \left| \frac1N\sum_{i=1}^N \varepsilon_i(f_\varphi-f^\star)(X_i) \right| \le C_{\eta,\sigma}r \left( \sqrt{\frac{M d\log(C_0dN)}{N}} + \sqrt M\,\frac{d\log(C_0dN)}{N} \right).
\]
The same bound, with $C_{\eta,\sigma}$ replaced by
$C_{\eta,\sigma,B_\xi}$, holds for the localized squared-loss class
$\mathcal L(r) := \left\{ (\boldsymbol{x},y)\mapsto (y-f_\varphi(\boldsymbol{x}))^2-(y-f^\star(\boldsymbol{x}))^2 : \|f_\varphi-f^\star\|_{L^2(\mathbb P_X)}\le r \right\}.$
\end{Proposition}

\begin{Lemma}
\label{lem:rip-multi-index-isomorphism}
Grant Assumption~\ref{ass:rip_multi_index}. 
Then, for every $x\ge1$ and $N\ge2$, with probability at least $1-4e^{-x}$, simultaneously for all $\varphi\in\mathcal P(S_2^{d-1})$,
\[
\left| (P-P_N)\cL_\varphi \right| \le \frac14 \|f_\varphi-f^\star\|_{L^2(\mathbb P_X)}^2 + C_{\eta,\sigma,B_\xi} \left[ \frac{M d\log(C_0dN)+x}{N} + \frac{M (d\log(C_0dN))^2}{N^2} \right].
\]
\end{Lemma}

\begin{Theorem}\label{thm:rip-multi-index-fixed-output}
Grant Assumption~\ref{ass:rip_multi_index}. There exist constants $c_0\in(0,1)$, $C\ge1$, and $C_0\ge e$, depending only on $\eta$, $B_\sigma,L_\sigma,M_\sigma,T_\sigma, |b_2|^{-1},|b_3|^{-1},|b_4|^{-1}, K_\sigma, \{b_m,b_{m+2}:m\in K_\sigma\}, \kappa_\sigma^{-1}$, and $B_\xi$, such that the following holds. Let $x\ge1$, $N\ge2$, and $\lambda\ge0$ satisfy $\lambda d\le c_0$. For $\lambda\ge 0$, define
\[
r_{\rm rip}^2:=C\left[\frac{M d\log(C_0dN)+x}{N}+\frac{M (d\log(C_0dN))^2}{N^2}+\psi(\lambda d)\right].
\]
Then, with probability at least $1-4e^{-x}$, $\|f_{\hat\varphi_\lambda}-f^\star\|_{L^2(\mathbb P_X)}^2+ \lambda\Ent_\tau^-(\hat\varphi_\lambda)\le r_{\rm rip}^2$. Moreover, for $m=2,3,4$,
\[
\Big\|\int_{S_2^{d-1}} \boldsymbol{w}^{\otimes m}\hat\varphi_\lambda(\mathrm{d}\boldsymbol{w})-\sum_{j=1}^M a_j^\star (\boldsymbol{w}_j^\star)^{\otimes m}\Big\|_F^2\le \frac{m!}{b_m^2}r_{\rm rip}^2 \,\mbox{ and }~\, \int_{S_2^{d-1}}\min_{1\le j\le M}\|\boldsymbol{w}-\boldsymbol{w}_j^\star\|_2^2\,\hat{\varphi}_{\lambda}(\mathrm{d}\boldsymbol{w})\le C_{\eta,\sigma}\sqrt M\,r_{\rm rip}.
\]
If $N\ge d\log(C_0dN)$, then the second statistical term is absorbed, and for $\lambda\ge 0$, $\|f_{\hat\varphi_\lambda}-f^\star\|_{L^2(\mathbb P_X)}^2+\lambda\Ent_\tau^-(\hat\varphi_\lambda)\le C\left[\frac{M d\log(C_0dN)+x}{N}+\psi(\lambda d)\right]$.
\end{Theorem}

\subsection{Proof of Corollary~\ref{coro:feature_learning_multi}}\label{sec:prop_feature_learning_multi}

\beginproof
Item~\emph{\ref{item:feature_evolution}} is verified by Corollary~\ref{cor:feature-kernel-movement}.
Let $\mathbf 1_{\fea}\in{L^2(\hat\varphi_\lambda)}$ be the constant-one function. For the spherical MFLD considered here, $f_{\hat\varphi_\lambda}(\cdot)=\langle\mathbf 1_{\fea},\phi_\neu(\cdot)\rangle_{L^2(\hat\varphi_\lambda)}$. We take $a_{\fea}=\hat a_N=\mathbf 1_{\fea}$. 
Since \(g_{\fea}=\hat g_N\), the alignment condition is trivially satisfied, for instance with \(\omega_N\equiv0\). Therefore, Item~\emph{\ref{item:alignment}} is verified. 
Moreover, since $g_{\fea}\circ\phi_\fea=f_{\hat\varphi_\lambda}$, Item~\emph{\ref{item:approximation_error}} follows from Theorem~\ref{thm:fixed-output-general-ie} and Theorem~\ref{thm:multi-index-fixed-output}. Finally, it remains to prove Item~\emph{\ref{item:top_k}}. Here we take the multi-index case as an example; the single-index case is analogous. For any $\vw$, let $1\leq j(\vw)\leq M$ be the unique index such that $\vw\in V_{j(\vw)}$, where $V_{j(\vw)}$ is the partition defined in Section~\ref{sec:proof_strategy}. Let $\cU$ be a coefficient sub-space defined as
\begin{align*}
&\cU = \left\{ \vw\mapsto \sum_{j=1}^M\1_{V_j}(\vw)(\alpha_j + \langle\bbeta_j,\vw-\vw_j^*\rangle): \alpha_j\in\bR,\, \bbeta_j\in\bR^d \right\}\subset L^2(\hat\varphi_\lambda).
\end{align*}
Then $\dim(\cU)\leq M(d+1)$. Let $\mathrm{Proj}^{\mathcal H}_{\cU}$ be the orthogonal projection in $\mathcal H_{\mathrm{feat}}$ onto the closed linear subspace of all $g\in\mathcal H_{\mathrm{feat}}$ identified by \eqref{eq:quotient_H_fea}; in other words, if $g\in\cH_\fea$ with $L^2(\hat\varphi_\lambda)$ coefficient $a$, i.e. $g(\cdot) = \inr{a, \varphi_\neu(\cdot)}_{L^2(\hat\varphi_\lambda)}$ then for $a_\cU=\argmin_{u\in\cU}\norm{u-a}_{L^2(\hat\varphi_\lambda)}$, we have   $\mathrm{Proj}^{\mathcal H}_{\cU}(g)=\inr{a_\cU,\varphi_{\neu}(\cdot)}_{\hat\varphi_\lambda}$. Let
\begin{align*}
&\sigma^{(1)}:(\vx,\vw) \mapsto \sigma(\langle\vw_{j(\vw)}^*,\vx\rangle) + \sigma'(\langle\vw_{j(\vw)}^*,\vx\rangle)\langle\vw-\vw_{j(\vw)}^*,\vx\rangle.
\end{align*}
For any $\vx$, there holds $\sigma^{(1)}(\vx,\cdot)\in\cU$. By Ky Fan's maximum principle (see \cite[Lemma 8.1.8]{stormerPositiveLinearMaps2013}), for any orthogonal projection $\mathrm{Proj}$ of rank at most $k$, there holds $\sum_{j>k}\sigma_j\leq \Tr((I-\mathrm{Proj})\Sigma)$. Therefore, $\sum_{j>k}\sigma_j\leq \Tr((I-\mathrm{Proj}^{\mathcal H}_{\cU})\Sigma)$ where $k=M(d+1)$. Now, since $\Sigma = \bE[\phi_{\mathrm{feat}}(X)\otimes\phi_{\mathrm{feat}}(X)|(X_i,Y_i)_{i=1}^N]$ is the covariance operator on $\mathcal H_{\mathrm{feat}}$, there holds
\begin{align*}
&\Tr((I-\mathrm{Proj}^{\mathcal H}_{\cU})\Sigma) = \bE[\|(I-\mathrm{Proj}^{\mathcal H}_{\cU})\phi_{\mathrm{feat}}(X)\|_{\mathcal H_{\mathrm{feat}}}^2|(X_i,Y_i)_{i=1}^N].
\end{align*}
Since $\phi_{\mathrm{feat}}(X)$ is represented by the coefficient $\sigma(\langle X,\cdot\rangle)$ in $L^2(\hat \varphi_\lambda)$ and $\sigma^{(1)}(X,\cdot)\in\cU$ for almost all $X$, the quotient characterization of the $\mathcal H_{\mathrm{feat}}$ norm gives
\begin{align*}
&\|(I-\mathrm{Proj}^{\mathcal H}_{\cU})\phi_{\mathrm{feat}}(X)\|_{\mathcal H_{\mathrm{feat}}}^2\leq \|\sigma(\langle X,\cdot\rangle)-\sigma^{(1)}(X,\cdot)\|_{L^2(\hat\varphi_\lambda)}^2.
\end{align*}
Consequently,
\begin{align*}
&\sum_{j>k}\sigma_j \leq \bE\left[\norm{\sigma(\langle X,\cdot\rangle) - \sigma^{(1)}(X,\cdot)}_{L^2(\hat\varphi_\lambda)}^2\big|(X_i,Y_i)_{i=1}^N\right].
\end{align*}
Now, by Taylor's expansion, if $\vw\in V_j$, then
\begin{align*}
&\left| \sigma(\langle\vw,X\rangle) - \sigma(\langle\vw_j^*,X\rangle) - \sigma'(\langle\vw_j^*,X\rangle)\langle\vw-\vw_j^*,X\rangle \right|\leq\frac{1}{2}M_\sigma\left|\langle \vw-\vw_j^*,X\rangle\right|^2.
\end{align*}
Therefore,
\begin{align*}
&\bE\left[\norm{\sigma(\langle X,\cdot\rangle) - \sigma^{(1)}(X,\cdot)}_{L^2(\hat\varphi_\lambda)}^2\big|(X_i,Y_i)_{i=1}^N\right]\leq \frac{3}{4}M_\sigma^2\int_{S_2^{d-1}}\norm{\vw-\vw_{j(\vw)}^*}_2^4d\hat\varphi_\lambda(\vw)\leq 3M_\sigma^2S_*(\hat\varphi_\lambda),
\end{align*}
where $S_*(\hat\varphi_\lambda) = \int_{S_2^{d-1}}\min_{1\leq j\leq M}\norm{\vw-\vw_j^*}_2^2 d\hat\varphi_\lambda(\vw).$ By Proposition~\ref{prop:stationary_distribution} and Proposition~\ref{prop:concentration_Levy_Milman} respectively, we have $S_*(\hat\varphi_\lambda)=o_\bP(1)$. Since \(\gamma_j=\sigma_j\), and \(a_{\fea}=\mathbf 1_{\fea}\) is a coefficient representation of \(g_{\fea}\), and \(\|g_{\fea}\|_{\mathcal H_{\mathrm{feat}}}\leq \|a_{\fea}\|_{L^2(\hat\varphi_\lambda)}=1\),
\begin{align*}
&\sum_{j>k}\gamma_j\langle g_{\fea},e_j\rangle_{\mathcal H_{\mathrm{feat}}}^2\leq \|g_{\fea}\|_{\mathcal H_{\mathrm{feat}}}^2\sum_{j>k}\sigma_j\leq \sum_{j>k}\sigma_j=o_{\mathbb P}(1).
\end{align*}
Thus Item~\emph{\ref{item:top_k}} is verified.
\endproof

\subsection{Proof of Corollary~\ref{cor:feature-kernel-movement}}\label{sec:proof_coro_feature_kernel_movement}
\begin{proof}[Proof of Corollary~\ref{cor:feature-kernel-movement}]
We work on the random event of Theorem~\ref{thm:fixed-output-general-ie} and of Theorem~\ref{thm:multi-index-fixed-output} respectively.

For the single-index case, by Proposition~\ref{prop:concentration_Levy_Milman}, there holds $\int\min\{ \|\vw-\vw_\star\|_2^2, \|\vw+\vw_\star\|_2^2 \}\mathrm{d}\hat\varphi_\lambda(\vw)\leq r_*$. Since $\sigma$ is bounded and Lipschitz, $\vw\mapsto\sigma(\langle\vw,\cdot\rangle)\sigma(\langle\vw,\cdot\rangle)\in L^2(\bP_X\otimes\bP_X)$ is Lipschitz with Lipschitz constant $2B_\sigma L_\sigma$, then
\begin{align*}
    \norm{K_{\hat\varphi_\lambda} - K_{\tilde\varphi_\lambda}}_{L^2(\bP_X\otimes\bP_X)}\leq C_{\mathrm{ker}}\sqrt{r_*},\mbox{ where } \tilde\varphi_\lambda = \hat\varphi_\lambda(\{\langle\vw,\vw_\star\rangle\geq 0\})\delta_{\vw_\star} + \hat\varphi_\lambda(\{\langle\vw,\vw_\star\rangle< 0\})\delta_{-\vw_\star}.
\end{align*}Take $X'$ be the independent copy of $X$. For any $\varphi$ such that $\varphi(\{\vw_\star,-\vw_\star\})=1$, since $\mathrm{He}_{\mathrm{IE}(\sigma)}(-t)=(-1)^{\mathrm{IE}(\sigma)}\mathrm{He}_{\mathrm{IE}(\sigma)}(t)$ for any $t$, there holds
\begin{align*}
    &\int \mathrm{He}_{\mathrm{IE}(\sigma)}(\langle\vw,X\rangle)\mathrm{He}_{\mathrm{IE}(\sigma)}(\langle\vw,X'\rangle)\mathrm{d}\varphi(\vw) \\
    &= \varphi(\{\vw_\star\})\mathrm{He}_{\mathrm{IE}(\sigma)}(\langle\vw,X\rangle)\mathrm{He}_{\mathrm{IE}(\sigma)}(\langle\vw,X'\rangle) + \varphi(\{-\vw_\star\})\mathrm{He}_{\mathrm{IE}(\sigma)}(\langle-\vw,X\rangle)\mathrm{He}_{\mathrm{IE}(\sigma)}(\langle-\vw,X'\rangle)\\
    &=\mathrm{He}_{\mathrm{IE}(\sigma)}(\langle\vw,X\rangle)\mathrm{He}_{\mathrm{IE}(\sigma)}(\langle\vw,X'\rangle).
\end{align*}For any $f,g\in L^2(\bP_X)$,  $\mathrm{Proj}_{\mathrm{IE}(\sigma),\mathrm{IE}(\sigma)}(fg) := \mathrm{Proj}_{\mathrm{IE}(\sigma)}f\mathrm{Proj}_{\mathrm{IE}(\sigma)}g$ is the projection from $L^2(\bP_X\otimes\bP_X)$ onto $C_{\mathrm{IE}(\sigma)}\otimes C_{\mathrm{IE}(\sigma)}$, where $C_{\mathrm{IE}(\sigma)}$ is the $\mathrm{IE}(\sigma)$-th homogeneous Wiener chaos in $L^2(\bP_X)$ defined in Lemma~\ref{lemma:C_m_is_Span_Hermite}, see also, for instance, \cite[Section 0.4]{aubrunAliceBobMeet2017}, for the tensor product of Hilbert spaces and its operators. From Lemma~\ref{lemma:proj-ridge-feature}, for any $\varphi(\{\vw_\star,-\vw_\star\})=1$, there holds
\begin{align*}
    &\mathrm{Proj}_{\mathrm{IE}(\sigma),\mathrm{IE}(\sigma)} K_\varphi(X,X') = \frac{b_{\mathrm{IE}(\sigma)}^2}{(\mathrm{IE}(\sigma)!)^2}\int \mathrm{He}_{\mathrm{IE}(\sigma)}(\langle\vw,X\rangle)\mathrm{He}_{\mathrm{IE}(\sigma)}(\langle\vw,X'\rangle)\mathrm{d}\varphi(\vw)\\
    &= \frac{b_{\mathrm{IE}(\sigma)}^2}{(\mathrm{IE}(\sigma)!)^2}\mathrm{He}_{\mathrm{IE}(\sigma)}(\langle\vw_\star,X\rangle)\mathrm{He}_{\mathrm{IE}(\sigma)}(\langle\vw_\star,X'\rangle).
\end{align*}Similarly, $\mathrm{Proj}_{\mathrm{IE}(\sigma),\mathrm{IE}(\sigma)} K_\tau(X,X') = \frac{b_{\mathrm{IE}(\sigma)}^2}{(\mathrm{IE}(\sigma)!)^2}\int \mathrm{He}_{\mathrm{IE}(\sigma)}(\langle\vw,X\rangle)\mathrm{He}_{\mathrm{IE}(\sigma)}(\langle\vw,X'\rangle)\mathrm{d}\tau(\vw).$ Moreover, by $\|K_\varphi - K_\tau\|_{L^2(\bP_X\otimes\bP_X)}^2 \geq \|\mathrm{Proj}_{\mathrm{IE}(\sigma),\mathrm{IE}(\sigma)} (K_\varphi - K_\tau)\|_{L^2(\bP_X\otimes\bP_X)}^2$ together with \eqref{eq:inner_product_Hermite_polynomials}, there holds
\begin{align*}
    &\|K_\varphi - K_\tau\|_{L^2(\bP_X\otimes\bP_X)}^2 \\
    &\geq \frac{|b_{\mathrm{IE}(\sigma)}|^4}{(\mathrm{IE}(\sigma)!)^4}\bigg( (\mathrm{IE}(\sigma)!)^2 - 2(\mathrm{IE}(\sigma)!)^2\int \langle\vw,\vw_\star\rangle^{2\mathrm{IE}(\sigma)}\mathrm{d}\tau(\vw) + (\mathrm{IE}(\sigma)!)^2\int\int\langle\vw,\vv\rangle^{2\mathrm{IE}(\sigma)}\mathrm{d}\tau(\vw)\mathrm{d}\tau(\vv) \bigg)\\
    &=\frac{|b_{\mathrm{IE}(\sigma)}|^4}{(\mathrm{IE}(\sigma)!)^2}\left( 1-\int \langle\vw,\vw_\star\rangle^{2\mathrm{IE}(\sigma)}\mathrm{d}\tau(\vw) \right) = \frac{|b_{\mathrm{IE}(\sigma)}|^4}{(\mathrm{IE}(\sigma)!)^2}\left(1 - \frac{(2\mathrm{IE}(\sigma))!!}{d(d+2)\cdots(d+2\mathrm{IE}(\sigma)-2)}\right),
\end{align*}where the last equality follows from the following observation. Let $G\sim\cN(\vzero,I_d)$, then $G/\|G\|_2\sim\tau$, and $\|G\|_2$ is independent of $G/\|G\|_2$, see, for instance, \cite[Appendix A.2]{aubrunAliceBobMeet2017}. Therefore, $$\int \langle\vw,\vw_\star\rangle^{2\mathrm{IE}(\sigma)}\mathrm{d}\tau(\vw) = \frac{\bE\langle G,\vw_\star\rangle^{2\mathrm{IE}(\sigma)}}{\bE\|G\|_2^{2\mathrm{IE}(\sigma)}} = \frac{(2\mathrm{IE}(\sigma))!!}{d(d+2)\cdots(d+2\mathrm{IE}(\sigma)-2)},$$ where the numerator is the moment of a standard Gaussian random variable, while the denominator is that of a chi-square random variable with degree-of-freedom $d$.

Now we deal with the multi-index case. We use Proposition~\ref{prop:stationary_distribution}. Similar to the single-index case, we have
\begin{align*}
    &\|K_{\hat\varphi_\lambda}-K_{\tilde\varphi_\lambda}\|_{L^2(\bP_X\otimes\bP_X)}\leq 2B_\sigma L_\sigma\sum_{j=1}^M\int_{V_j}\|\vw-\vw_j^\star\|_2\,\mathrm{d}\hat\varphi_\lambda(\vw)\lesssim \sqrt{r_*},\mbox{ where }\tilde\varphi_\lambda = \sum_{j=1}^M \hat\varphi_\lambda(V_j)\delta_{\vw_j^\star}.
\end{align*}
On the other hand, $\|K_{\tilde\varphi_\lambda}-K_{\varphi^\star}\|_{L^2(\bP_X\otimes\bP_X)} \leq B_\sigma^2 \sum_{j=1}^M \left|\int_{V_j}\mathrm{d}\hat\varphi_\lambda-a_j^\star\right|\lesssim \sqrt{r_*}.$ By the triangular inequality, $\|K_{\hat\varphi_\lambda}-K_{\varphi^\star}\|_{L^2(\bP_X\otimes\bP_X)}\lesssim\sqrt{r_*}$. Now we compute the $(m,m)$-Hermite chaos projection for $1\le m\le M$. Since $\mathrm{Proj}_{m,m}=\mathrm{Proj}_m\otimes\mathrm{Proj}_m$, the Hermite identity and the independence of $X,X'$ give
\begin{align*}
    &\|\mathrm{Proj}_{m,m}(K_{\varphi^\star}-K_\tau)\|_{L^2(\bP_X\otimes\bP_X)}=\frac{|b_m|^2}{m!}\bigg\|\sum_{j=1}^M a_j^\star(\vw_j^\star)^{\otimes m}\otimes(\vw_j^\star)^{\otimes m}-\int_{S_2^{d-1}}\vw^{\otimes m}\otimes\vw^{\otimes m}\,\mathrm{d}\tau(\vw)\bigg\|_F.
\end{align*}
Hence, by the contraction property of orthogonal projections and the triangular inequality, $\|K_{\hat\varphi_\lambda}-K_\tau\|_{L^2(\bP_X\otimes\bP_X)}\ge \max_{1\le m\le M}\frac{|b_m|^2}{m!}\big\|\sum_{j=1}^M a_j^\star(\vw_j^\star)^{\otimes m}\otimes(\vw_j^\star)^{\otimes m}-\int_{S_2^{d-1}}\vw^{\otimes m}\otimes\vw^{\otimes m}\,\mathrm{d}\tau(\vw)\big\|_F-C_{\rm ker}\sqrt{r_*}$. By the assumption on $r_*$, the proof is complete.
\end{proof}

\section{Supplementary Lemmas}\label{sec:supp_lemmas}
This section collects technical lemmas used in several parts of the appendix. They are separated from the main proof blocks to keep the main arguments easier to follow.

\begin{Lemma}\label{lemma:levy_milman_concentration}
    Let $(\cX,d_\cX)$ be a metric space, $\mu\in\cP(\cX)$. Suppose there exists $z\in\cX$ such that $\int d_\vx^2(\vx,z)\mathrm{d}\mu(\vx)\leq \varepsilon$ for some $\varepsilon>0$. For any $r>0$, let $A_r=\{\vx\in\cX: d_\cX(\vx,A)\leq r\}$ and $\alpha_\mu(r)=\sup\{1-\mu(A_r): \mu(A)\geq \frac{1}{2}\}$ be the L\'evy concentration function, \cite[Section 1.3]{ledouxConcentrationMeasurePhenomenon2005}. Then for any $r>2\sqrt{2\varepsilon}$, $\alpha_\mu(r)\leq\frac{4\varepsilon}{r^2}$. Moreover,
    \begin{align*}
        \forall 1\mbox{-Lipschitz } F:\cX\to\bR,\quad \forall r>2\sqrt{2\varepsilon},\quad \mu\left(\left| F(X) - \bE_{X\sim \mu}[F(X)] \right| \geq r\right)\leq \frac{8\varepsilon}{r^2}.
    \end{align*}
\end{Lemma}
\beginproof
By Markov's inequality, for any $u>0$, $\mu(d_\cX(X,z)\geq u)\leq \frac{\varepsilon}{u^2}$, where $X$ is a random variable distributed according to $\mu$. Since $F$ is $1$-Lipschitz, $|F(X)-F(z)|\leq d_\cX(X,z)$ and $|\bE[F(X)]-F(z)|\leq \int_\cX d_\cX(\vx,z)\mathrm{d}\mu(\vx)\leq\sqrt{\varepsilon}$. Therefore, when $t>\sqrt{\varepsilon}$, $\{|F(X)-\bE[F(X)]|\geq t\}\subset\{|F(X)-F(z)|\geq t-\sqrt{\varepsilon}\}\subset \{d_\cX(X,z)\geq t-\sqrt{\varepsilon}\}$. Applying to $u=t-\sqrt{\varepsilon}$, for any $t>\sqrt{\varepsilon}$, there holds $\mu(|F(X)-\bE[F(X)]|\ge t)\leq \frac{\varepsilon}{(t-\sqrt{\varepsilon})^2}$. Setting $t=r$ together with $\frac{\varepsilon}{(r-\sqrt{\varepsilon})^2}\leq \frac{4\varepsilon}{r^2}$ yield the desired concentration inequality. For the L\'evy concentration function, for any $A$ such that $\mu(A)\geq \frac{1}{2}$, let $u=\frac{r}{2}$. If $r>2\sqrt{2\varepsilon}$, then $\frac{\varepsilon}{u^2}<\frac{1}{2}$, and consequently, $\mu(B_\cX(z,u)) = \mu(d(X,z)\leq u)\geq 1-\frac{\varepsilon}{u^2}>\frac{1}{2}$. Since $\mu(A)\geq\frac12$, there holds $\mu(A\cap B(z,u))\geq \mu(A)+\mu(B(z,u))-1>0$, and hence $A\cap B(z,u)\neq\emptyset$. Take $\vy\in A\cap B(z,u)$. Take any $\vx\notin A_r$, then $d_\cX(\vx,A)>r$, hence $r<d_\cX(\vx,A)\leq d_\cX(\vx,\vy)\leq d_\cX(\vx,z) + d_\cX(z,\vy)\leq d_\cX(\vx,z)+u,$ that is, $d_\cX(\vx,z)>r-u=\frac{r}{2}$. This implies that $A_r^c\subset\{\vx\in\cX: d_\cX(\vx,z)>\frac{r}{2}\}$. By Markov's inequality again, $\mu(A_r^c)\leq \mu\left(d_\cX(\vx,z)>\frac{r}{2}\right)\leq\frac{\varepsilon}{(r/2)^2}=\frac{4\varepsilon}{r^2}.$ Since the above analysis holds for any $\mu(A)\geq \frac{1}{2}$, there holds $\alpha_\mu(r)\leq\frac{4\varepsilon}{r^2}$.
\endproof

\begin{Proposition}
\label{prop:no_interpolation_spherical_mean_field}
Let
$
\underline\sigma:=\inf_{t\in\mathbb R}\sigma(t)$, and $ \overline\sigma:=\sup_{t\in\mathbb R}\sigma(t).$
Let $p_{\mathrm{out}}:=\mathbb P\left(Y\notin [\underline\sigma,\overline\sigma]\right).$
If $p_{\mathrm{out}}>0$, then
\[
\mathbb P\left(\exists \varphi\in\mathcal P(S^{d-1}_2): f_\varphi(X_i)=Y_i,\ i=1,\ldots,N\right)
\leq (1-p_{\mathrm{out}})^N.
\]
\end{Proposition}

\beginproof
Let
$
P_\sigma\mathcal F=\left\{(f_\varphi(X_1),\ldots,f_\varphi(X_N)):\varphi\in\mathcal P(S^{d-1}_2)\right\}\subset\mathbb R^N.
$
Next, for every $\boldsymbol{w}\in S^{d-1}_2$ and every $\boldsymbol{x}\in\mathbb R^d$, $\sigma(\langle \boldsymbol{w},\boldsymbol{x}\rangle)\in[\underline\sigma,\overline\sigma].$
Hence, for every $\varphi\in\mathcal P(S^{d-1}_2)$,
$
f_\varphi(\boldsymbol{x})=\int_{S^{d-1}_2}\sigma(\langle \boldsymbol{w},\boldsymbol{x}\rangle)\,\varphi(d\boldsymbol{w})
\in
[\underline\sigma,\overline\sigma].
$
It follows that $P_\sigma\mathcal F
\subset
[\underline\sigma,\overline\sigma]^N.$
Therefore, if $Y_i\notin [\underline\sigma,\overline\sigma]$ for some $i$, then no $\varphi\in\mathcal P(S^{d-1}_2)$ can satisfy $f_\varphi(X_i)=Y_i$, and hence no such $\varphi$ can interpolate the whole training sample. The probabilistic statement follows trivially. 
\endproof

\begin{Lemma}\label{lem:rip-coordinate-rounding}
Assume $\|G_\star-I_M\|_{\rm op}\le\eta<1/2$. There exist constants $c_\eta>0$ and $C_\eta<\infty$, depending only on $\eta$, such that the following holds.
Let $\boldsymbol{z}\in\mathbb R^M$ and $r\ge0$ satisfy $\boldsymbol{z}^\top G_\star \boldsymbol{z}+r^2=1$. Define $\delta :=r^2+\sum_{j=1}^M z_j^2(1-z_j)^2$. Then
\[
\forall\vr_\perp\perp\Range(W_\star),\,\|\vr_\perp\|_2=1,\quad \min_{1\le j\le M} \|W_\star \boldsymbol{z}+r\boldsymbol{r}_\perp-\boldsymbol{w}_j^\star\|_2^2 \le C_\eta \delta.
\]
\end{Lemma}

\beginproof
If $\delta\ge c_\eta$, then the claim is immediate because both $W_\star \boldsymbol{z}+\boldsymbol{r}_\perp$ and $\boldsymbol{w}_j^\star$ are unit vectors, so the distance squared is at most $4$, and $4\le \frac4{c_\eta}\delta.$ Thus it remains to consider $\delta<c_\eta$, where $c_\eta$ will be chosen small enough.

For each coordinate, we have
$|z_j|\,|1-z_j| \le \sqrt\delta.$
If $|z_j|\le1/2$, then $|1-z_j|\ge1/2$, so we get
$|z_j|\le2\sqrt\delta.$
If $|z_j|>1/2$, then
$|1-z_j|\le2\sqrt\delta.$
Thus each coordinate is either close to $0$ or close to $1$.

We first show that at least one coordinate is close to $1$. If no coordinate satisfies $|z_j|>1/2$, then all $|z_j|\le1/2$, so $z_j^2(1-z_j)^2\ge \frac14z_j^2.$
Hence
$\|\boldsymbol{z}\|_2^2\le4\delta.$
Therefore $1=\boldsymbol{z}^\top G_\star \boldsymbol{z}+r^2 \le (1+\eta)\|\boldsymbol{z}\|_2^2+r^2 \le 4(1+\eta)\delta+\delta.$
This is impossible if $c_\eta$ is small enough.

Next we show uniqueness. Suppose two distinct coordinates $j\ne k$ satisfy $|z_j|>1/2$ and $|z_k|>1/2$. Then $|1-z_j|\le2\sqrt\delta$ and $|1-z_k|\le2\sqrt\delta $ hold, so $z_j^2+z_k^2 \ge 2(1-2\sqrt\delta)^2$. Thus $1 = \boldsymbol{z}^\top G_\star \boldsymbol{z}+r^2 \ge (1-\eta)\|\boldsymbol{z}\|_2^2 \ge 2(1-\eta)(1-2\sqrt\delta)^2.$ Since $\eta<1/2$, the right-hand side is strictly larger than $1$ when $\delta$ is small enough. This is impossible.

Hence there is a unique $j_\star$ such that $|z_{j_\star}|>1/2$. For this coordinate, $|1-z_{j_\star}|\le2\sqrt\delta$. For every $k\ne j_\star$, we have $|z_k|\le1/2$, and hence $z_k^2\le4z_k^2(1-z_k)^2$. Therefore $\|\boldsymbol{z}-\boldsymbol{e}_{j_\star}\|_2^2 = (z_{j_\star}-1)^2 + \sum_{k\ne j_\star}z_k^2 \le C\delta.$
Finally, $\|W_\star \boldsymbol{z}+r\boldsymbol{r}_\perp-\boldsymbol{w}_{j_\star}^\star\|_2^2 = \|W_\star(\boldsymbol{z}-\boldsymbol{e}_{j_\star})\|_2^2+r^2 \le (1+\eta)\|\boldsymbol{z}-\boldsymbol{e}_{j_\star}\|_2^2+r^2 \le C_\eta\delta.$
\endproof

\begin{Proposition}\label{prop:convergence_speed}
Suppose there exists an absolute constant $\nL\label{L_bound}>0$ such that $\|f^\star\|_{L^\infty(\bP_X)},\|\phi\|_\infty\leq\oL{L_bound}$, where $\phi:(a,\vw)\in\Theta\mapsto a\sigma(\langle\cdot,\vw\rangle):\bR^d\to\bR$. Suppose further that $$\sup_{\vx,\vtheta}\|\nabla_\vtheta\phi(\vx,\vtheta)\|_2, \mbox{ and } \sup_{\vx,\vtheta}\|\nabla_\vtheta^2\phi(\vx,\vtheta)\|_{\op}<\infty.$$ Take $\nu_0$ to be the uniform distribution over $\Theta$. Then there exists an absolute constant $C$ that depends only on $\oL{L_bound}$ and $B_\xi$, such that the nonlinear Fokker-Planck equation defined in \eqref{eq:nonlinear_Fokker_Planck} satisfies the following result: for any $N\in\bN_+$, for any $\{\vx_i,y_i\}_{i=1}^N$, for any $t\geq 0$, and $\lambda> 0$,
\begin{align*}
&P_N\ell_{\nu_t}^\lambda - P_N\ell_{\hat\nu_\lambda}^\lambda \leq \exp\left(-C\exp(-\frac{C}{\lambda}) t\right)(P_N\ell_{\nu_0}^\lambda - P_N\ell_{\hat\nu_\lambda}^\lambda),\mbox{ and }\\
&\KL(\nu_t\|\hat\nu_\lambda)\leq \frac{1}{\lambda}\exp\left(-C\exp(-\frac{C}{\lambda}) t\right)(P_N\ell_{\nu_0}^\lambda - P_N\ell_{\hat\nu_\lambda}^\lambda).
\end{align*}
\end{Proposition}
The following proof closely follows \cite{nitanda_convex_2022,chizat_mean-field_2022}. Moreover, the convergence in the sense of Wasserstein-2 distance may also be established via Talagrand's transportation inequality, see, for instance, \cite[Section 9.3]{villani_topics_2021}. The convergence in $L^2(\bP_X)$ distance may also be established by combining the KL convergence and the Pinsker's inequality.

\beginproof
Let $\tau$ be the uniform distribution on $\Theta$. Let $\rho_t=\frac{d\nu_t}{d\tau}$ satisfy the nonlinear Fokker--Planck equation $\partial_t\rho_t=\nabla\cdot(\rho_t\nabla\frac{\delta P_N\ell_{\nu_t}}{\delta\nu_t})+\lambda\Delta\rho_t$ defined in \eqref{eq:nonlinear_Fokker_Planck} (we already used the property $\nabla\cdot(\nabla\rho_t)=\Delta\rho_t$ of Laplacian operator here). For any $\nu\in\cP(\Theta)$, we define $\pi_\nu^\lambda$ by $$\frac{d\pi_\nu^\lambda}{d\tau}:\vtheta\mapsto \frac{\exp(-\frac{1}{\lambda}\frac{\delta P_N\ell_\nu}{\delta\nu}(\vtheta))}{\int_\Theta \exp\left(-\frac{1}{\lambda}\frac{\delta P_N\ell_\nu}{\delta\nu}(\vtheta')\right)\,\mathrm{d}\tau(\vtheta')}.$$ We say that a probability measure $\mu$ satisfies the log-Sobolev inequality with parameter $\alpha$ if, for any $\rho\ll\mu$, one has $\KL(\rho\|\mu)\leq \frac{1}{2\alpha}I(\rho\|\mu)$, where $I(\rho\|\mu)=\int_\Theta \|\nabla\log\frac{d\rho}{d\mu}(\vtheta)\|_2^2\,\mathrm{d}\rho(\vtheta)$.

We first prove that $\tau$ satisfies a log-Sobolev inequality. Take any $\vtheta^\circ\in\Theta$ and any $\kappa>0$. Define $H_\kappa:\vtheta\mapsto \frac{\kappa}{2}\|\vtheta-\vtheta^\circ\|_2^2$, and define $\mathrm{d}\gamma_{\kappa,\Theta}:\vtheta\mapsto \frac{\exp(-H_\kappa(\vtheta))}{\int_\Theta \exp(-H_\kappa(\vu))\,\mathrm{d}\tau(\vu)}\,\mathrm{d}\tau(\vtheta)$. Since $\nabla_\vtheta^2 H_\kappa(\vtheta)=\kappa I$, the function $H_\kappa$ is $\kappa$-strongly convex. Since $\Theta$ is convex, by the Bakry--{\'E}mery criterion \cite{bakryDiffusionsHypercontractives1985}, $\gamma_{\kappa,\Theta}$ satisfies a log-Sobolev inequality with parameter $\kappa$. By the definition of $\gamma_{\kappa,\Theta}$, $\mathrm{d}\tau=Z_\kappa \exp(H_\kappa)\,\mathrm{d}\gamma_{\kappa,\Theta}$, where $Z_\kappa$ is a normalizing constant. Let $\psi=-H_\kappa$. Then one can show that $\sup\psi-\inf\psi\leq \frac{\kappa}{2}\diam^2(\Theta)$. By the Holley--Stroock perturbation criterion \cite{holleyLogarithmicSobolevInequalities1987}, $\tau$ satisfies a log-Sobolev inequality with parameter at least $\kappa\exp(-\frac{\kappa}{2}\diam^2(\Theta))$. Taking $\kappa=\frac{2}{\diam^2(\Theta)}$, we obtain that $\tau$ satisfies a log-Sobolev inequality with parameter $\frac{2}{e\diam^2(\Theta)}$.

It is easy to prove that, for any $\vtheta_1,\vtheta_2$, the following inequality holds
\begin{align*}
\left| \frac{\delta P_N\ell_\nu}{\delta\nu}(\vtheta_1) - \frac{\delta P_N\ell_\nu}{\delta\nu}(\vtheta_2) \right| \leq 4\oL{L_bound}(2\oL{L_bound} + B_\xi).
\end{align*}
Therefore, for $\psi:\vtheta\mapsto \frac{1}{\lambda}\frac{\delta P_N\ell_\nu}{\delta\nu}(\vtheta)$, by applying the Holley--Stroock perturbation principle \cite{holleyLogarithmicSobolevInequalities1987}, we obtain that the probability measure $\frac{\exp(-\psi)}{\int \exp(-\psi)\,\mathrm{d}\tau}\,\mathrm{d}\tau$ satisfies a log-Sobolev inequality with parameter $$\tilde\alpha := \frac{2}{e\mathrm{diam}^2(\Theta)}\exp\left(-\frac{4\oL{L_bound}(2\oL{L_bound}+B_\xi)}{\lambda}\right).$$Furthermore, the probability measure $\frac{\exp(-\psi)}{\int \exp(-\psi)\,\mathrm{d}\tau}\,\mathrm{d}\tau=\pi_\nu^\lambda$. Therefore, it satisfies a log-Sobolev inequality with the same parameter. Similarly, $\pi_{\hat\nu_\lambda}^\lambda=\hat\nu_\lambda$ also satisfies a log-Sobolev inequality with the same parameter.

We aim to apply \cite[Proposition~1]{nitanda_convex_2022}
(see also \cite[Lemma~3.4]{chizat_mean-field_2022}), for which we need to verify its assumptions. Here, we only check Assumption 1, since the remaining conditions are straightforward. A direct computation yields $\nabla_{\vtheta}\frac{\delta P_N\ell_\nu}{\delta\nu}(\vtheta)=-\frac{2}{N}\sum_{i=1}^N(Y_i-f_\nu(X_i))\nabla_\vtheta \phi(X_i,\vtheta)$. For any $\nu_1,\nu_2\in\cP(\Theta)$ and any $\vtheta_1,\vtheta_2\in\Theta$, we have
$$\left\| \nabla_{\boldsymbol{\theta}} \frac{\delta P_N \ell_{\nu_1}}{\delta \nu} (\boldsymbol{\theta}_1) - \nabla_{\boldsymbol{\theta}} \frac{\delta P_N \ell_{\nu_2}}{\delta \nu} (\boldsymbol{\theta}_2) \right\|_2$$$$\le \frac{2}{N} \sum_{i=1}^N \left\| (f_{\nu_2}(X_i) - f_{\nu_1}(X_i)) \nabla_{\boldsymbol{\theta}} \phi(X_i, \boldsymbol{\theta}_1) + (Y_i - f_{\nu_2}(X_i)) \left( \nabla_{\boldsymbol{\theta}} \phi(X_i, \boldsymbol{\theta}_1) - \nabla_{\boldsymbol{\theta}} \phi(X_i, \boldsymbol{\theta}_2) \right) \right\|_2$$$$\le \frac{2}{N} \sum_{i=1}^N \left( |f_{\nu_2}(X_i) - f_{\nu_1}(X_i)| \|\nabla_{\boldsymbol{\theta}} \phi(X_i, \boldsymbol{\theta}_1)\|_2 + |Y_i - f_{\nu_2}(X_i)| \|\nabla_{\boldsymbol{\theta}} \phi(X_i, \boldsymbol{\theta}_1) - \nabla_{\boldsymbol{\theta}} \phi(X_i, \boldsymbol{\theta}_2)\|_2 \right)$$
Since $\vtheta\mapsto \phi(X_i,\vtheta)$ is Lipschitz and $\Theta$ is compact, we have $|f_{\nu_1}(X_i)-f_{\nu_2}(X_i)|=|\int_\Theta \phi(X_i,\vtheta)\,\mathrm{d}(\nu_1-\nu_2)(\vtheta)|\leq W_1(\nu_1,\nu_2)\sup_{\vtheta\in\Theta}\|\nabla_\vtheta\phi(X_i,\vtheta)\|_2$, where $W_1$ is the Wasserstein-1 metric. On a compact set, $W_1(\nu_1,\nu_2)\leq W_2(\nu_1,\nu_2)$, where $W_2$ is the Wasserstein-2 metric. Moreover, by the Lipschitzness of $\nabla_\vtheta\phi$, one can prove that $\|\nabla_\vtheta\phi(X_i,\vtheta_1)-\nabla_\vtheta\phi(X_i,\vtheta_2)\|_2\leq \|\vtheta_1-\vtheta_2\|_2\sup_{\vtheta\in\Theta}\|\nabla_\vtheta^2\phi(X_i,\vtheta)\|_{\op}$.

By \cite[Proposition~1]{nitanda_convex_2022}, for any $\nu\ll\tau$, one has
\begin{align}\label{eq:entropy_sandwich}
\lambda\KL(\nu\|\hat\nu_\lambda)\leq P_N\ell_\nu^\lambda-P_N\ell_{\hat\nu_\lambda}^\lambda\leq \lambda\KL(\nu\|\pi_\nu^\lambda).
\end{align}
From the continuity equation $\partial_t \rho_t + \nabla_{\boldsymbol{\theta}} \cdot (\rho_t v_t) = 0$ that governs the evolution of the measure, where the Wasserstein velocity field corresponding to the nonlinear Fokker-Planck equation is given by $v_t = - \nabla_{\boldsymbol{\theta}} \left( \frac{\delta P_N \ell_{\nu_t}}{\delta \nu} + \lambda \log \rho_t \right)$, see \cite{chizat_mean-field_2022}. According to the Wasserstein chain rule \cite[Lemma 2.2]{chizat_mean-field_2022}, the time derivative of the objective functional $P_N \ell_{\nu_t}^\lambda$ along the trajectory is
\begin{align*}
\frac{d}{dt} P_N \ell_{\nu_t}^\lambda = \int_\Theta \left\langle \nabla_{\boldsymbol{\theta}} \left( \frac{\delta P_N \ell_{\nu_t}}{\delta \nu} + \lambda \log \rho_t \right), v_t \right\rangle \rho_t \,\mathrm{d}\tau.
\end{align*}
Applying the divergence theorem over the compact parameter space $\Theta$, we can expand the above equation into the sum of a volume integral and a boundary integral:
\begin{align*}
\frac{d}{dt} P_N \ell_{\nu_t}^\lambda = - \int_\Theta \left( \frac{\delta P_N \ell_{\nu_t}}{\delta \nu} + \lambda \log \rho_t \right) \nabla_{\boldsymbol{\theta}} \cdot (\rho_t v_t) \,\mathrm{d}\tau + \int_{\partial \Theta} \left( \frac{\delta P_N \ell_{\nu_t}}{\delta \nu} + \lambda \log \rho_t \right) (\rho_t v_t \cdot \boldsymbol{n}) \,\mathrm{d}S.
\end{align*}
Because the system satisfies the no-flux reflecting boundary condition on $\partial \Theta$ due to \eqref{eq:nonlinear_Fokker_Planck}, the normal component of the probability flux vanishes everywhere, i.e., $(\rho_t v_t) \cdot \boldsymbol{n} = 0$. Therefore, the boundary integral strictly vanishes.
Substituting $\nabla_{\boldsymbol{\theta}} \cdot (\rho_t v_t) = -\partial_t \rho_t$ from the continuity equation into the remaining volume integral yields
\begin{align*}
\frac{d}{dt} P_N \ell_{\nu_t}^\lambda = \int_\Theta \left( \frac{\delta P_N \ell_{\nu_t}}{\delta \nu} + \lambda \log \rho_t \right) \partial_t \rho_t \,\mathrm{d}\tau.
\end{align*}
On the other hand, by the nonlinear Fokker--Planck equation together with the fact that $\nabla_\vtheta\log\rho_t = \frac{1}{\rho_t}\nabla_\vtheta\rho_t$, we know that
\begin{align*}
\partial_t\rho_t=\nabla_\vtheta\cdot\left[\rho_t\nabla_\vtheta\left(\frac{\delta P_N\ell_{\nu_t}}{\delta\nu}+\lambda\log\rho_t\right)\right].
\end{align*}
Substituting this into the previous identity yields
\begin{align*}
\frac{d}{dt}P_N\ell_{\nu_t}^\lambda=\int_\Theta\left(\frac{\delta P_N\ell_{\nu_t}}{\delta\nu}+\lambda\log\rho_t\right)\nabla_\vtheta\cdot\left[\rho_t\nabla_\vtheta\left(\frac{\delta P_N\ell_{\nu_t}}{\delta\nu}+\lambda\log \rho_t\right)\right]\,\mathrm{d}\tau.
\end{align*}
Let $u_t(\vtheta)=\frac{\delta P_N\ell_{\nu_t}}{\delta\nu}(\vtheta)+\lambda\log \rho_t(\vtheta)$. Then the right-hand side of the above equation can be written as $\int_\Theta u_t\nabla_\vtheta\cdot(\rho_t\nabla_\vtheta u_t)\,\mathrm{d}\tau$. By integration by parts, we obtain
\begin{align*}
\int_\Theta u_t\nabla_\vtheta\cdot(\rho_t\nabla_\vtheta u_t)\,\mathrm{d}\tau=-\int_\Theta \|\nabla_\vtheta u_t\|_2^2\rho_t\,\mathrm{d}\tau+\int_{\partial\Theta}u_t\rho_t\nabla_\vtheta u_t\cdot \vn\,\mathrm{d}S,
\end{align*}
and by \eqref{eq:nonlinear_Fokker_Planck}, the second term is $0$. Therefore, $\frac{d}{dt}P_N\ell_{\nu_t}^\lambda=-\int_\Theta\norm{\nabla_\vtheta\left(\frac{\delta P_N\ell_{\nu_t}}{\delta\nu}+\lambda\log\rho_t\right)}_2^2\,\mathrm{d}\nu_t$. We now compute $\nabla_\vtheta u_t$. By the definition of $\pi_{\nu_t}^\lambda$, we have $\log\frac{d\nu_t}{d\pi_{\nu_t}^\lambda}(\vtheta)=\log\rho_t(\vtheta)+\lambda^{-1}\frac{\delta P_N\ell_{\nu_t}}{\delta\nu}(\vtheta)+\log\int_\Theta\exp(-\frac{1}{\lambda}\frac{\delta P_N\ell_{\nu_t}}{\delta\nu}(\vtheta))\,\mathrm{d}\tau(\vtheta)$. Hence, taking the gradient with respect to $\vtheta$ on both sides gives $\nabla_\vtheta\left(\frac{\delta P_N\ell_{\nu_t}}{\delta\nu}+\lambda\log\rho_t\right)=\lambda\nabla_\vtheta\log\frac{d\nu_t}{d\pi_{\nu_t}^\lambda}$. In summary,
\begin{align*}
\frac{d}{dt}P_N\ell_{\nu_t}^\lambda=-\lambda^2 I(\nu_t\|\pi_{\nu_t}^\lambda).
\end{align*}
By the log-Sobolev inequality, we have $I(\nu_t\|\pi_{\nu_t}^\lambda)\geq \tilde\alpha\KL(\nu_t\|\pi_{\nu_t}^\lambda)$. By \eqref{eq:entropy_sandwich}, we have $\KL(\nu_t\|\pi_{\nu_t}^\lambda)\geq \lambda^{-1}(P_N\ell_{\nu_t}^\lambda-P_N\ell_{\hat\nu_\lambda}^\lambda)$. Combining the preceding two inequalities yields $\frac{d}{dt}P_N\ell_{\nu_t}^\lambda\leq-\tilde\alpha(P_N\ell_{\nu_t}^\lambda-P_N\ell_{\hat\nu_\lambda}^\lambda)$. Since $P_N\ell_{\hat\nu_\lambda}^\lambda$ does not depend on $t$, we have $\frac{d}{dt}P_N\ell_{\hat\nu_\lambda}^\lambda=0$, and hence
\begin{align*}
\frac{d}{dt}(P_N\ell_{\nu_t}^\lambda-P_N\ell_{\hat\nu_\lambda}^\lambda)\leq-\tilde\alpha(P_N\ell_{\nu_t}^\lambda-P_N\ell_{\hat\nu_\lambda}^\lambda).
\end{align*}
By Gronwall's inequality, we obtain $P_N\ell_{\nu_t}^\lambda-P_N\ell_{\hat\nu_\lambda}^\lambda\leq \exp(-\tilde\alpha t)(P_N\ell_{\nu_0}^\lambda-P_N\ell_{\hat\nu_\lambda}^\lambda)$. Applying \eqref{eq:entropy_sandwich} again, we obtain the convergence of the KL divergence.
\endproof

\begin{Proposition}\label{prop:convergence_spherical_MFLD}
Let $\tau$ be the uniform distribution on $S_2^{d-1}$. Let $\rho_0:\vw\mapsto 1$. Define the nonlinear Fokker-Planck equation by $\partial_t\rho_t=\nabla_S\cdot\left(\rho_t\nabla_S\frac{\delta P_N\ell_{\nu_t}}{\delta\nu}+\lambda\nabla_S\rho_t\right)$, where $\nabla_S$ is the Riemannian gradient on $S_2^{d-1}$ and $\nabla_S\cdot$ is the Riemannian divergence; see \cite[Chapter 3]{boumal_introduction_2023}. Let $\nu_t = \rho_t \,\mathrm{d}\tau$. Then, under the same conditions as in Proposition~\ref{prop:convergence_speed}, for any $t\geq 0$ and $\lambda> 0$, the following holds
\begin{align*}
&P_N\ell_{\nu_t}^\lambda - P_N\ell_{\hat\nu_\lambda}^\lambda \leq \exp\left(-Cd\exp(-\frac{C}{\lambda}) t\right)(P_N\ell_{\nu_0}^\lambda - P_N\ell_{\hat\nu_\lambda}^\lambda),\mbox{ and }\\
&\KL(\nu_t\|\hat\nu_\lambda)\leq \frac{1}{\lambda}\exp\left(-Cd\exp(-\frac{C}{\lambda}) t\right)(P_N\ell_{\nu_0}^\lambda - P_N\ell_{\hat\nu_\lambda}^\lambda).
\end{align*}
\end{Proposition}
\beginproof
The proof of Proposition~\ref{prop:convergence_spherical_MFLD} is almost identical to that of Proposition~\ref{prop:convergence_speed}. The only difference is that, in this case, the log-Sobolev inequality for $\tau$ can be obtained directly from the Beckner inequality (see \cite{beckner_sobolev_1992}). In fact, it satisfies the log-Sobolev inequality with parameter $d-1$.
For completeness, we provide the connection between the form of the log-Sobolev inequality in \cite{beckner_sobolev_1992} and the KL--Fisher information form of the log-Sobolev inequality used in this paper.
For any smooth function $f$ satisfying $\int f^2\,\mathrm{d}\mu=1$, assume that $\frac{\rho}{2}\int f^2\log f^2\,\mathrm{d}\mu\leq \int \|\nabla f\|_2^2\,\mathrm{d}\mu$. Take any $\eta\ll\mu$, and write $r=\frac{d\eta}{d\mu}$ and $f=\sqrt r$. Then $\int r\,\mathrm{d}\mu=1$, and $\int f^2\log f^2\,\mathrm{d}\mu=\int r\log r\,\mathrm{d}\mu=\KL(\eta\|\mu)$. On the other hand, $I(\eta\|\mu)=\int \|\nabla\log\frac{d\eta}{d\mu}\|_2^2\,\mathrm{d}\eta=\int \|\nabla\log r\|_2^2 r\,\mathrm{d}\mu$. Since $\nabla\log r=\nabla\log f^2=2\frac{\nabla f}{f}$, we have $\|\nabla\log r\|_2^2r=4\|\nabla f\|_2^2$. Therefore, this implies that $I(\eta\|\mu)=4\int \|\nabla f\|_2^2\,\mathrm{d}\mu\geq 2\rho\KL(\eta\|\mu)$. Since $\eta$ is arbitrary, this means that $\mu$ satisfies a log-Sobolev inequality with parameter $\rho$.
\endproof

We remark that the proof using the Holley--Stroock argument \cite{holleyLogarithmicSobolevInequalities1987} does not appear to be optimal, at least in the multi-index problem. Indeed, for $\hat\nu_\lambda$ with a multi-spike structure to be a perturbation of $\nu_0$ (the uniform distribution), the scale of this perturbation must be extremely large---this is precisely why the log-Sobolev constants in Proposition~\ref{prop:convergence_speed}, Proposition~\ref{prop:convergence_spherical_MFLD} depend exponentially on $\lambda$. However, since we already know that $\hat\nu_\lambda$ has a multi-spike structure, we conjecture that there exists a proof method not relying on the Holley--Stroock perturbation criterion, which could improve the factor $\exp(-C/\lambda)$ in Proposition~\ref{prop:convergence_speed} and in Proposition~\ref{prop:convergence_spherical_MFLD} to a polynomial order in $\lambda^{-1}$.

\subsection{A counter-example}\label{sec:counter_example}

This section gives a counterexample for misspecified single-index models. It shows that the feature-learning result does not automatically extend beyond the well-specified setting.

\begin{Proposition}\label{prop:counter_example}
Let $X \sim \cN(0, I_d)$ ($d \ge 2$) and $\Theta = [-A, A] \times W B_2^d$. Consider the model class $\mathcal{F}_{A,W} = \left\{ x \mapsto \int_\Theta a \sigma(\langle \boldsymbol{w}, x \rangle) \nu(\mathrm{d}a, \mathrm{d}\boldsymbol{w}) : \nu \in \mathcal{P}(\Theta) \right\}$. Assume $\sigma$ is globally Lipschitz, and denote $L_\sigma = \|\sigma'\|_\infty < \infty$. Then for any $s > 0$, there exist $h \in W_2^s(\gamma_1)$ and $\boldsymbol{w}^\star \in S^{d-1}$ such that $f^\star(x) = h(\langle \boldsymbol{w}^\star, x \rangle)$ satisfies $\inf_{\nu \in \mathcal{P}(\Theta)} \|f_\nu - f^\star\|_{L^2(P_X)}^2 > 0$, where $\gamma_1$ denotes the standard one-dimensional Gaussian measure $\cN(0, 1)$, and $W_2^s(\gamma_1)$ is the corresponding Gaussian Sobolev space of order $s$ consisting of functions whose weak derivatives up to order $s$ are square-integrable with respect to $\gamma_1$.
\end{Proposition}

\beginproof
Fix an arbitrary $\boldsymbol{w}^\star \in S^{d-1}$ and let $T = \langle \boldsymbol{w}^\star, X \rangle \sim \cN(0, 1)$. For any $f_\nu \in \mathcal{F}_{A,W}$, consider its conditional expectation given $T=t$, defined as $g_\nu(t) := \mathbb{E}[f_\nu(X) \mid T=t]$. By the orthogonal decomposition $X = t \boldsymbol{w}^\star + \boldsymbol{Z}_\perp$, where $\boldsymbol{Z}_\perp \sim \cN(0, I_d - \boldsymbol{w}^\star {\boldsymbol{w}^\star}^\top)$, we can write $g_\nu(t) = \int_\Theta a \mathbb{E}_{\boldsymbol{Z}_\perp} [\sigma(t \langle \boldsymbol{w}, \boldsymbol{w}^\star \rangle + \langle \boldsymbol{w}, \boldsymbol{Z}_\perp \rangle)] \nu(\mathrm{d}a, \mathrm{d}\boldsymbol{w})$. Differentiating with respect to $t$ under the integral sign yields $g_\nu'(t) = \int_\Theta a \langle \boldsymbol{w}, \boldsymbol{w}^\star \rangle \mathbb{E}_{\boldsymbol{Z}_\perp} [\sigma'(t \langle \boldsymbol{w}, \boldsymbol{w}^\star \rangle + \langle \boldsymbol{w}, \boldsymbol{Z}_\perp \rangle)] \nu(\mathrm{d}a, \mathrm{d}\boldsymbol{w})$. Since $\sigma$ is $L_\sigma$-Lipschitz, $|a| \le A$, and $\|\boldsymbol{w}\|_2 \le W$, the Cauchy-Schwarz inequality gives $|\langle \boldsymbol{w}, \boldsymbol{w}^\star \rangle| \le W$. Thus, $$|g_\nu'(t)| \le \int_\Theta |a| |\langle \boldsymbol{w}, \boldsymbol{w}^\star \rangle| L_\sigma \nu(\mathrm{d}a, \mathrm{d}\boldsymbol{w}) \le A W L_\sigma.$$ Denote $L := A W L_\sigma$. Hence, for all $\nu \in \mathcal{P}(\Theta)$, $g_\nu$ satisfies $\mathrm{Lip}(g_\nu) \le L$.

By Jensen's inequality applied to the conditional expectation, $\|f_\nu - f^\star\|_{L^2(P_X)}^2 = \mathbb{E} \left[ (f_\nu(X) - h(T))^2 \right] \ge \mathbb{E}_T \left[ (\mathbb{E}[f_\nu(X) \mid T] - h(T))^2 \right] = \|g_\nu - h\|_{L^2(\gamma_1)}^2$. Therefore, it suffices to construct a Sobolev function $h$ bounded away from all $L$-Lipschitz functions in $L^2(\gamma_1)$.

Let $h(t) = \eta \sin(\omega t)$, where $\eta > 0$ is fixed, and $\omega$ is chosen sufficiently large such that $\omega > \frac{16\pi L}{\eta}$. Since the sine function is infinitely differentiable with bounded derivatives, for any finite $s > 0$, $h \in C^\infty(\mathbb{R}) \subset W_2^s(\gamma_1)$. We claim that $\inf_{\mathrm{Lip}(g) \le L} \|g - h\|_{L^2(\gamma_1)}^2 > 0$.

Fix $R > 2$. Within the interval $[-R, R]$, consider pairs of disjoint intervals $I_k^+ = \left[\frac{2\pi k + \pi/3}{\omega}, \frac{2\pi k + 2\pi/3}{\omega}\right]$ and $I_k^- = \left[\frac{2\pi k + 4\pi/3}{\omega}, \frac{2\pi k + 5\pi/3}{\omega}\right]$, restricting only to those $k$ such that both intervals are completely contained in $[-R, R]$. On $I_k^+$, the phase satisfies $\omega t \in [2\pi k + \pi/3, 2\pi k + 2\pi/3]$, which implies $\sin(\omega t) \ge \sin(\pi/3) = \frac{\sqrt{3}}{2} > \frac{1}{2}$, and thus $h(t) \ge \frac{\eta}{2}$. Similarly, on $I_k^-$, the phase lies in $[2\pi k + 4\pi/3, 2\pi k + 5\pi/3]$, which implies $\sin(\omega t) \le -\frac{\sqrt{3}}{2} < -\frac{1}{2}$, and thus $h(t) \le -\frac{\eta}{2}$. Note that $|I_k^+| = |I_k^-| = \frac{\pi}{3\omega}$.

Suppose for contradiction that an $L$-Lipschitz function $g$ simultaneously satisfies $\int_{I_k^+} (g - h)^2 \mathrm{d}t < \frac{\eta^2}{64} |I_k^+|$ and $\int_{I_k^-} (g - h)^2 \mathrm{d}t < \frac{\eta^2}{64} |I_k^-|$. Then by the mean value theorem for definite integrals, there exist points $s_k \in I_k^+$ and $t_k \in I_k^-$ such that $|g(s_k) - h(s_k)| < \frac{\eta}{8}$ and $|g(t_k) - h(t_k)| < \frac{\eta}{8}$. Using the triangle inequality and the pointwise bounds on $h$, we obtain $g(s_k) > h(s_k) - \frac{\eta}{8} \ge \frac{\eta}{2} - \frac{\eta}{8} = \frac{3\eta}{8}$, and $g(t_k) < h(t_k) + \frac{\eta}{8} \le -\frac{\eta}{2} + \frac{\eta}{8} = -\frac{3\eta}{8}$. Consequently, $|g(s_k) - g(t_k)| \ge \frac{3\eta}{4}$. However, by definition, $|s_k - t_k| \le \frac{2\pi}{\omega}$. The $L$-Lipschitz continuity of $g$ then requires $\frac{3\eta}{4} \le |g(s_k) - g(t_k)| \le L |s_k - t_k| \le L \frac{2\pi}{\omega}$, which implies $\omega \le \frac{8\pi L}{3\eta}$. This strictly contradicts the choice of $\omega > \frac{16\pi L}{\eta}$.

Therefore, for each pair $(I_k^+, I_k^-)$, at least one interval $I \in \{I_k^+, I_k^-\}$ satisfies $\int_I (g - h)^2 \mathrm{d}t \ge \frac{\eta^2}{64} |I|$. Since the number of periods completely contained in $[-R, R]$ is at least $cR\omega$, for some absolute constant $c > 0$. Summing the integrals over these intervals yields $\int_{-R}^R (g(t) - h(t))^2 \mathrm{d}t \ge c_1 \eta^2 R$, where $c_1 > 0$ is an absolute constant. 

Finally, since the Gaussian pdf is uniformly bounded from below on the compact set $[-R, R]$ by $\inf_{|t| \le R} \frac{1}{\sqrt{2\pi}} e^{-t^2/2} = \frac{1}{\sqrt{2\pi}} e^{-R^2/2} := m_R > 0$, we obtain $\|g - h\|_{L^2(\gamma_1)}^2 \ge \int_{-R}^R (g(t) - h(t))^2 \gamma_1(\mathrm{d}t) \ge m_R c_1 \eta^2 R := c_{R,\eta} > 0$. This lower bound holds uniformly for any $g$ with $\mathrm{Lip}(g) \le L$. Together with the Jensen's inequality argument, this concludes $\inf_{\nu \in \mathcal{P}(\Theta)} \|f_\nu - f^\star\|_{L^2(P_X)}^2 \ge c_{R,\eta} > 0$.
\endproof

\section{Further Discussions}

\subsection{Comparison with norm-based bounds and compression bounds}\label{sec:comparision_with_other_rates}

It is useful to compare our results with two common ways of proving estimation error rates for neural networks. The first is norm-based. Norm-based rates measure the size of a network by a norm of its parameters, 
a margin-normalized norm, or a related complexity measure
~\cite{neyshaburNormBasedCapacity2015,bartlettSpectrallyNormalized2017,
neyshaburPACBayesSpectrally2018}. Some of these rates use localization in the sense that the ERM property 
restricts the analysis to a smaller region around the learned predictor
~\cite{bartlettLocalRademacher2005,koltchinskiiLocalRademacher2006,
koltchinskii_oracle_2011}. However, the structure used in the analysis is still fixed before training. The norm, the kernel, or the metric in which the local region is measured are not learned from the data.

A second approach is based on compression. Compression rates control generalization when the learned predictor can be replaced by a smaller object, for example by pruning, quantization, or a shorter code, while keeping the empirical error almost unchanged~\cite{aroraCompression2018,zhouNonVacuous2019,Suzuki2020Compression}. These rates can be strongly data-dependent, because the size of the compressed representation may depend on the trained network. Still, the main object of the analysis is the final size of the representation. Such rates \emph{assume} the compression structure and do not explain why the training dynamics finds a feature space in which the target function has low effective dimension.

Our theory takes a different viewpoint. We show that MFLD learns a problem-dependent feature structure through its hidden-layer distribution, which induces a data-dependent RKHS. In this learned local structure, the target function is aligned with a small number of the principal directions, and the latent estimator can use this alignment to obtain sharper, problem-dependent rates. Thus the low-dimensional structure is not imposed by a fixed norm, nor is it only measured after compression. \emph{It is learned by the training dynamics and then exploited by the estimator.} This distinction is not only conceptual; it is reflected in the prediction and parameter-recovery guarantees obtained in the single- and multi-index models studied  in Section~\ref{sec:feature_learning_Gaussian_single_multi_index}.

\subsection{Proof of Proposition~\ref{prop:stationary_distribution}}\label{sec:proof_statioanry}
\beginproof The first half of the corollary follows from Proposition~\ref{prop:multi-index-voronoi-concentration}, Jensen's inequality, and Theorem~\ref{thm:multi-index-fixed-output}.
By Markov's inequality, for any $\rho>0$,
\begin{align*}
    &\hat\varphi_\lambda(S_2^{d-1}\backslash S_\rho) = \hat\varphi_\lambda\left(\left\{ \forall 1\leq j\leq M:\, \|\vv-\vw_j^\star\|_2^2 > \rho^2 \right\}\right) = \int_{S_2^{d-1}}\1\left\{ \min_{1\leq j\leq M}\|\vv-\vw_j^\star\|_2^2 > \rho^2 \right\}\mathrm{d}\hat\varphi_\lambda(\vv)\\
    &\leq \int_{S_2^{d-1}}\frac{\min_{1\leq j\leq M}\|\vv-\vw_j^\star\|_2^2}{\rho^2}\mathrm{d}\hat\varphi_\lambda(\vv) = \frac{S_\star(\hat\varphi_\lambda)}{\rho^2},\mbox{ where }S_\star(\varphi) = \int_{S_2^{d-1}} \min_{1\le j\le M} \|\boldsymbol w-\boldsymbol w_j^\star\|_2^2 \,d\varphi(\boldsymbol w).
\end{align*}The upper bound of $S_\star(\hat\varphi_\lambda)$ comes from the combination of Theorem~\ref{thm:multi-index-fixed-output} (or Theorem~\ref{thm:rip-multi-index-fixed-output}) together with Proposition~\ref{prop:multi-index-localization} (or Lemma~\ref{lem:rip-low-degree-localization}).
\endproof

We include the following lemma, which shows that in high-temperature regime, $\hat\varphi_\lambda$ does not exhibit the concentration phenomena.
\begin{Lemma}\label{lemma:high_temperature_regime}
    Grant Assumption~\ref{ass:multi_index} (or Assumption~\ref{ass:rip_multi_index} respectively). There exists absolute constants $C_{\mathrm{cap}}$ and $C_Y = (2B_\sigma+B_\xi)^2$, such that for any $\rho>0$, and any Borel set $A_\rho$,
    \begin{align*}
        \hat\varphi_\lambda(S_2^{d-1}\backslash A_\rho)\geq 1-\tau(A_\rho) - \sqrt{\frac{C_Y}{2\lambda}}.
    \end{align*}In particular, when $\rho\to 0$, $(A_\rho)_\rho$ satisfies $\tau(A_\rho)\to0$, and $\lambda=\omega(1)$, $\hat\varphi_\lambda(S_2^{d-1}\backslash A_\rho)\geq 1- o(1)$ almost surely. In particular, $\hat\varphi_\lambda(S_2^{d-1}\backslash S_\rho)\geq 1-o(1)$ where $S_\rho = \cup_{j\leq M}\{\vw: \|\vw-\vw_j^\star\|_2\leq\rho\}$.
\end{Lemma}
\beginproof 
Notice that for any $\varphi$, $|Y_i-f_\varphi(X_i)|^2\leq C_Y$ almost surely, then $|P_N\ell_\varphi|\leq C_Y$ almost surely. Since $\Ent_\tau^-(\tau)=0$, by the definition of $\hat\varphi_\lambda$, $\lambda\Ent_\tau^-(\hat\varphi_\lambda)\leq P_N\ell_{\hat\varphi_\lambda}^\lambda \leq P_N\ell_{\tau}^\lambda = P_N\ell_\tau\leq C_Y$. On the other hand, $\KL(\hat\varphi_\lambda\|\tau)=\Ent_\tau^-(\hat\varphi_\lambda)\leq \frac{C_Y}{\lambda}$, which, by Pinsker's inequality (see, for instance, \cite[Lemma 2.5]{tsybakov_introduction_2009}), gives that $\|\hat\varphi_\lambda-\tau\|_{\mathrm{TV}}\leq\sqrt{\frac{C_Y}{2\lambda}}$, where $\|\cdot\|_{\mathrm{TV}}$ is the total variation norm. In particular, $\hat\varphi_\lambda(S_2^{d-1}\backslash A_\rho)\geq \tau(S_2^{d-1}\backslash A_\rho) - \sqrt{\frac{C_Y}{2\lambda}} = 1-\tau(A_\rho) - \sqrt{\frac{C_Y}{2\lambda}}$. By \cite[Proposition 5.1]{aubrunAliceBobMeet2017}, $\tau(S_\rho) \leq \sum_{j=1}^M\tau\{\vw: \|\vw-\vw_j^\star\|_2\leq\rho\}\leq \frac{1}{2}M\rho^{d-1}$, which completes the proof.
\endproof

\subsection{On the activation functions}\label{sec:comment_activation_functions}

Assumption~\ref{ass:fixed-output-general-ie} is mild and is satisfied by many standard smooth bounded nonconstant activations, such as sigmoid, tanh, Gaussian bump $t\mapsto \exp(-t^2/2)$, sine, and arctan. The multi-index assumption is stronger: for the chosen value of $M$, it requires the Hermite coefficients $b_1,\ldots,b_M$ of $\sigma$ to all be nonzero. This is a finite non-degeneracy condition and is not automatic from smoothness or boundedness. It is typically satisfied by generic shifts of smooth bounded activations, for example $\sigma(t+\mu)$ outside a finite or discrete exceptional set of shifts $\mu$. Unshifted symmetric activations should be treated with care: odd activations such as tanh, sine, and arctan have zero even Hermite coefficients, even activations such as the Gaussian bump have zero odd Hermite coefficients, and the unshifted sigmoid has zero positive even Hermite coefficients after centering. Hence these unshifted activations are not, in general, examples of the multi-index assumption unless the required coefficients are checked directly.

\subsection{Additional Classes of Problems Effectively Solved by MFLD}\label{sec:problems_solved_byMFLD}

Fix an activation function $\sigma$. For $K\ge1$, write $\beta_K(\sigma):=\min_{1\le m\le K}|b_m(\sigma)|$, and say that $\sigma$ is $K$-admissible if $\sigma\in C_b^3(\mathbb R)$ and $\beta_K(\sigma)>0$. For $0<\Delta\le2$, define
\[
    \mathcal A_{d,K,\Delta,\sigma} := \left\{ x\mapsto \sum_{j=1}^K\alpha_j\sigma(\langle w_j,x\rangle): \alpha_j>0,\ \sum_{j=1}^K\alpha_j=1,\  w_j\in S_2^{d-1},\ \min_{i\ne j}(1-\langle w_i,w_j\rangle)\ge\Delta \right\}, \]
with the separation condition understood as void when $K=1$.

By Theorem~\ref{thm:multi-index-fixed-output}, if $f^\circ\in\mathcal A_{d,K,\Delta,\sigma}$ and $\sigma$ is $K$-admissible, then, for $\lambda d$ small enough, spherical MFLD satisfies with probability at least $1-4e^{-x}$, 
\[ 
     \|f_{\hat\varphi_\lambda}-f^\circ\|_{L^2(\mathbb P_X)}^2 + \lambda\Ent_\tau^-(\hat\varphi_\lambda) \le R_{N,\lambda,x,\sigma}(K,\Delta) \lesssim_{K,\Delta,\sigma,B_\xi} \frac{Kd+d\log(dN)+x}{N} + \psi(\lambda d), 
\] 
where $R_{N,\lambda,x,\sigma}(K,\Delta) := C_{K,\Delta,\sigma,B_\xi}\Big\{ \tfrac{Kd+d\log(C_{0,K,\Delta,\sigma,B_\xi}dN)+x}{N} + \psi(\lambda d)\Big\}$.
Thus, for $L^2$-accuracy $\varepsilon$, MFLD is certified to learn the union of all classes $\mathcal A_{d,K,\Delta,\sigma}$ for which $R_{N,\lambda,x,\sigma}(K,\Delta)\le\varepsilon^2$.

This observation explains why some misspecified single-index links are still covered by our theory. Let $v^\star\in S_2^{d-1}$, set $w^\star=2v^\star$, and consider $f^\circ(x)=h(\langle w^\star,x\rangle)$ with 
\[ 
    h(t)=\frac12\sigma(t/2)+\frac12\sigma(-t/2), \qquad f^\circ(x) = \frac12\sigma(\langle v^\star,x\rangle) + \frac12\sigma(\langle -v^\star,x\rangle) = f_{\frac12\delta_{v^\star}+\frac12\delta_{-v^\star}}(x). 
\] 
Viewed as a single-index model with link $h$, this target is misspecified relative to the learner’s activation $\sigma$. However, as a mean-field model with activation $\sigma$, it is a well-specified two-index model with separation $\Delta=2$. Therefore, if $\sigma$ is $2$-admissible, MFLD learns this target at rate $O((d\log(dN)+x)/N)$ in the low-temperature regime, up to constants depending on $\sigma$ and the noise level.

\bibliographystyle{alpha}
\bibliography{biblio}
\end{document}